\documentclass{amsart}
\usepackage[foot]{amsaddr}

\usepackage[hang,flushmargin]{footmisc}

\usepackage[utf8]{inputenc}
\usepackage{amsmath, amsthm, amssymb}
\usepackage[usenames,dvipsnames,svgnames,table]{xcolor}
\usepackage[margin=1.3in]{geometry}
\usepackage{enumerate}
\usepackage{cite}
\usepackage{hyperref}
\usepackage{cleveref}
\usepackage{esint,dsfont}
\usepackage{xcolor}
\usepackage{mathabx}
\usepackage[normalem]{ulem}
\usepackage{verbatim}
\usepackage{tikz-cd}
\usepackage{xfrac}

\usepackage{tabularray}

\usepackage{etoolbox}
\cslet{blx@noerroretextools}\empty

\usepackage{autonum}

\newtheorem{theorem}{Theorem}[section]
\newtheorem{proposition}[theorem]{Proposition}
\newtheorem{lemma}[theorem]{Lemma}
\newtheorem{definition}[theorem]{Definition}

\newtheorem{corollary}[theorem]{Corollary}
\theoremstyle{definition}

\newcommand{\N}{\mathbb N}

\newcommand{\R}{\mathbb R}

\newcommand{\eps}{\varepsilon}

\newcommand{\ep}{\eps}

\newcommand{\dd}{\, \mathrm{d}}

\newcommand{\1}{\mathds{1}}

\newcommand{\pmeps}{{\rm pm}}
\newcommand{\hypnu}{{\rm hyp}}

\newcommand{\cK}{\mathcal{K}}

\newcommand{\cZ}{\mathcal{Z}}

\newcommand{\dbar}[1]{\bar{\bar{#1}}}

\newcommand{\be}{\begin{equation}}
\newcommand{\ee}{\end{equation}}

\newcommand{\dx}{\dd x}

\newcommand{\sssection}[1]{\subsubsection{#1}}
\newcommand{\ssection}[1]{\subsection{#1}}

\newsavebox{\accentbox}
\newcommand{\compositeaccents}[2]{\sbox\accentbox{$#2$}#1{\usebox\accentbox}}

\DeclareMathOperator{\supp}{supp}

\DeclareMathOperator{\sign}{sign}
\DeclareMathOperator{\Int}{Int}
\DeclareMathOperator{\argmin}{argmin}
\DeclareMathOperator{\argmax}{argmax}

\numberwithin{equation}{section}

\begin{document}

\title[Negative chemotaxis traveling waves]{Speed-up of traveling waves by negative chemotaxis}

\author{Quentin Griette}
\address[Quentin Griette]{Laboratoire de Mathématiques Appliquées du Havre, Université Le Havre Normandie}
\email{quentin.griette@univ-lehavre.fr}

\author{Christopher Henderson}
\address[Christopher Henderson]{Department of Mathematics, University of Arizona}
\email{ckhenderson@math.arizona.edu}

\author{Olga Turanova}
\address[Olga Turanova]{Department of Mathematics, Michigan State University}
\email{turanova@msu.edu}

\begin{abstract}
We consider the traveling wave speed for Fisher-KPP (FKPP) fronts under the influence of repulsive chemotaxis and provide an almost complete picture of its asymptotic dependence on parameters representing the strength and length-scale of chemotaxis.  Our study is based on establishing the convergence to the porous medium FKPP traveling wave and a hyperbolic FKPP-Keller-Segel traveling wave in certain asymptotic regimes.  In this way, it clarifies the relationship between three equations that have each garnered intense interest on their own.  Our proofs involve a variety of techniques ranging from entropy methods and decay of oscillations estimates to a general description of the qualitative behavior to the hyperbolic FKPP-Keller-Segel equation.  For this latter equation, we, as a part of our limiting arguments, establish a new explicit lower bound on the minimal traveling wave speed and provide a novel construction of traveling waves that extends the known existence range to all parameter values.
\end{abstract}

\maketitle

\section{Introduction}

In this paper, we are concerned with the problem of front propagation for the Fisher-KPP (FKPP) equation influenced by a Keller-Segel chemotaxis term:
\be\label{e.nontw_FKPPKS}
	\left\{\begin{split}
		U_t + \chi (V_x U)_x &= U_{xx} + U(1-U)
			\\
		- d V_{xx} &= U - V
	\end{split}\right.
	\qquad\text{ in } (0,\infty)\times \R,
\ee
with the condition $V(t,\cdot) \in L^\infty$ (to guarantee uniqueness for the second equation in~\eqref{e.nontw_FKPPKS}).  
Here $\chi \in \R$ and $d >0$ are the strength of the chemotaxis and its length-scale, respectively.  In this paper, we are interested in `negative chemotaxis,' which corresponds to $\chi < 0$.  This is the phenomenon in which individuals secrete a chemical (chemorepellent) that repels nearby individuals when sensed by them.  To rephrase this slightly, intraspecific interactions manifest as a `drift' that `spreads out' the population. 
The model~\eqref{e.nontw_FKPPKS} and others like it have been studied extensively in the past few decades, see~\cite{KellerSegel} for the original derivation of the Keller-Segel equation and~\cite{Perthame_transport, Murray} for overviews of its significance in mathematical biology.

In reaction-diffusion systems such as~\eqref{e.nontw_FKPPKS}, one can understand front propagation through the study of traveling wave solutions, which are special solutions of the form $V(t,x) = V(x-\bar ct)$ and $U(t,x) = U(x-\bar ct)$ for some $\bar c\in \R$, after a slight abuse of notation.  We are motivated by `invasion fronts,' so that we consider $U(-\infty) = 1$ and $U(+\infty)= 0$.  In this case,~\eqref{e.nontw_FKPPKS} becomes
\be\label{e.unscaled_tw}
	\left\{\begin{split}
		- \bar c U_x + \chi \left( V_x U\right)_x &= U_{xx} + U(1-U)
		\\ - d V_{xx} &= U - V
	\end{split}
	\right.
	\qquad\text{ in } \R,
\ee
with the additional condition that $V\in L^\infty$. 
The existence of these solutions as well as the positivity of the speed, $\bar c > 0$, has been proved using routine methods~\cite{Henderson2021}.

Our goal is to understand how the behavior of $\bar c = \bar c_{\chi,d}$ depends on $\chi$ and $d$.  We seek a precise description of how the nonlocal drift $V_x$ `speeds up' the traveling wave.  
(It is known that chemotaxis never `slows down' the traveling wave; that is, $c_{\chi,d} \geq 2$ for all $\chi, d$.  See \cite[Proposition~1.3]{Henderson2021}. But as we discuss later, it may happen that $c_{\chi, d}=2$ even when $-\chi>0$). 
As we discuss in the sequel, this connects to three ongoing threads of research: enhancement of propagation by advection, the role of nonlinearity in front propagation (i.e., `pushed' and `pulled' fronts), and the effect of nonlocality in front propagation.

When $\chi/d$ and $d$ are sufficiently small, it is known that the minimal speed $\bar c_{\chi,d} = 2$, while the minimal speed satisfies $\bar c_{\chi,d} \approx \sfrac{|\chi|}{2\sqrt d}$ when $1 \ll -\chi \ll d$~\cite{Henderson2021}.  Here we complete the picture, showing, roughly that there are positive constants $c^*_{{\rm pm}, \eps}$ and $c^*_{\hypnu, \bar\nu}$, such that
\be\label{e.c81001}
	\liminf_{\substack{-\chi \to 1/\eps,\\ d \ll -\chi}} \frac{\bar c_{\chi,d}}{\sqrt{-\chi}}
		\geq c^*_{{\rm pm}, \eps},
\ee
where $\eps \geq 0$, and 
\be\label{e.c81002}
	\liminf_{\substack{-\chi \to \infty,\\ -\chi/ d \to \bar \nu}} \frac{\bar c_{\chi,d}}{\sqrt{-\chi}}
		\geq c^*_{\hypnu, \bar\nu},
\ee
where $\bar \nu > 0$.   

The constants $c_{\rm pm, \ep}^*$ and $c^*_{\hypnu, \bar\nu}$ are the minimal wave speeds for, respectively, the porous medium FKPP model~\eqref{e.pm} and the hyperbolic FKPP-Keller-Segel model~\eqref{e.hyp} (see below).  The former is explicit (see~\eqref{e.c^*_pm}), taking the value $c_{\rm pm, 0}^* = \sfrac{1}{\sqrt 2}$ in the case $-\chi \to \infty$, that is, $\eps = 0$.  The latter is not explicit; however, we provide a positive lower bound for it.

We also obtain partially matching upper bounds.  For the limit~\eqref{e.c81001}, we construct a particular sequence of traveling waves for which the lower bound is attained in the limit (in the case $\eps = 0$).  For the limit~\eqref{e.c81002}, we construct a sequence of traveling waves that, after scaling, converge to a discontinuous traveling wave, which is presumed to be the minimal speed wave (see the discussion in~\cite{GrietteMagalZhao}).

Our arguments are based on the convergence of suitably rescaled solutions of~\eqref{e.unscaled_tw} to the hyperbolic FKPP and porous medium FKPP equations mentioned above.  This clarifies the relationship between the three equations and involves the development of novel estimates in settings where regularity does not arise from ellipticity.

The first step in our analysis is to perform a scaling that allows for the possibility of a nontrivial asymptotic limit. We define,
\be
	u(x)
		= U(x \sqrt{|\chi|}),
	\quad
	v(x)
		= V(x\sqrt{|\chi|})
	\quad
	c
		= \frac{\bar c}{\sqrt{|\chi|}}, \text{ and}\quad \nu = \frac{d}{-\chi}.
\ee
Then,~\eqref{e.unscaled_tw} yields: $u(-\infty) = 1$, $u(+\infty) = 0$, $v\in L^\infty$, and
\be\label{e.tw}
\tag{TW}
	\left\{
	\begin{split}
		- c u_x
			- (v_x u)_x
			&= \frac{1}{|\chi|} u_{xx} + u(1-u)\\
		- \nu v_{xx}
			&= u - v
	\end{split}
	\right.
	\qquad\text{ in } \R.
\ee
We use $c^*_{\chi,\nu}$  to denote the minimal speed of traveling wave solutions to~\eqref{e.tw} (see~\eqref{e.minimal_speed}). 
We now recast our goal with the new objects in hand.  The two limits~\eqref{e.c81001} and~\eqref{e.c81002} correspond, now, to understanding the scaling of the minimal speed $c^*_{\chi,\nu}$ when, respectively, $(-\frac{1}{\chi},\nu) \to (\eps,0)$ for $\ep\geq 0$ and $(-\frac{1}{\chi},\nu) \to (0,\bar\nu)$ for $\bar \nu>0$.  We refer to the former limit as the porous medium regime and the latter as the hyperbolic regime.  Notice that in each case, there is a loss of ellipticity in~\eqref{e.tw}, and the consequential degeneration of regularity estimates is a major source of difficulty in our analysis.  We now present some  heuristics that clarify this terminology and motivate our main results.

\subsection*{Heuristics and summary of main results} 
The first asymptotic limit we consider is $\nu\rightarrow 0$ and $-\chi \to \sfrac1\eps$, for $\ep\geq 0$. Note that the second equation in~\eqref{e.tw} suggests that $u$ and $v$ should approach the same limit in this regime. Using this ansatz along with the formal limits $c \to c_{\pmeps}$ and $u,v\to u_{\pmeps}$ suggests 
\be\label{e.pm}
\tag{PME}
	- c_{\pmeps} (u_{\pmeps})_x
	-\left(( u_{\pmeps})_x u_{\pmeps}\right)_x
		= \eps (u_{\pmeps})_{xx} + u_{\pmeps}(1- u_{\pmeps})
			\qquad\text{ in } \R.
\ee
We establish that this convergence does indeed occur. The main challenge is obtaining enough compactness to ensure convergence of the nonlinear terms in~\eqref{e.tw}: namely, the quadratic term $u^2$ and, especially, the term $(v_x u)_x$. 
This model, the porous medium FKPP equation, is well understood, see~\cite{Aronson1980,dePabloVazquez, KawasakiShigesadaIinuma} and references therein.  In particular, it is known that the minimal speed of~\eqref{e.pm} is strictly positive and an explicit expression for this quantity is available; see~\eqref{e.c^*_pm} in Subsection \ref{ss:halfline}. We use this, together with our convergence result, to deduce a lower bound on the limiting speed $c_{\pmeps}$, from which the estimate~\eqref{e.c81001} follows.

The second asymptotic limit we consider is $\nu\rightarrow \nu_{\hypnu}>0$ and $-\chi \to \infty$. In this case, the limiting equation is expected to be
\be\label{e.hyp}
\tag{HYP}
\left\{
	\begin{split}
		- c_{\hypnu} (u_{\hypnu})_{x}
			- \left((v_{\hypnu})_ x  u_{\hypnu}\right)_x
			&=  u_{\hypnu}(1- u_{\hypnu})\\
		- \nu_{\hypnu} (v_{\hypnu})_{xx}
			&= u_{\hypnu} - v_{\hypnu}
	\end{split}
	\qquad\text{ in } \R.
	\right.
\ee 
This equation has been introduced and studied in ~\cite{FuGrietteMagal-JMB,FuGrietteMagal-DCDSB, FuGrietteMagal-M3AS}. 
In addition,~\eqref{e.hyp} and similar models  are used in modeling tumor growth; see, e.g.,~\cite{PerthameVauchelet,FuGrietteMagal-JMB,KimTuranova,PerthameQuirosVazquez,KimPozar}. We prove that, in this regime, traveling wave solutions to~\eqref{e.tw} do indeed converge to those of~\eqref{e.hyp}.   Two key challenges are that  solutions to~\eqref{e.hyp} are irregular (in some cases, discontinuous), and that the theory of solutions to~\eqref{e.hyp}  is not as well-developed as that for~\eqref{e.pm}.  Thus, a large portion of our analysis is devoted to characterizing the general behavior of solutions to~\eqref{e.hyp} and establishing a lower bound on the minimal speed $c^*_{\hypnu, \nu_\hypnu}$. We are thus able to establish a lower bound on any limiting speed $c_{\hypnu}$, and from there deduce~\eqref{e.c81002}.

In addition, for each $\nu_{\hypnu}>0$, we  construct solutions of~\eqref{e.hyp}, called sharp traveling waves, and show that, as $\nu_{\hypnu}\rightarrow 0$, these special solutions converge to those of \eqref{e.pm} with $\ep=0$. We use this to deduce partially matching upper bounds on the speeds. 

We postpone the rigorous statements of our main results until the next section. For the convenience of the reader, we summarize them here:
\begin{description}
\item[\Cref{t.lower.pm}] Traveling wave solutions of~\eqref{e.tw} converge, as  $\nu\rightarrow 0$ and $\frac{1}{-\chi}\rightarrow \ep$, for $\ep\geq 0$, to those of~\eqref{e.pm}. Moreover, the limiting speed is bounded from below away from zero.
\item[\Cref{t.lower.hyp}] Traveling wave solutions of~\eqref{e.tw} converge, as $\frac{1}{-\chi}\rightarrow 0$ and $\nu\rightarrow \nu_{\hypnu}$, for  $\nu_{\hypnu}>0$, to those of~\eqref{e.hyp}. Moreover, the limiting speed is bounded from below away from zero.
\item[\Cref{t.upper.hyp} and \Cref{t.upper.pm}] There exist sharp traveling wave solutions to ~\eqref{e.hyp}, and they converge, as $\nu_{\hypnu}\rightarrow 0$, to those of \eqref{e.pm} with $\ep=0$. This yields an upper bound on the minimal speed of traveling wave solutions to~\eqref{e.tw} in the $\frac{1}{-\chi}\rightarrow 0$ and $\nu\rightarrow \nu_{\hypnu}$ limit.

 \end{description}

\section{Background and main results}\label{s.main results}

This section is devoted to rigorously stating our main results and describing their proofs and significance. In the first subsection, we state several preliminary facts and fix notation. Then, in Subsection \ref{ss.pme results}, we discuss our results on the porous medium regime; in particular, we state Theorem \ref{t.lower.pm}. Subsection \ref{ss:tw for hyp results} is devoted to our work on traveling wave solutions for~\eqref{e.hyp}. Then, in Subsection \ref{ss:hyp results}, we state Theorem \ref{t.lower.hyp}, which concerns the limit of solutions to~\eqref{e.tw} in the hyperbolic regime, and describe its proof. Our work on the matching lower bounds for the speeds, Theorems \ref{t.upper.hyp} and \ref{t.upper.pm}, is then discussed in Subsection \ref{ss:upper results}. Finally, Subsection \ref{ss:related work} is devoted to discussion of related work.

\ssection{Preliminaries}

Before precisely stating and discussing our results, it is useful to make some basic observations about traveling wave solutions $(u,v)$ to~\eqref{e.tw}.  First, 
we note that $v$ is given by the convolution of $u$ and a kernel:
\be\label{e.kernel}
	v(x)
		= (\cK_\nu * u)(x)
	\qquad
	\text{where}
	\quad
	\cK_\nu(x) = \frac{1}{\sqrt \nu} \cK\left( \frac{x}{\sqrt \nu}\right)
	\quad\text{and}\quad
	\cK(x) = \frac{1}{2} e^{-|x|}. 
\ee
Second, using~\eqref{e.kernel}, a simple comparison principle argument shows that
\be\label{e.0<u<1}
	0 < u, v < 1.
\ee
Finally, note that the first equation in~\eqref{e.tw} may also be written as
\be\label{e.tw2}
	- (c + v_x) u_x
		= \frac{1}{|\chi|} u_{xx} + u \left( \frac{\nu + v}{\nu} - \left(\frac{\nu + 1}{\nu}\right) u\right).
\ee
This is often useful in the sequel.  

\subsubsection{Discussion of minimal speeds}
\label{ss:halfline}
Here, we discuss and fix notation for minimal speeds of traveling waves for the various problems that we work with in this paper. 
We begin with,
\be
\label{e.minimal_speed}
	c^*_{\chi,\nu} = \inf \{c : \text{ there is a traveling wave solution } (c,u, v) \text{ to }\eqref{e.tw} \}.
\ee
In analogous models, there is an infinite half-line of speeds admitting traveling waves.  We believe that this property holds here: a traveling wave solution $(c, u, v)$ to~\eqref{e.tw}  exists for all $c \in [c_{\chi,\nu}^*,\infty)$. We do not prove this here, although we believe that the proof is straightforward.   Instead, we simply note that \cite[Theorem~1.2]{HamelHenderson} and~\eqref{e.kernel} imply that
\be
	c^*_{\chi,\nu} \geq \frac{2}{\sqrt{-\chi}},
\ee
while an easy compactness argument yields that the infimum in~\eqref{e.minimal_speed} is attained.  In this sense, the term minimal speed for $c^*_{\chi,\nu}$ is justified, although we caution the reader that this is often used in the context of having a half-line of speeds.

Next, we denote,
\be
c^*_{\pmeps,\ep} = \inf \{c : \text{ there is a traveling wave solution } (c_{\pmeps}, u_{\pmeps}) \text{ to }\eqref{e.pm} \}
\ee
This quantity, and the corresponding traveling wave solutions, are well-understood. For instance, in the case $\ep=0$, it is known \cite{Aronson1980} that no traveling wave solutions to~\eqref{e.pm} exist with speed $c_{\pmeps}\in [0,c^*_{\pmeps,0} )$; that the traveling wave with speed $c^*_{\pmeps,0} $ is sharp (which means $\{u_{\pmeps}>0\}=(\omega, +\infty)$ for some finite $\omega$); and that if $c_{\pmeps}>c^*_{\pmeps,0}$, then $u_{\pmeps}$ is positive on all of $\R$.  
Moreover, we have the following expression,
~\cite{Aronson1980,dePabloVazquez, KawasakiShigesadaIinuma}
\be\label{e.c^*_pm}
	c^*_{\pmeps,\ep}
		= \begin{cases}
			\frac{1}{\sqrt 2} + \sqrt 2 \eps
				\qquad &\text{ if } 1 > 2 \eps,\\
			2 \sqrt \eps
				\qquad &\text{ if } 1 \leq 2 \eps,
		\end{cases}
\ee
which may be found (upon carrying out the appropriate rescaling) in \cite[Eq (32)]{KawasakiShigesadaIinuma} and \cite[Section 2]{Aronson1980} for $\ep=0$.

Finally, we will denote, 
\be
	c^*_{\hypnu,\nu_{\hypnu}} = \inf \{c : \text{ there is a traveling wave solution } (c_{\hypnu}, u_{\hypnu}, v_{\hypnu}) \text{ to }\eqref{e.hyp}\}.
\ee
In our work, we consider solutions to~\eqref{e.hyp} in the sense of Definition \ref{def:hyp}. In ~\cite{FuGrietteMagal-M3AS, GrietteMagalZhao} specific traveling wave solutions to~\eqref{e.hyp} have been constructed and estimates on their speeds obtained. However, the behavior of  arbitrary traveling waves has not, to our knowledge, been studied before.  In particular, to our knowledge, no lower bounds exist in the literature for $c^*_{\hypnu,\nu}$.

\ssection{The porous medium regime}
\label{ss.pme results}

Our main result on this limit is:

\begin{theorem}\label{t.lower.pm}
Fix any $\eps \geq 0$.
\begin{enumerate}[(i)]
	\item \label{t.speed.pm}
		The minimal speeds have the asymptotics:
		\be
			\liminf_{-\chi \to \frac{1}{\eps}, \nu \to 0} c_{\chi,\nu}^* \geq c^*_{\pmeps,\ep}.
		\ee
	\item \label{t.convergence.pm}
		Consider any sequence $(\chi_n, \nu_n)$ and any corresponding traveling wave solutions $(c_n, u_n, v_n)$ of~\eqref{e.unscaled_tw}.  If $\limsup c_n < \infty$, then there exists $(c,u)$ solving~\eqref{e.pm} and a subsequence indexed by $n_k$ such that, under the normalization
		\be\label{e.normalization.pm}
			\min_{x \leq 0} u_n(x) = u_n(0)= \delta
		\ee
		for some fixed $\delta \in (0,1)$, the following convergence of $(c_{n_k}, u_{n_k}, v_{n_k})$ to $(c, u)$ holds:
		\be
			c_{n_k} \to c, 
			\quad u_{n_k} \to u \text{ in } L^\infty_{\rm loc},
			\quad\text{ and }\quad
			v_{n_k} \rightharpoonup u \text{ in }H^1_{\rm loc}.
		\ee
\end{enumerate}
\end{theorem}

We make three notes.  First, the normalization~\eqref{e.normalization.pm} is not a restriction: the system~\eqref{e.tw} is translation invariant and $u$ is a continuous function connecting $1$ and $0$, so~\eqref{e.normalization.pm} holds after a suitable translation.

Second, due to the Sobolev embedding theorem, the convergence of $v_{n_k}$ occurs in $C^{\sfrac{1}{2}}_{\rm loc}$.  Surprisingly, this is stronger than the notion of convergence of $u_{n_k}$. This is related to the main difficulties and the method of proof, see below.

Finally, we note that \Cref{t.lower.pm}.(i) follows directly from \Cref{t.lower.pm}.(ii).  Indeed, by simply taking a sequence $\chi_n$ and $\nu_n$ such that corresponding speed $c_n$ attains the limit inferior of $c^*_{\chi,\nu}$, the convergence in \Cref{t.lower.pm}.(ii) implies that the limiting speed $c$ is larger than $c^*_{\pmeps,\ep}$.  We describe this more precisely in the proof of \Cref{t.lower.pm} (see \Cref{s.lower.pm}).

We now describe the difficulties inherent in proving \Cref{t.lower.pm}.(ii). First, we note the possible issue of degeneracy of $v_n$ as $\nu_n \to 0$.  Indeed,  the second equation in~\eqref{e.tw} together with~\eqref{e.0<u<1} yield the immediate bound,
\be\label{e.c81003}
	\|(v_n)_x\|_{L^\infty} \leq O(\sfrac{1}{\sqrt\nu_n}).
\ee
Since the right-hand side approaches infinity as $\nu_n$ approaches zero, one  cannot rely simply on the relative compactness of $C^1$ in $L^\infty$.  Further, examining~\eqref{e.tw}, it is clear that, even in the $\eps>0$ case, one cannot rely on elliptic regularity theory to provide strong enough estimates on $u_n$ to take the limit in \Cref{t.lower.pm}, as {\em a priori} the $v_n$-coefficients may blow up like $O(\sfrac{1}{\nu_n})$.

A na\"ive first attempt to prove \Cref{t.lower.pm} might be to use the $L^\infty$-bounds in~\eqref{e.0<u<1} to pass to a weak-$*$ limit in $L^\infty$.  This, however, will fail due to the quadratic terms.  Even the $u^2$ term in~\eqref{e.tw} is problematic, as weak-* convergence in $L^\infty$ is not sufficient to guarantee that the weak-$*$ limit of $u_{n}^2$ is $u^2$.  More worrisome is the $\left((v_n)_xu_n\right)_x$ term.  
Hence, one requires greater regularity of $u_n$ or $v_n$, uniform in $n$, to pass to the limit.

Since we are performing a `vanishing viscosity' limit, a second possible approach might be to take inspiration from the robust theory of viscosity solutions~\cite{CrandallIshiiLions} that was developed to solve vanishing viscosity problems and attempt to perform a half-relaxed limit~\cite{BarlesPerthame}, which requires only $L^\infty$-regularity of the involved functions.  One's optimism for this approach grows when considering the analogy with numerical schemes, which were proven to converge via visocity solution methods~\cite{BarlesSouganidis1991}, due to the similarity of $v_x$ to a discrete derivative of $u$:
\be
	v_x(x) \approx \frac{u(x +\sqrt\nu) - u(x-\sqrt\nu)}{2\sqrt\nu},
\ee
at least in an averaged sense.  This approach, however, does not work, as the work in~\cite{BarlesSouganidis1991} requires the assumption that the scheme is `monotone.'  This, roughly, translates to~\eqref{e.tw} admitting a comparison principle, which it does not.  As such, this approach does not work.

To overcome these difficulties, we combine two major ingredients.  The first is an energy estimate, \Cref{lem:apriori bd nu to 0}:  by multiplying~\eqref{e.tw} by $\log u$, integrating by parts, and using that $\cK_\nu = \varphi_\nu*\varphi_\nu$ for some $\varphi$ (see~\eqref{e.c71901}), we find,
\be
	\int |\varphi_\nu * (u_n)_x|^2 \dd x
		\leq c.
\ee
Since $\varphi_\nu$ is `nice enough,' this estimate leads to a uniform $H^1$-bound and H\"older continuity of $v_n$.  Hence, we can take a weak limit of $v_n$ in $H^1$, which is sufficient regularity to pass to a limit with the $\left((v_n)_x u_n\right)_x$ term in~\eqref{e.tw}.  Unfortunately, it is does not immediately help with the $u_n^2$ term in~\eqref{e.tw}.

The next major ingredient, \Cref{lem:u-v}, is a decay of oscillations estimate, which states, roughly,
\be\label{e.oscillations}
	\max_{[x - \nu_n^{\sfrac{1}{4}}, x + \nu_n^{\sfrac{1}{4}}]} u_n
		- \min_{[x - \nu_n^{\sfrac{1}{4}}, x + \nu_n^{\sfrac{1}{4}}]} u_n
		\leq O(\nu_n^{\sfrac{1}{8}}).
\ee
Note that this is not enough to provide a uniform bound in any H\"older space, but it is enough, along with the regularity of $v_n$ to imply that
\be
	\|u_n-v_n\|_{L^\infty}
		\leq O(\nu_n^{\sfrac{1}{8}}).
\ee
This, along with the convergence of $v_n$, is enough to understand the convergence of $u_n^2$.  
The key \Cref{lem:u-v} is established by using the regularity of $v_n$ and a partial monotonicity result (\Cref{l.no_local_min}).

\ssection{Traveling wave solutions of~\eqref{e.hyp}}
\label{ss:tw for hyp results}

As described in the introduction, a key part of our work is understanding solutions  of~\eqref{e.hyp} and bounds on the traveling wave speeds. This subsection is devoted to describing our main results on this. 

\sssection{Definition and general properties of hyperbolic traveling waves}
 The first equation in~\eqref{e.hyp} is degenerate, so, before we can proceed, the notion of solution must be clarified.  To motivate the definition, notice that the first equation in~\eqref{e.hyp} may also be written (cf.~\eqref{e.tw2}):
\be\label{e.hyp2}
	- (c + (v_\hypnu)_x)  (u_{\hypnu})_x
		= u_{\hypnu} \left( \frac{\nu_{\hypnu} + v_{\hypnu}}{\nu_{\hypnu}} - \left(\frac{\nu_{\hypnu} + 1}{\nu_{\hypnu}}\right) u_{\hypnu}\right).
\ee
This is often useful in the sequel.  In particular, we see that $u_{\hypnu}$ is satisfies a `nice' ordinary differential equation away from any zeros of $c + (v_{\hypnu})_x$.  Hence, we can restrict to classical solutions away from these singular points. This leads to the following:
\begin{definition}
\label{def:hyp}
	Let $c_{\hypnu}\in [0,+\infty)$, $u_{\hypnu}\in L^\infty$, and $v_{\hypnu}\in W^{2,\infty}$, and denote
	\be
		\cZ := \{x : c_{\hypnu} + (v_{\hypnu})_x(x) = 0\}.
	\ee
	We say $(c_{\hypnu},u_{\hypnu},v_{\hypnu})$ is a solution to~\eqref{e.hyp} if the second equation in~\eqref{e.hyp} is satisfied almost everywhere and
	$u_{\hypnu}$ is nonnegative, bounded, in $C^1_{\rm loc}(\cZ^c)$,  satisfies~\eqref{e.hyp} pointwise on $\cZ^c$, and satisfies
	\be\label{e.c81501}
		u_{\hypnu} \left( \frac{\nu_{\hypnu} + v_{\hypnu}}{\nu_{\hypnu}} - \left(\frac{\nu_{\hypnu} + 1}{\nu_{\hypnu}}\right) u_{\hypnu}\right)
			= 0
			\qquad\text{ in } \Int(\cZ).
	\ee
\end{definition}
This definition  follows along the lines of previous works~\cite{FuGrietteMagal-JMB,FuGrietteMagal-DCDSB, FuGrietteMagal-M3AS}. 

Amazingly, despite being a fairly weak notion of solution, \Cref{def:hyp} is strong enough to prove that, for any traveling wave, $\cZ$ must either be empty or a single point and, if there is a point in $\cZ$, then there is a jump discontinuity at that point:

\begin{proposition}[Hyperbolic traveling waves]\label{prop:cardZleq1}
	Suppose that $(c,u,v)$ is a solution of~\eqref{e.hyp} in the sense  of Definition \ref{def:hyp}. Consider that $u$ is nonconstant; that is, assume that both $\{u < 1\}$ and $\{u > 0\}$ have positive measure.  Then $c>0$ and
	there are two possibilities:
	\begin{enumerate}
		\item $\mathcal Z=\varnothing$. In that case, $u\in C^\infty_{\rm loc}(\mathbb R)$ is a classical solution to \eqref{e.hyp}.
		\item  $\mathcal Z$ is consists of a single point: $\mathcal{Z}=\{x_0\} $. In that case, $u$ has a single jump discontinuity at $x_0$, with $\{u>0\} = (-\infty, x_0)$. 
		Moreover, $u \in C^\infty_{\rm loc}(\R\setminus\{x_0\})$ and $u$ satisfies, at the jump,
		\be
			u(x_0^-) = \frac{\nu_{\hypnu} + v(x_0)}{\nu_{\hypnu} +1}.
		\ee
	\end{enumerate}
\end{proposition}

Although this full claim is difficult to prove, the reason that $\cZ$ is nowhere dense is quite easy to see.  Consider a traveling wave solution $(c_{\hypnu}, u_{\hypnu}, v_{\hypnu})$ to~\eqref{e.hyp} with positive speed $c_{\hypnu}>0$.  Were $\Int(\cZ)$ to be nontrivial, then
\be\label{e.c81502}
	0 = u_{\hypnu}\left( \frac{\nu_{\hypnu}+v_{\hypnu}}{\nu_{\hypnu}} - \frac{\nu_{\hypnu} + 1}{\nu_{\hypnu}} u_{\hypnu}\right)
		\quad\text{ in } \Int(\cZ).
\ee
On the other hand, the definition of $\cZ$ implies that $(v_{\hypnu})_{xx} \equiv 0$ on $\Int(\cZ)$ so that, by~\eqref{e.hyp},
\be\label{e.c92901}
	u_{\hypnu} = v_{\hypnu}
	\quad\text{ in }
	\Int(\cZ).
\ee
These two equalities (that is,~\eqref{e.c81502} and~\eqref{e.c92901}) hold simultaneously only if $v \equiv 0$ or $v \equiv 1$.  This, however, is not compatible with the definition of $\cZ$, which implies that $(v_{\hypnu})_x = -c_{\hypnu} < 0$.  This contradiction implies that $\cZ$ must be nowhere dense.

The proof of 
\Cref{prop:cardZleq1}, which is given in Section \ref{s.technical}, is based on a careful analysis of~\eqref{e.hyp}. There are three major steps to this: (1) we use the observation from~\cite[Proof of Lemma 5.4]{FuGrietteMagal-M3AS} that $u$ satisfies a formula that is explicit in $v$ on $\cZ^c$ (\Cref{lem:explicit-solution}); (2) we show a strong maximum principle type argument: $u$ cannot tend to $0$ at some point in $\cZ^c$ without being uniformly zero on the entire maximal interval in $\cZ^c$; (3) $u$ cannot tend to positive limits on the endpoints of a maximal interval of $\cZ$ due to a technical argument using the convexity of $v$ (coming from the second equation in~\eqref{e.hyp} and the explicit form of the limit coming from (1)) near the endpoints of the interval.

Then (3) shows that any maximal interval of $\cZ^c$ must be half-infinite.  This implies that $\cZ$ is either empty, a single point, or a closed interval.  On the other hand, we have already described above why $\cZ$ has an empty interior.  It follows that $\cZ$ is either empty or a single point.

\subsubsection{Exponential decay}

We establish another key property that holds for solutions of~\eqref{e.hyp}, \eqref{e.tw}, as well as a certain `slab problem' for~\eqref{e.tw}: namely,  once $u$ reaches a critical level $\nu/(\nu+1)$, it decays exponentially. This is stated precisely in \Cref{lem:exp decay}.  

The proof of \Cref{lem:exp decay} is quite intricate, but essentially boils down to the fact that if $u$ were approximately constant on a large interval, then, due to~\eqref{e.kernel}, $v\approx u$ and $v_x \approx 0$ on that interval.  These two approximate equalities, when combined with~\eqref{e.hyp2}, imply that $u\approx 1$ or $u\approx 0$, which is not consistent with the fact that $u$ is approximately constant and $0 < u < 1$.  Hence, $u$ must `drop' by a proportion over every fixed large interval.  The main complications in proving this are (1) suitably quantifying the above heuristics, and (2) dealing with the $u_{xx}$ term in the case~\eqref{e.tw}.

The exponential decay of \Cref{lem:exp decay} is used in various places in the paper, such as \Cref{t.upper.hyp} (see below).  As such, it is important that it applies uniformly for~\eqref{e.hyp} and~\eqref{e.tw}.  Hence, complication (2) is unavoidable in our work. It is also the main source of difficulty.

\sssection{Bounds on the speed of solutions of~\eqref{e.hyp}}\label{s.hyp_speed}

In order to show that \Cref{t.lower.hyp} does not yield a trivial bound, we show that $c_{\hypnu, \nu_{\hypnu}}^* > 0$:  
\begin{proposition}\label{p.speed.hyp}
	Fix $\nu_M > \nu_m>0$.
	Then there exists  $\underline c>0$, depending only on $\nu_M$ and $\nu_m$, such that, if $\nu_{\hypnu} \in (\nu_m,\nu_M)$, then,
	\be\label{e.c81209}
		c^*_{{\hypnu}, \nu_{\hypnu}} \geq \underline c.
	\ee
	If $(c_{\nu_{\hypnu}},u_{\nu_{\hypnu}},v_{\nu_{\hypnu}})$ is a traveling wave solution of~\eqref{e.hyp} and the singular set $\cZ$ of $v_{\nu_{\hypnu}}$ consists of a single point, then we have the refined estimate:
	\be\label{e.c81208}
		c_{\nu_{\hypnu}} \in \left( \frac{\sqrt \nu_\hypnu}{2\nu_\hypnu + 1}, \frac{1}{2\sqrt\nu_\hypnu}\right).
	\ee
\end{proposition}

We note that the dependence on $\nu_m$ can probably be removed by a limiting argument similar to the one contained in the proof of \Cref{t.upper.pm}, below (see, e.g., \Cref{lem:c leq 2} and \Cref{lem:c nonzero}).  

The bounds in~\eqref{e.c81208} match those proved in~\cite[Theorem~2.7]{FuGrietteMagal-M3AS} for the particular traveling wave constructed there.  The  novelty in~\eqref{e.c81208} is that it holds for {\em any} traveling wave with a discontinuity, while the proof in~\cite{FuGrietteMagal-M3AS} relies on the monotonicity of the constructed wave.  Our proof uses essentially the same observations as~\cite{FuGrietteMagal-M3AS} with an additional partial monotonicity result (\Cref{lem:monotonicity.hyp}).

The main contribution of \Cref{p.speed.hyp} is the generality of the bound~\eqref{e.c81209}.  It has been observed in~\cite{FuGrietteMagal-M3AS, GrietteMagalZhao} that if $(c,u,v)$ is a solution of~\eqref{e.hyp} then
\be\label{e.c81210}
	c \geq \sup(-v_x)
\ee
holds, and quite precise bounds have been proved for particular waves that have been constructed.  Unfortunately, there is no general lower bound on the speed, and {\em a priori}~\eqref{e.c81210} leaves open the possibility of arbitrarily slow traveling waves if $v$ is very `flat.'   In fact, by analogy with the FKPP equation, one might expect that the `flatter' the traveling wave, the {\em faster} the speed, in contrast to~\eqref{e.c81210}.

To establish \Cref{p.speed.hyp}, we use the exponential decay estimate on $u$ (\Cref{lem:exp decay}) to obtain a universal lower bound on $-v_x$.

\ssection{The hyperbolic regime}
\label{ss:hyp results}
With the results described in the previous subsection in hand, we  establish: 
\begin{theorem}\label{t.lower.hyp}
Let $\nu_{\hypnu} > 0$.
\begin{enumerate}[(i)]
	\item \label{t.speed.hyp}
		The minimal speeds have the asymptotics: 
			\be
				\liminf_{-\chi \to \infty, \nu \to \nu_{\hypnu}} c_{\chi,\nu}^*
					\geq  c^*_{\hypnu,\nu_{\hypnu}}.
			\ee
	\item \label{t.convergence.hyp}
		Consider any sequence $(\chi_n, \nu_n)$ such that $-\chi_n \to \infty$ and $\nu_n \to \nu_{\hypnu}$, and let $(c_n,u_n,v_n)$ be any corresponding traveling wave solutions to~\eqref{e.tw}.  If $\sup  c_n < \infty$, then there exists $(c,u,v)$ solving~\eqref{e.hyp}  and a subsequence indexed by $n_k$ such that, under the normalization
		\be\label{e.normalization.hyp}
			u_n(0) = \delta,
		\ee
		for any $\delta \in (0, \sfrac{\nu_{\hypnu}}{\nu_{\hypnu}+1})$, the following convergence of $(c_{n_k}, u_{n_k}, v_{n_k})$ to $(c,u,v)$ holds: there is a set $\cZ$ that is either empty or contains a single point such that 
		\begin{itemize}
		\item $u_{n_k}$ and $v_{n_k}$ converge to $u$ and $v$, respectively, locally uniformly in $C^k(\cZ^c)$, for any $k$, 
		\item $v_{n_k}$ converges to $v$ weak-$*$ in $W^{2,\infty}$, and
		\item $c_{n_k}$ converges to $c$.
		\end{itemize}
\end{enumerate}
\end{theorem}

We now discuss the proof of \Cref{t.lower.hyp}.  As above, \Cref{t.lower.hyp}.(i) reduces to the case \Cref{t.lower.hyp}.(ii), so we only discuss the latter.  The main difficulty here is clear:~\eqref{e.tw} loses ellipticity and the limiting equation~\eqref{e.hyp} has no ellipticity.

Thus, in order to take a limit, one proceeds in the following way.  First, since
\be
	\limsup_{n\to\infty} \|v_n\|_{W^{2,\infty}} \leq \frac{1}{\nu_{\hypnu}}
	\quad\text{ and }\quad
	\|u_n\|_{L^\infty}\leq 1
\ee
one can take weak-$*$ limits to obtain $v$ and $u$ that are related by $- \nu v_{xx} = u-v$ weakly.  The limit above allows us to define $\cZ=\{x\ :\ c + v_x(x)=0\}$. However, we need to establish stronger convergence of $u_n$ and $v_n$ to eliminate the possibility that this `singular set' $\cZ$ might be quite large and complicated.

An additional problem is that, at this point, our notion of convergence of the $u_n$ is not strong enough to deduce~\eqref{e.hyp} as a limiting equation in $\cZ^c$.  However, an indication that it is possible to deduce better regularity of the $u_n$ in $\cZ^c$ is that~\eqref{e.tw2} yields
\be
	(u_n)_x
		= \frac{1}{-(c_n + (v_n)_x)} \left( \frac{1}{|\chi_n|} (u_n)_{xx} + u_n \left( \frac{\nu_n + v_n}{\nu_n} - \frac{\nu_n + 1}{\nu_n} u_n\right)\right),
\ee
which, imagining that $|\chi_n| = \infty$ momentarily, yields a $W^{1,\infty}$ bound for $u_n$ on compact subsets of $\cZ^c$ when $n$ is sufficiently large.  By `perturbing' off of this observation, we are able to show that, subsequentially, $u_n$ converges to $u$ in $C^k(\cZ^c)$ for any $k$, see \Cref{l.hyp_n_bounds}.

At this point in the proof,  $\cZ$ and the behavior of $u$ on $\cZ$ is not understood.  An argument based on the partial monotonicity of $u$ (\Cref{l.no_local_min}) that is reminiscent of the decay of oscillations argument discussed above allows us to show that~\eqref{e.c81501} is satisfied.  This final ingredient allows us to conclude that the liming object solves~\eqref{e.hyp} in the sense of \Cref{def:hyp}.

\subsection{Upper bounds on the asymptotics of the minimal speed}

\label{ss:upper results}

It is natural to wonder how sharp the bounds on the minimal speeds $c^*_{\pmeps,\eps}$ and $c^*_{\hypnu,\nu_\hypnu}$ in \Cref{t.lower.pm} and \Cref{t.lower.hyp} are.  In contrast to the arguments of these theorems, one would like to take a limit of a sequence of {\em minimal speed} traveling waves. Unfortunately, there is no known characterization of the minimal speed waves of~\eqref{e.tw}.

We must, instead, approach the problem by constructing a sequence of traveling waves directly using what we know about the minimal speed waves for~\eqref{e.pm} and~\eqref{e.hyp}.  As described in Subsection \ref{ss:halfline}, it is known  that the minimal speed traveling wave for~\eqref{e.pm} when $\eps = 0$ is the one that is $0$ after some $x_0$: $u(x) = 0$ for $ x > x_0$.  For~\eqref{e.hyp}, it is believed that the minimal speed traveling wave is a discontinuous one, corresponding to case (ii) in \Cref{prop:cardZleq1} (see the discussion after Theorem 1.4 in \cite{GrietteMagalZhao}).    
Hence, in each case, we construct a sequence of traveling wave solutions to~\eqref{e.tw} that approximate these waves.

\subsubsection{The hyperbolic case}  
Our first construction is related to the scaling associated to the hyperbolic model.  It is proved in \Cref{s.upper.hyp}.

\begin{theorem}\label{t.upper.hyp} 
Fix any $\nu_{\hypnu}>0$ and sequences $\chi_n\rightarrow -\infty$ and $\nu_n\rightarrow \nu_{\hypnu}$. There exists corresponding solutions $(c_n, u_n, v_n)$ to~\eqref{e.tw}  and a solution   $(c_{\nu_\hypnu}, u_{\nu_\hypnu}, v_{\nu_\hypnu})$ to~\eqref{e.hyp},  such that:
\begin{enumerate}[(i)]
	\item $u_{\nu_{\hypnu}}$, $v_{\nu_{\hypnu}}$, $u_{n}$, and $v_{n}$ are nonincreasing in $x$ for all $n$;
	\item we have,
	\be
\label{e.092926}
(v_{\nu_\hypnu})_x(0)+c_{\nu_\hypnu}=0;
\ee
	\item along  a subsequence $\chi_{n_k}\rightarrow -\infty$ and $\nu_{n_k}\rightarrow \nu_{\hypnu}$, the following limits hold uniformly locally in $\R$, $C^1(\R\setminus \{0\})$, and weak-$*$ in $W^{2,\infty}$, respectively:
		\be
			\lim_{n_k\rightarrow \infty} (c_{n_k}, u_{n_k}, v_{n_k})
				= (c_{\nu_\hypnu}, u_{\nu_\hypnu}, v_{\nu_\hypnu}).
		\ee
\end{enumerate}
\end{theorem}

The traveling wave solution of~\eqref{e.hyp} constructed in \Cref{t.upper.hyp}  is discontinuous: this follows by combining item~\eqref{e.092926} with case (ii) in \Cref{prop:cardZleq1} (see also \cite[Theorem~2.7]{FuGrietteMagal-M3AS}).  Interestingly, this provides a completely different construction of traveling wave solutions to~\eqref{e.hyp} from that of~\cite{FuGrietteMagal-M3AS}.  Further, this construction works for all values of $\nu_{\hypnu}$, whereas that of~\cite{FuGrietteMagal-M3AS} involves a restriction to $\nu_{\hypnu}$ smaller than some threshold (see \cite[Assumption~2.3]{FuGrietteMagal-M3AS}).

The important consequence of the discontinuity of $u_{\nu_\hypnu}$ is the partial converse to \Cref{t.lower.hyp} that comes from
combining \Cref{t.upper.hyp} with \Cref{p.speed.hyp}:
\begin{corollary}
For $\nu_{\hypnu}>0$, the minimal speeds have the asymptotics:
\be
	\limsup_{-\chi \to \infty, \nu \to \nu_{\hypnu}} c_{\chi,\nu}^*
    		\leq c_{\nu_{\hypnu}}
		\leq \frac{1}{2\sqrt{\nu_{\hypnu}}},
\ee
where $c_{\nu_\hypnu}$ is the speed associated to the family defined in \Cref{t.upper.hyp}.
\end{corollary}

The construction of the wave in \Cref{t.upper.hyp} follows the standard procedure: use  
the Leray-Schauder index to construct a solution to the `slab problem,' i.e., an appropriate approximate solution to~\eqref{e.tw} on a `slab' $[-L,L]$, and then take $L\to\infty$.  
The main difficulty is in ensuring that the construction yields a discontinuous traveling wave solution of~\eqref{e.hyp}: in other words, ensuring that ~\eqref{e.092926} holds. To this end, taking advantage of the structure of~\eqref{e.tw}, we consider 
\be
	\tilde u(x) := u_{n}(x) e^{|\chi|\frac{c_{n} x + v_{n}(x)}{2}}, 
\ee
and notice that any maximum of  $\tilde{u}$ occurs where $|(v_{n})_x + c| \leq O(1/|\chi_n|)$ (see~\eqref{e.67517}).

The goal is, thus, to guarantee the existence of a maximum and show that it remains `near' the origin, regardless of $\chi_n$.  The existence of the maximum is guaranteed by working with the `slab problem,' whose boundary conditions guarantee an interior maximum (see~\eqref{e.c81301}).  From the form of $\tilde u$ and that $v_{n}$ is bounded, it follows that a maximum cannot occur too far to the left.  On the other hand, the aforementioned exponential decay (\Cref{lem:exp decay}) guarantees that $(v_{n})_x$ decays exponentially, so that $|(v_n)_x + c|\leq O(1/|\chi|)$ cannot hold too far to the right.

We had previously mentioned that it is crucial for the exponential decay estimate to apply uniformly to~\eqref{e.hyp} and~\eqref{e.tw}.  This last step is one of the reasons for this: we must be able to apply it to solutions of~\eqref{e.tw} on the `slab' uniformly as $\chi\to-\infty$.

\subsubsection{The porous medium case}

Taking any decreasing, discontinuous traveling wave solutions to~\eqref{e.hyp}, we prove that, as $\nu_\hypnu\searrow0$, they converge to the minimal speed traveling wave of~\eqref{e.pm}.  This is contained in \Cref{s.upper.pm}.
$
$
\begin{theorem}\label{t.upper.pm}
Consider the family of traveling wave solutions $(c_{\nu_\hypnu}, u_{\nu_\hypnu},v_{\nu_\hypnu})$ to~\eqref{e.hyp} constructed in \Cref{t.upper.hyp}. Then
	\be
			\lim_{\nu_{\hypnu} \to 0} ( c_{\nu_\hypnu}, u_{\nu_\hypnu})
				= \left(\frac{1}{\sqrt 2}, u_{\pmeps}\right),
		\ee
		where $u_{\pmeps}$ is the unique minimal speed traveling wave solution to~\eqref{e.pm} with $\eps = 0$ and $\{u_{\pmeps} > 0\} = (-\infty,0)$.
\end{theorem}

The proof of \Cref{t.upper.pm} relies on the ideas and estimates developed for \Cref{t.lower.pm}.  The main issue is to ensure that the limiting object has support to the left of the origin; that is, $u_{\nu_{\hypnu}}\not\to 0$ on {\em all} of $\R$.  To do this, we show first obtain a preliminary lower bound on $c$ using `bulk-burning rate' style arguments (see~\cite{CKOR}) and then leverage the fact that $(v_{\nu_\hypnu})_x(0) = - c_{\nu_\hypnu} < 0$ to obtain a uniform lower bound on $u_{\nu_\hypnu}\approx v_{\nu_\hypnu}$ on $(-\infty,0)$ as $\nu_{\nu_\hypnu} \searrow0$.

By combining \Cref{t.upper.pm} and \Cref{t.upper.hyp} with a careful double limit, we arrive at the following converse to \Cref{t.lower.hyp}:
\begin{corollary}\label{c.pm}
For $\nu_{\hypnu}>0$, the minimal speeds have the asymptotics:
\be
	\lim_{-\chi \to \infty, \nu \to 0} c_{\chi,\nu}^*
		= c_{\pmeps,0}^*
		= \frac{1}{\sqrt 2}.
\ee
\end{corollary}
As the proof of \Cref{c.pm} is elementary, we omit it.

\subsection{Related work}
\label{ss:related work}

Front propagation in the three models~\eqref{e.nontw_FKPPKS},~\eqref{e.pm}, and~\eqref{e.hyp} have been the subject of intense interest in recent years; see, e.g.,~\cite{Bramburger, NadinPerthameRyzhik,SalakoShenPP,SalakoShenPE,SalakoShenXue,HamelHenderson,Henderson2021,FuGrietteMagal-JMB,Bertsch-etal-2020,MR4401504,MR4303774} for a selection of those works closest to the present one.  We note, however, that there is also active study on the behavior on finite domains: see, e.g.,~\cite{TelloWinkler, HorstmannWinkler} and articles referencing them.  The majority of the work related to~\eqref{e.nontw_FKPPKS} is dedicated to the aggregation case $\chi > 0$.   This is, perhaps, due to the historical interest in positive chemotaxis that stems from its elegant theory of blow-up~\cite{BlanchetDolbeaultPerthame, CarlenFigalli, KiselevXu}.  Although there are many interesting and difficult features to study in the positive chemotaxis case, it is unlikely that front speed-up will occur.  Heuristically,  aggregation `pushes' a traveling wave to the left on average.

\subsubsection{Propagation enhancement by advection}

In the context of turbulent combustion, it has long been understood that advection can `enhance' reactions by exposing unreacted fluid to the reacting region.  For flows that are {\em imposed}, that is, the drift term is linear, this has been thoroughly investigated over the past several decades, see~\cite{ElSmailyKirsch,HamelZlatos,MajdaSouganidis, NolenXin1,NolenXin2,RyzhikZlatos,XinYuZlatos,XinYuG,XinYuHJ} for a representative selection of the literature.  Although the interpretation of the equation is quite different here, the effect we are studying is the same.  Here, the `reaction' is reproduction, and the advective effect of chemotaxis that speeds up fronts is that individuals feel a `push' to less populated areas, where the per capita reaction rate $1-u$ is highest.

What makes the proofs difficult in our setting is the fact that the flow $-\chi v_x$ is not imposed, it depends nonlinearly on $u$.  Further, it depends nonlocally on $u$; that is, if the profile of $u$ is changed at some $x_{\rm far}$, it changes the flow $-\chi v_x(x)$ at every $x$, even if $|x_{\rm far} - x| \gg 1$ (see~\eqref{e.v'}).  Heuristically, this makes the behavior of the drift hard to predict as one must understand the entire profile of $u$, not just its local behavior.  Technically, this nonlocality means that~\eqref{e.tw} (and~\eqref{e.hyp}) do not enjoy a comparison principle.  Given the reliance on the construction of sub- and supersolutions in the study of reaction-diffusion equations, the lack of a comparison principle is a serious issue.

\subsubsection{Nonlocal advection}

The previous paragraph leads to another area that, probably due to its difficulty, has been much less explored: the role of nonlocal advection in front propagation.  We note a few examples below.  A major motivation for us is to develop an understanding of models where nonlocality plays an essential role, as few precise results exist at present.

Beyond the chemotaxis results discussed above, as well as other -taxis effects (e.g.,~\cite{demircigil2022self, calvez2022mathematical, WalkerWebb, MarchantNorburySherratt}), there is the reactive-Boussinesq system that, in a sense, dates back to~\cite{MalhamXin}.  This, roughly, represents a model for turbulent combustion in which the density differences arising from the temperature changes lead to a buoyancy force that induces a drift in the fluid.  Due to its extreme complexity, beyond the existence of traveling waves, very few precise results exist; see, e.g.,~\cite{BKV, BerestyckiConstantinRyzhik, ConstantinLewickaRyzhik, TexierPicardVolpert,VCKRR,VladimirovaRosner}.  We note also a `Burgers-Boussinesq' model studied in~\cite{CRRV,BramburgerHenderson}.

\subsubsection{Pushed and pulled fronts}

One of the oldest problems in reaction-diffusion is that of `pushed' versus `pulled' fronts: where is the important behavior that drives the front forward?  More precisely, is the wave pulled by growth far ahead of the front in the linearized regime where $u\approx 0$ or is the wave pushed by nonlinear behavior near the front?  While this can be phrased in many ways, with varying precision and technical sophistication, for our purposes, this boils down to whether the front speed is linearly or nonlinearly determined.  This question is also called the selection problem, referring to whether the linear or nonlinear behavior `selects' the minimal speed.  We refer to~\cite{EbertVanSaarloos,BBDKL, VS1, VS2, Stokes} for early discussions of this problem, as well as~\cite{avery2022pushed, AveryScheel, AHR_Burgers, AHR_rde, CrooksGrinfeld} for recent progress on it.  We note especially~\cite{GGHR} for an interesting treatment of pushed.

In the context of~\eqref{e.unscaled_tw}, fronts are `pulled' when the minimal speed $\bar c_{\chi,d} = 2$ and fronts are `pushed' when the minimal speed $\bar c_{\chi,d} > 2$.  The main result of~\cite{Henderson2021} is that there is a pulled regime when $|\chi|$ and $d$ are sufficiently small and there is a pushed regime when $1 \ll - \chi \ll d$.  The current work completes this picture by establishing the precise `pushed' behavior in the remaining regimes  $1 \ll -\chi \approx d$ and $d \ll -\chi$ (Theorems~\ref{t.lower.hyp} and \ref{t.upper.hyp}).   Very interestingly, the authors of~\cite{avery2022pushed} develop a numerical approach to finding the pushed-pulled transition point and apply it to many examples, including~\eqref{e.tw}.  Our results in the $\nu \to 0$ limit agree with their numerical conclusions: the transition takes place at $-\chi = \sfrac{1}{\eps} = 2$ (see~\eqref{e.c^*_pm}).

\subsubsection*{A summary of the results when $\chi < 0$}

We provide here a table indicating all possible asymptotic regimes and the available results.  It is not meant to be authoritative.  It obscures the various assumptions and subtleties surrounding the minimal speed traveling wave described in more detail above.  The motivation is to help the reader keep track of the many asymptotic regimes.

In the table below, we use $\approx$ to denote an exact asymptotic limit, and we use $\lessapprox$ to designate an inequality that holds in the asymptotic limit.  In contrast, we use $=$ to mean that the speed has the exact value for $\chi$ and $d$ ``sufficiently in that asymptotic regime.''

\NineColors{saturation=low}
\begin{center}
\begin{tblr}{
 hlines, vlines,
 rows = {10mm}, columns = {c},
 row{1} = {5mm}
}
	\SetRow{gray9!30} {\bf Asymptotic regime} & {\bf Speed} & {\bf Reference}
	
	\\	\hline
	
		$1 \ll -\chi \ll d$
			&   $ c \approx \frac{|\chi|}{2\sqrt d}$
			& \cite{Henderson2021}
	
		\\
		
		{$d \ll -\chi$, $ -\chi \gtrapprox 2$}
			&  $ c \approx  \frac{\sqrt{|\chi|}}{\sqrt 2} +\frac{\sqrt 2}{\sqrt{|\chi|}}$
			& Theorems \ref{t.lower.pm} and \ref{t.upper.pm}
	
	\\
	
		{$d \ll -\chi$, $ -\chi \lessapprox 2$}
			&  $c \approx  2$ 
			& Theorems \ref{t.lower.pm} and \ref{t.upper.pm}
	
		\\
		
		$-\chi \ll d \ll 1$
			& $\displaystyle c = 2$
			& \cite{Henderson2021}
			
	\\
		
		$1 \ll - \chi \approx d$
			& $\frac{\sqrt d}{2 + \left(\frac{d}{|\chi|}\right)^{\sfrac32}} \lessapprox c \lessapprox \frac{\sqrt d}{2}$
			& Theorem~\ref{t.upper.hyp}
	
	\\
		
		$- \chi \approx d \ll 1$
			& $c \approx 2$
			&  See below
			
	\\

\end{tblr}
\end{center}

Let us discuss two points.  First, one might expect that negative chemotaxis always speeds up fronts, but this is not the case.  The fourth row reflects that, if $-\chi$ and $d$ are sufficiently small, the speed is simply $2$, as in the case $\chi = d = 0$~\cite{Henderson2021} (see also~\cite{SalakoShenXue}, for an earlier proof in the $\chi >0$ case). 

Second, the bottom row of the table has not been handled directly in this or any other paper.  It is, however, a fairly straightforward case.  It is easy to check that, in this asymptotic regime, $u$ and $v$ are uniformly smooth.  Thus, by compactness and~\eqref{e.kernel}, $\|u-v\|_{L^\infty} \to 0$.  The arguments in~\cite[Lemma~3.2]{Henderson2021} then readily yield that the minimal speed wave satisfies
\be
	c\leq 2 + \frac{|\chi|}{d} \|u - v\|_{L^\infty} + C\frac{|\chi|}{\sqrt d}
		\to 2.
\ee

\subsubsection*{Acknowledgements}\ \smallskip

CH was partially supported by NSF grants DMS-2003110 and DMS-2204615.  
OT was partially supported by NSF grants DMS-1907221 and DMS-2204722. 
QG was partially supported by ANR grant  ``Indyana'' number  ANR-21-CE40-0008. 
The authors acknowledge support of the Institut Henri Poincar\'e (UAR 839 CNRS-Sorbonne Université), and LabEx CARMIN (ANR-10-LABX-59-01).  CH thanks P.E.~Souganidis for a helpful discussion of the results in~\cite{BarlesSouganidis1991}. QG thanks P.~Magal for fruitful exchanges on the hyperbolic problem.

\section{A few preliminary facts}
\label{s.general}

We shall use several results from~\cite{Henderson2021}. There, these results are established for $\chi$ either positive or negative, but here we only state the version for $\chi<0$ since that is the context of the present work.

First is a monotonicity result, with our scaling taken into account.  This is a rephrasing of \cite[Lemma~2.3]{Henderson2021}, but they are equivalent.
\begin{lemma}\cite[Lemma 2.3]{Henderson2021}
\label{lem:Chris monot}
	Let $\chi < 0$ and $\nu >0$.  Suppose that $(c,u,v)$ be a solution to~\eqref{e.tw}.  Define
	\be
		x_d = \inf \left\{x : u(x) < \frac{2}{2+\sfrac{1}{\nu}}\right\}.
	\ee
	Then $u$ is nonincreasing on $(x_d,\infty)$.
\end{lemma}
We note that an analogous result (\Cref{lem:monotonicity.hyp}) holds for the hyperbolic case~\eqref{e.hyp}.  This, however, requires some understanding of the structure of $\cZ$ (recall \Cref{def:hyp}) in order to prove, so we postpone it.

Next, we record a basic property of $u$, which is the basis of \Cref{lem:Chris monot} and will be important in many of our arguments.
\begin{lemma}\label{l.no_local_min}
Let $\chi<0$, $\nu>0$, and let $(c,u,v)$ be a solution to~\eqref{e.tw}. If $x_m$ is a  local minimum (resp. maximum) of  $u$ then  $u(x_m) \geq  (\nu + v(x_m))/(\nu + 1)$ (resp. $\leq$). 
\end{lemma}
\begin{proof}
Indeed, suppose that $x_m$ is a local minimum of $u$.  Then we have that, by~\eqref{e.tw2},
\be
	u \left(1 + \frac{v}{\nu} - \frac{\nu + 1}{\nu} u\right)
		= -(c + v_x) u_x - \frac{1}{|\chi|} u_{xx}
		\leq 0.
\ee
By the nonnegativity of $u$, we have that
\be
	1 + \frac{v}{\nu} 
		\leq \frac{\nu + 1}{\nu} u
	\quad\text{ or, equivalently, }\quad
	\frac{\nu + v}{\nu + 1} \leq u,
\ee
as claimed. The proof of the second claim is very similar and is omitted.
\end{proof}

From~\eqref{e.kernel}, an elementary calculation yields the following oft-used expression for $v_x$:
\be
\label{e.v'}
\begin{split}
	v_x(x) 
		= \frac{1}{2\nu} \int_0^\infty e^{-\frac{y}{\sqrt\nu}} (u(x+y) - u(x-y)) \dd y.
	\end{split}
\ee
Another consequence of~\eqref{e.kernel}, which is crucial to the proof of \Cref{t.lower.pm}, is the following refinement:

\begin{lemma}\label{lem:conv twice}
Let  $\nu>0$.  Then
\be\label{eq:conv twice}
	\cK_\nu = \varphi_\nu * \varphi_\nu,
\ee
where the kernel $\varphi_\nu$ is defined as,
\be
\label{eq:def phi nu}
	\varphi_\nu(x) = \frac{1}{\sqrt \nu} \varphi\left( \frac{x}{\sqrt \nu}\right)
	\qquad\text{with}\quad
	\varphi(x)
		= \frac{1}{\pi} K_0(|x|) \geq 0,
\ee
where $K_0$ is the modified Bessel function of the second kind that has the asymptotics
\be\label{e.Bessel_asymptotics}
	K_0(x)
	\approx\begin{cases}
		\log \frac{2}{|x|} - \gamma
			\qquad &\text{ if } 0 < x \ll 1,
		\\
		\sqrt{\frac{\pi}{2x}} e^{-x}
			\left(1 + O\left(\frac{1}{x}\right)\right)
			\qquad &\text{ if } x \gg 1.
	\end{cases}
\ee
Here $\gamma \approx .5772$ is the Euler-Mascheroni constant.  Further, $\|\varphi_\nu\|_{L^1} = 1$.
\end{lemma}
\begin{proof}
The scaling in $\nu$ is clear, so we show the argument for the case $\nu = 1$.  From~\eqref{e.kernel} and a direct computation, we find
\be
	\hat \cK(\xi) = \frac{1}{\sqrt{2\pi}}
		\int e^{-|x| + i \xi x} \dd x
		= \frac{1}{\sqrt{2\pi}} \frac{1}{|\xi|^2 + 1}.
\ee
Hence, we can take
\be\label{e.c71903}
	\hat \varphi(\xi)
		= (2\pi)^{\sfrac{1}{4}}\sqrt{\hat \cK}(\xi)
		=\frac{1}{\sqrt{\xi^2 +1}}.
\ee
By the convolution theorem, it follows that
\be\label{e.c71901}
	\cK = \compositeaccents{\widecheck}{\widehat{\cK}}
		= \sqrt{2\pi} \compositeaccents{\widecheck}{(\hat\varphi\hat\varphi)}
		= \varphi* \varphi,
\ee
from which~\eqref{eq:conv twice} follows.

The fact that $\varphi$ is given by~\eqref{eq:def phi nu} follows from~\cite[equation~9.6.21]{AbramowitzStegun}.  Next, the $L^1$-norm of $\varphi$ is computed easily:
\be
	\left(\int |\varphi(x)| \dd x\right)^2
		= \left(\int \varphi(x) \dd x\right)^2
		= (\sqrt{2 \pi} \hat\varphi(0))^2
		= \sqrt{2\pi} \hat \cK(0)
		= \int \cK(x) \dd x
		= 1,
\ee
where second and fourth equalities are due to the definition of the Fourier transform, the third equality is the definition of $\varphi$~\eqref{e.c71903}, and the final equality is a direct computation.  This concludes the proof.
\end{proof}

Finally, we will use the property that if $\phi\in L^1(\R)$ is even
 and $f,g\in L^1(\R)\cap L^\infty(\R)$, then,
 \be
 \label{eq:conv fact}
 \int_\R f(x)(\phi*g)(x)\, \dd x = \int_\R (f*\phi)(x)g(x)\, \dd x.
 \ee

\section{The porous medium scaling regime: proof of \Cref{t.lower.pm}}
\label{s.lower.pm}

In this section, we prove \Cref{t.lower.pm}.  
First, we obtain some preliminary estimates on $u$ and $v$.  The main estimates are a bound in $H^1$ of $v$ that is uniform in $\chi$ and $\nu$ as well as a decay of oscillations estimate on $u$.  The proof of \Cref{t.lower.pm}.(ii) is contained in \Cref{ss.pm prop pf}.

\ssection{Preliminary lemmas: regularity of $v$}
\label{ss.pm prelim}
A key element of the proof of  \Cref{t.lower.pm} is the following identity:
\begin{lemma}
    \label{lem:apriori bd nu to 0}
  Let  $(c,u,v)$ be a solution to~\eqref{e.tw}.   For any $\nu>0$ and $\chi<0$, we have 
\[
	\int (\varphi_\nu* u_x)^2    \, \dd x + \frac{1}{|\chi|}\int \frac{(u_x)^2}{u}\, \dd x + \int |u (1-u) \log u|\, \dd x=c.
\]
\end{lemma}

Before proving \Cref{lem:apriori bd nu to 0}, we point out two  bounds that follow from this and that are useful in taking the limit in \Cref{t.lower.pm}. 
\begin{corollary}\label{c.holder}
Under the assumptions of \Cref{lem:apriori bd nu to 0}, we have that
\be
	\int(v_x)^2\,\dd x\leq  c 
	\quad\text{ and } \quad
	[v]_{C^{1/2}}
		\leq \sqrt c.
\ee
\end{corollary}
\begin{proof}
	First, we have, by the expression~\eqref{eq:conv twice} for $\mathcal{K}_\nu$ (recall that $v = \cK_\nu*u$~\eqref{e.kernel}), 
	Young's inequality for convolutions,  the fact that $\varphi_\nu$ has $L^1$-norm one (\Cref{lem:conv twice}), and \Cref{lem:apriori bd nu to 0},
\be
\label{e.121056}
\|v_x\|_{L^2} = \|\phi_\nu*(\phi_\nu*u)_x\|_{L^2} \leq \|(\phi_\nu*u)_x\|_{L^2} 
 \leq \sqrt{c}.
\ee

Next, notice that, for any $x < y$,
\be\label{e.c71904}
	|v(x) - v(y)|
	\leq \int_x^y |v_x(z)| dz
		\leq \sqrt{|y - x|} \Big( \int_x^y v_x(z)^2 dz\Big)^{1/2}
		\leq \sqrt{c |x-y|},
\ee
where the last inequality follows from ~\eqref{e.121056}. The desired estimate thus follows.
\end{proof}

We now prove the crucial identity, \Cref{lem:apriori bd nu to 0}. 

\begin{proof}[Proof of Lemma \ref{lem:apriori bd nu to 0}]

We multiply the first equation in~\eqref{e.tw} by $\log u$ and integrate over $[-L,L]$ for $L\gg 1$ to obtain,
\[
	-c\int_{-L}^L u_x\log u \, \dd x =  \int_{-L}^L (v_x u)_x \log u \, \dd x + \frac{1}{|\chi|}\int_{-L}^L u_{xx}\log u\, \dd x + \int_{-L}^L u (1-u) \log u\, \dd x.
\]
Integrating by parts where possible, we find
\be
	\begin{split}
		- c &u(L) \log u(L) + cu(-L)\log u(-L)  + c\int_{-L}^L u_x\, \dd x 
		\\&
		=  v_x(L) u(L) \log u(L) - v_x(-L) u(-L) \log u(-L) - \int_{-L}^L v_x u_x  \, \dd x
			\\&\quad
			+\frac{1}{|\chi|} (u_x(L) \log u(L)  - u_x(-L) \log u(-L)) - \frac{1}{|\chi|}\int_{-L}^L \frac{(u_x)^2}{u}\, \dd x
			+ \int_{-L}^L u (1-u) \log u\, \dd x.
	\end{split}
\ee
We claim that all boundary terms vanish as $L\to\infty$.  Indeed, by assumption $u(-\infty) = 1$ and $u(+\infty)= 0$, so that elliptic regularity implies that $u_x(\pm \infty) = 0 = u_{xx}(\pm \infty)$.  Similarly, $v_x(\pm \infty) = 0 = v_{xx}(\pm\infty)$.  Thus, after integrating also the last term on the left, we find
\be
	- c =  - \int v_x u_x  \, \dd x - \frac{1}{|\chi|}\int \frac{(u_x)^2}{u}\, \dd x + \int u (1-u) \log u\, \dd x.
\ee

Here is the key step. From the expression~\eqref{eq:conv twice} for $v$ and properties of convolution, we find $v_x=\varphi_\nu*\varphi_\nu*u_x$. Using this in the first term on the right-hand side, and then recalling the property~\eqref{eq:conv fact} of convolution, we find,
\begin{align}
-c &=  - \int (\varphi_\nu*\varphi_\nu*u_x) u_x  \, \dd x - \frac{1}{|\chi|}\int \frac{(u_x)^2}{u}\, \dd x + \int u (1-u) \log u\, \dd x\\
&=  - \int (\varphi_\nu* u_x)^2  \, \dd x - \frac{1}{|\chi|}\int \frac{(u_x)^2}{u}\, \dd x + \int u (1-u) \log u\, \dd x.
\end{align}
The desired estimate follows upon recalling that $(1-u)\log u \leq 0$ for all $u \geq 0$. 
\end{proof}

\ssection{Preliminary lemmas: decay of oscillations of $u$}

Formally taking $\nu$ to zero in the second equation of~\eqref{e.tw} indicates that we should expect $u-v$ to converge to zero.  This is exactly what we now prove, in a quantitative way.  A significant issue, though, is that we do not have {\em any} regularity estimates on $u$ that are uniform in $\nu$ and $\chi$. 
To get around this, we use the uniform estimates that we have established on $v$ to prove the following decay of oscillations of $u$ as $\nu \to 0$.

\begin{lemma}
\label{lem:u-v}
If $(c,u,v)$ is a solution to~\eqref{e.tw} with $\chi<0$ and $0<\nu<\infty$, then there is a universal constant $C>0$ such that
\be\label{e.c81602}
	\max_{[x_0-\nu^{\sfrac{1}{4}}, x_0 + \nu^{\sfrac{1}{4}}]} u(x)
		- \min_{[x_0-\nu^{\sfrac{1}{4}}, x_0 + \nu^{\sfrac{1}{4}}]} u(x)
		\leq C(\sqrt c + 1) \nu^{\sfrac{1}{8}}.
\ee
As a result, we have
\be\label{e.c81601}
	\|u - v\|_{L^\infty} \leq C (\sqrt c + 1) \nu^{\sfrac{1}{8}}.
\ee
\end{lemma}

\begin{proof}
First, note that we may restrict to $\nu$ sufficiently small so that
\be
\label{eq:nudelta ineq}
	\frac{1}{ e^{\nu^{-\sfrac{1}{4}}}}
		< \nu^{\sfrac{1}{8}}
		\quad\text{ and }\quad \nu < 1.
\ee
Indeed, when $\nu$ is large enough that~\eqref{eq:nudelta ineq} does not hold then the proof is finished by possibly increasing $C$ and using~\eqref{e.0<u<1}.

Next, we note that~\eqref{e.c81601} follows easily from~\eqref{e.c81602} and~\eqref{e.kernel}.  Indeed, for any $x_0$, we have
\[
	\begin{split}
		v(x_0) &\leq \int_{-\nu^{\sfrac{1}{4}}}^{\nu^{\sfrac{1}{4}}} \frac{e^{-\frac{|y|}{\sqrt\nu}}}{2\sqrt\nu} \left(\max_{[x_0-\nu^{\sfrac{1}{4}}, x_0 + \nu^{\sfrac{1}{4}}]} u\right) \dd y
				+ \int_{[-\nu^{\sfrac{1}{4}}, \nu^{\sfrac{1}{4}}]^c} \frac{e^{-\frac{|y|}{\sqrt\nu}}}{2\sqrt\nu} \dd y
			\\&
			\leq \int_{-\nu^{\sfrac{1}{4}}}^{\nu^{\sfrac{1}{4}}} \frac{e^{-\frac{|y|}{\sqrt\nu}}}{2\sqrt\nu} \left(u(x_0) + C (\sqrt c + 1) \nu^{\sfrac{1}{8}}\right) \dd y
				+ e^{-\nu^{-\sfrac{1}{4}}}
			\leq u(x_0) + (C (\sqrt c + 1) + 1) \nu^{\sfrac{1}{8}}.
	\end{split}
\]
The argument for the lower bound on $v(x_0)$ is similar, and~\eqref{e.c81601} follows.

Fix any $x_0 \in \R$. We actually shall prove
\be\label{e.c72002}
	\begin{split}
	&
		\max_{[x_0-\nu^{\sfrac{1}{4}}, x_0 + \nu^{\sfrac{1}{4}}]} u(x)
			\leq v(x_0) + (\sqrt{3c}+1)\nu^{\sfrac{1}{8}}
		\qquad\text{ and}
	\\&
		\min_{[x_0-\nu^{\sfrac{1}{4}}, x_0 + \nu^{\sfrac{1}{4}}]} u(x)
			\geq v(x_0) - (\sqrt{3c}+1)\nu^{\sfrac{1}{8}},
	\end{split}
\ee
from which~\eqref{e.c81602} follows. 
We prove these two bounds separately, beginning with the bound on the maximum. A key ingredient in both is the regularity of $v$: due to \Cref{c.holder}, for all $x \in [x_0-3\nu^{\sfrac{1}{4}}, x_0 + 3\nu^{\sfrac{1}{4}}]$, 
\be\label{e.c7155}
	|v(x)-v(x_0)|
		\leq   \sqrt{3c}\nu^{\sfrac{1}{8}}.
\ee

We now begin by proving the upper bound on the maximum in~\eqref{e.c72002}.  Our goal is to leverage \Cref{l.no_local_min} to obtain an upper bound on $u$.  A substantial complication, however, is that \Cref{l.no_local_min} is only applicable at a local extremum, while the maximum of $u$ over $[x_0 - \nu^{\sfrac{1}{4}}, x_0 + \nu^{\sfrac{1}{4}}]$ may occur at the boundary and we do not {\em a priori} have control over $u(x_0 \pm \nu^{\sfrac{1}{4}})$.  As such, our approach is to identify an interval {\em containing} $[x_0-\nu^{\sfrac{1}{4}}, x_0 + \nu^{\sfrac{1}{4}}]$ for which we can establish a suitable bound on $u$ at the end points.  

We now find such an interval.  For any $x \in [x_0 - 2\nu^{\sfrac{1}{4}}, x_0 + 2\nu^{\sfrac{1}{4}}]$, let
\be
	y_x := \underset{[x-\nu^{\sfrac{1}{4}}, x + \nu^{\sfrac{1}{4}}]}{\argmin} u.
\ee
Then, due to~\eqref{e.kernel},~\eqref{eq:nudelta ineq}, and the fact that $\|u\|_{L^\infty}\leq 1$, we find
\be\label{e.c81207}
	\begin{split}
	v(x)
		&\geq \int_{[x - \nu^{1/4}, x + \nu^{1/4}]} \frac{1}{2\sqrt \nu}e^{-\frac{|x-y|}{\sqrt\nu}} \left(\min_{[x - \nu^{1/4}, x + \nu^{1/4}]} u\right) dy
		\\&= u(y_x) \left(1 - e^{-\nu^{-1/4}}\right)
		\geq u(y_x) - \nu^{\sfrac{1}{8}}.
	\end{split}
\ee
Thus, rearranging the above inequality and using (\ref{e.c7155}) yields,
\[
	u(y_x)
		\leq v(x)+ \nu^{\sfrac{1}{8}}
		\leq v(x_0)+ \sqrt{3c}\nu^{\sfrac{1}{8}} + \nu^{\sfrac{1}{8}}. 
\]
Applying~\eqref{eq:nudelta ineq},  we find
\be
\label{eq:u(yx)}
	u(y_x)\leq v(x_0)+( \sqrt{3c}+1)\nu^{\sfrac{1}{8}}.
\ee
Notice that this is exactly the inequality that we wished to prove, albeit only for $y_x$.  This suggests that the interval we should work on has endpoints $y_x$ for well-chosen $x$, which is what we do now.

Let $y_\pm = y_{x_0 \pm 2\nu^{\sfrac{1}{4}}}$ and notice that, due to~\eqref{eq:nudelta ineq} and the definition of $y_\cdot$,
\be\label{e.x_pm}
	y_- \in [x_0 - 3\nu^{\sfrac{1}{4}}, x_0 - \nu^{\sfrac{1}{4}}]
	\qquad\text{ and }\qquad
	y_+ \in [x_0 + \nu^{\sfrac{1}{4}}, x_0 + 3\nu^{\sfrac{1}{4}}].
\ee
As a result, $[x_0 - \nu^{\sfrac{1}{4}}, x_0 + \nu^{\sfrac{1}{4}}] \subset [y_-, y_+]$.   We shall now use the argument, based on \Cref{l.no_local_min} and outlined above, to
 establish, 
\be\label{e.c7154}
	\max_{[y_-, y_+]} u(x)
		\leq v(x_0) + ( \sqrt{3c}+1) \nu^{\sfrac{1}{8}}.
\ee
According to~\eqref{e.x_pm}, the first inequality in~\eqref{e.c72002} follows from~\eqref{e.c7154}.

Let $x_m\in[y_-, y_+]$ be the maximum of $u$ over $[y_-,y_+]$.  If $x_m$ is one of $y_\pm$, then~\eqref{e.c7154} holds by virtue of~\eqref{eq:u(yx)}.  Thus, let us consider the case that $x_m \in (y_-, y_+)$.  Then $x_m$ is the location of a local maximum.  Using Lemma \ref{l.no_local_min}, we find, at $x_m$,
\be
	\max_{[y_-,y_+]} u
	= u(x_m)
	\leq \frac{1}{1+\nu} \left( v(x_m) + \nu\right)
	\leq v(x_m) + \nu.
\ee
Recalling~\eqref{e.c7155} and~\eqref{e.x_pm}, we have
\be
	\max_{[y_-,y_+]} u
		\leq v (x_0)
			+  \sqrt{3c}\nu^{\sfrac{1}{8}} + \nu
		\leq v(x_0) + (\sqrt{3 c}+1) \nu^{\sfrac{1}{8}}.
\ee
This concludes the proof of the first inequality in~\eqref{e.c72002}.

The proof of the second inequality in~\eqref{e.c72002} follows along the same lines.  We include it in order to show the necessary (slight) modifications; however, we provide less exposition due to its similarities to the proof above.  First, fix any $x \in [x_0 - 3\nu^{\sfrac{1}{4}}, x_0 + 3\nu^{\sfrac{1}{4}}]$ and let
\be
	\tilde y_x = \underset{ [x_0 - \nu^{\sfrac{1}{4}}, x_0 + \nu^{\sfrac{1}{4}}]}{\argmax} u.
\ee
Then, using the expression~\eqref{e.kernel} for $v$ and the definition of $\mathcal{K}_\nu$ we find,
\be
	\begin{split}
	v(x)  &= \int_{[x-\nu^{1/4}, x+\nu^{1/4}]} \cK_\nu(x-y)u(y)\, dy + \int_{|y-x|>\nu^{1/4}} \cK_\nu(x-y)u(y)\, dy 
		\\&\leq \int_{[x-\nu^{1/4}, x+\nu^{1/4}]}
			\frac{1}{2\sqrt\nu}
			e^{-\frac{|x-y|}{\sqrt\nu}}  
			\left( \max_{[x-\nu^{1/4}, x+\nu^{1/4}]} u\right) dy
			+  e^{-\nu^{-1/4}}
		\\&
		\leq u(\tilde y_x)
			+  \nu^{\sfrac{1}{8}}.
	\end{split}
\ee 
Thus, recalling~\eqref{e.c7155}, we find
\be
\label{e.c9274}
	v(x_0)
		\leq v(x)+ \sqrt{3c}\nu^{\sfrac{1}{8}}
		\leq u(\tilde{y}_x)+(\sqrt{3 c} + 1)\nu^{\sfrac{1}{8}}.
	\ee
	As in the proof of the first inequality in~\eqref{e.c72002}, let $\tilde{y}_\pm =\tilde{y}_{ x_0 \pm 2\nu^{\sfrac{1}{4}}}$, and, as before, notice $[x_0-\nu^{1/4}, x_0+\nu^{1/4}]\subset[\tilde{y}_-, \tilde{y}_+]$.  Hence, it is enough to establish the stronger claim: 
	\be
	\label{eq:stronger min}
	\min_{[\tilde{y}_-, \tilde{y}_{+}]} u(x)
		\geq v(x_0) - ( \sqrt{3c}+1) \nu^{\sfrac{1}{8}}.
\ee
To this end, let $x_{\min}$ be the location of the minimum of $u$ over the interval $[\tilde{y}_-, \tilde{y}_+]$.  If $x_{\min} = \tilde{y}_{\pm}$, then we are finished by~\eqref{e.c9274}.  Otherwise, $x_{\min}$ is an interior minimum and we find, via Lemma \ref{l.no_local_min}, 
\be
		u(x_{\min})
			\geq \frac{v(x_{\min}) + \nu}{1 + \nu} 
			\geq \frac{v(x_{\min})}{1 + \nu}
			\geq (1-\nu) v(x_{\min})
			\geq v(x_{\min}) - \nu
			\geq v(x_0) -  \sqrt{3c} \nu^{\sfrac{1}{8}} - \nu,
\ee
where we use that $\sfrac{1}{(1+\nu)} \geq 1 - \nu$ in the third inequality, that $v \leq 1$ to get the fourth inequality, and the estimate~\eqref{e.c7155} in the last inequality.  The claim then follows by~\eqref{eq:nudelta ineq}.  This concludes the proof.
\end{proof}

\ssection{Proof of \Cref{t.lower.pm}} \label{ss.pm prop pf}

We establish the second part of \Cref{t.lower.pm}.
\begin{proof}[Proof of \Cref{t.lower.pm}.(ii)]
	First we address the notion of convergence.  Since $c_n$ is bounded uniformly above, by assumption, there is $c$ such that $c_n \to c$ as $n\to\infty$
	up to passing to a subsequence.  Similarly, by \Cref{c.holder} 
	and the bound $\|v\|_{L^\infty}\leq 1$, we obtain $u \in H^1\cap C^{\sfrac{1}{2}} \cap L^\infty$
	such that, as $n\to\infty$,
	\be\label{e.c72003}
		v_n \to u
			\qquad \text{ weakly in } H_{\rm loc}^1
				\text{ and strongly in } C^\alpha_{\rm loc},
	\ee
	for any $\alpha \in (0,\sfrac{1}{2})$.
	Finally, due to \Cref{lem:u-v} and~\eqref{e.c72003}, we have
	\be\label{e.c72006}
		u_n \to u
			\qquad\text{ in } L^\infty_{\rm loc}. 
	\ee
	
	Next, we investigate what qualitative properties $u$ enjoys.  First, by~\eqref{e.0<u<1}, we have $0 \leq u \leq 1$.  Next, from the normalization~\eqref{e.normalization.pm} and the uniform convergence of $u_n$ to $u$~\eqref{e.c72006}, it follows that
	\be\label{e.c81606}
		\min_{x\leq 0} u(x) = u(0) = \delta.
	\ee

	Next, we argue that $u(-\infty) = 1$.  By \Cref{lem:apriori bd nu to 0},  \Cref{c.holder}, and the convergence of $v_n$ and $u_n$ to $u$, implies that
	\be\label{e.c81607}
		[u]_{C^{\sfrac{1}{2}}}
			\leq \sqrt c
			\quad\text{ and }\quad
		\int |u_x|^2 \dd x + \int |u(1 - u) \log u| \dd x
			\leq c
			< +\infty.
	\ee
	By~\eqref{e.c81606} and the regularity of $u$~\eqref{e.c81607}, the above implies that
	\be\label{e.c81701}
		u(-\infty) = 1.
	\ee
	We note that a byproduct of~\eqref{e.c81607} is that
	\be
		c> 0
	\ee
	as~\eqref{e.c81606} and~\eqref{e.c81701} imply that $u$ is nonconstant.  We use this in the proof of monotonicity below.  Further, it must be $u(+\infty) = 0$ (otherwise the integral on the left hand side of~\eqref{e.c81607} would be infinite). 
	
	Next, we show that $u$ is a distributional solution to~\eqref{e.pm}.	
	To this end, fix any $\psi \in C^2_c(\R)$.  At the level of $u_n$, multiply~\eqref{e.tw} by $\psi$ and integrate by parts to obtain:
	\be
		c_n \int u_n \psi_x \dd x
			+ \int (v_n)_x u_n \psi_x \dd x
			= \frac{1}{|\chi_n|} \int u_n \psi_{xx} \dd x
				+ \int u_n(1-u_n) \psi \dd x.
	\ee
	Using~\eqref{e.c72003} and~\eqref{e.c72006} and taking $n\to\infty$, we find
	\be\label{e.distributional}
		c \int u \psi_x \dd x
			+ \int u_x u \psi_x \dd x
			= \eps \int u \psi_{xx} \dd x
				+ \int u (1 - u ) \psi \dd x,
	\ee
	that is, $\bar u$ is a distributional solution to~\eqref{e.pm}.  
\end{proof}

To deduce the desired bound on the limiting speed, we appeal to the results of~\cite{Aronson1980}  concerning solutions of~\eqref{e.pm}. However, the notion of solution given in  \cite[equation (2.3)]{Aronson1980} is, at first glance, different from that of distributional solutions. However, as we establish below, the two notions are equivalent. (We also remark that these notions are the same as that of viscosity solution \cite{CrandallIshiiLions}; however, since we do not use this fact in our work, we do not provide a proof).

We also note that if $\ep>0$, then solutions of~\eqref{e.pm} are classical; thus, the following lemma is only interesting when considering solutions of~\eqref{e.pm} in the case $\ep=0$. 

\begin{lemma} 
\label{p.pm classical}
Suppose $0<\alpha<1$ and 
$u\in H^1_{\mathrm{loc}}(\mathbb{R})$ satisfying $0\leq u\leq 1$ on $\R$, with $\lim_{x\rightarrow \infty}u(x)=1$ and $\lim_{x\rightarrow -\infty}u(x)=0$. Suppose $c>0$. Then the following are equivalent:
\begin{enumerate}
\item \label{item:dist soln} $(u,c)$ is a distributional solution of~\eqref{e.pm}.
\item \label{item:Aronson soln} there exists $\omega\in (-\infty,+\infty]$ such that $\{u>0\}=(-\infty, \omega)$ and $u\in C^{2,\alpha}(-\infty, \omega)$.  Further, $u$ satisfies~\eqref{e.pm} classically on $(-\infty, \omega)$ and is strictly decreasing on $(-\infty, \omega)$.
\end{enumerate}
 \end{lemma}
 We note that   the hypotheses of the lemma and item \eqref{item:Aronson soln}  comprise exactly the definition of solution in~\cite{Aronson1980}.

 We first state and prove the following basic fact that we will use twice in the proof of \Cref{p.pm classical}.
 \begin{lemma}
 \label{lem.pm no min}
 Suppose $0\leq u\leq 1$ is a classical, nonconstant solution of~\eqref{e.pm} on $(a,b)$, for some $a<b$. Suppose  $u$ has a local minimum at $\tilde x \in (a,b)$ and $u(\tilde{x})\neq 1$. Then $u(\tilde{x})=0$.
 \end{lemma}
 \begin{proof} 
 Since $\tilde{x}$ is an interior minimum, we have $u_x(\tilde{x})=0$ and $-u_{xx}(\tilde{x})\leq 0$. Evaluating~\eqref{e.pm} at $\tilde{x}$ then yields,
\be
0\geq u(\tilde x)(1-u(\tilde x)).
\ee
 Since $u(\tilde{x})\neq 1$ and $0\leq u\leq 1$ holds, we find $u(\tilde{x})=0$, as desired.
 \end{proof}

\begin{proof}[Proof of \Cref{p.pm classical}]
If $\eps >0$, elliptic regularity theory implies that $u$ is smooth and is a classical solution to~\eqref{e.pm} and therefore the conclusion of the proposition holds. Thus, we assume $\ep=0$. 

Let $u$ be as in the statement of the lemma. 
We shall establish that $u$ is a distributional solution if and only if item \eqref{item:Aronson soln} holds. 

If $u$ is a solution in the sense of item \eqref{item:Aronson soln}, then the fact that $u$ is a classical solution of~\eqref{e.pm} on the region where $u>0$ immediately yields that $u$ is a distributional solution. 
	
Thus, let us assume that $u$ is a distributional solution of~\eqref{e.pm}. 	 Let
	\be
		\omega = \inf\{ x : u(x) = 0\}.
	\ee
	We remark that our assumption $u(-\infty)=1$ implies $\omega>-\infty$. Fix any $x_0 < \omega$.  By the continuity of $u$ and the fact that $u(-\infty)=1$, we have
	\be
		\inf_{(-\infty, x_0)}u(x) > 0.
	\ee
	Thus, \eqref{e.pm} is a uniformly elliptic equation with $C^{\alpha}$ coefficients on $(-\infty,x_0)$.  Elliptic regularity theory implies that $u \in C^{2,\alpha}(-\infty, x_0)$ and solves~\eqref{e.pm} in the classical sense on $(-\infty,x_0)$.  Since $x_0$ is arbitrary, it follows that $u$ solves~\eqref{e.pm} classically on $(-\infty, \omega)$.

	We shall now establish that $u$ is nonincreasing on $(-\infty, \omega)$. Indeed, suppose not. Then $u$ has a local minimum at some $\tilde{x} \in (-\infty, \omega)$ with $u(\tilde{x})\neq 1$ (here we are using that $u$ is nonconstant, which is true by assumption.) \Cref{lem.pm no min} therefore implies $u(\tilde{x})=0$, contradicting the definition of $\omega$. Therefore, we conclude that $u$ is indeed nonincreasing on $(-\infty, \omega)$.

	Let us once more fix an arbitrary $x_0\in (-\infty, \omega)$. A standard strong maximum principle argument, after differentiating~\eqref{e.pm}, shows that $u_x < 0$ on $(-\infty, x_0)$. Since $x_0$ is arbitrary, it follows that $u$ solves~\eqref{e.pm} classically on $(-\infty, \omega)$ and is strictly decreasing there.  
	
	Now we note that if $\omega = +\infty$, then the proof is complete. So, let us assume $\omega$ is finite. 
We shall now establish that $u\equiv 0$ on $(\omega, \infty)$. To this end, suppose not. Hence, suppose, by way of contradiction, that there is some $x_M>\omega$ such that $u(x_M) > 0$.  Since $u$ cannot have a positive local minimum by \Cref{lem.pm no min}, either $u$ has a global maximum on $(\omega,\infty)$ or $u$ is nondecreasing on $(\omega,\infty)$.  In either case, we may take $x_M$ such that, for all sufficiently small $\mu$,
	\be\label{e.c81702}
		u(x_M) \in (0,1) \qquad\text{ and }\qquad
		u_x \geq 0
			~~\text{ on }~~
		(x_M-\mu, x_M).
	\ee
	As we showed earlier in the proof, we have,
	\be\label{e.c81607bis}
		u > 0 \text{ and } u_x \leq 0
		~~\text{ in }~~
		(-\infty, \omega).
	\ee
	For any $\mu > 0$, take $\psi_\mu$ to be a $C^2$ test function such that
	\be\label{e.c81608}
		\begin{split}
		&\supp(\psi_\mu) \subset [\omega - \mu, x_M]
		\quad
		(\psi_\mu)_x > 0 
			\quad\text{ in } (\omega - \mu, \omega),
		\quad
		\psi_\mu = 1
			\quad\text{ on } [\omega, x_M-\mu]
		\\&
		\text{and }\quad
		(\psi_\mu)_x < 0
			\quad\text{ in } (x_M-\mu, x_M ).
		\end{split}
	\ee
	Applying~\eqref{e.distributional} with this choice of test function $\psi_\mu$ and recalling that $\eps = 0$, we find
	\be\label{e.c81609}
		\begin{split}
			\int &\psi_\mu u (1-u) \dd x
			= c \int u(\psi_\mu)_x \dd x
				+ \int u_x u (\psi_\mu)_x \dd x
			\\&
			= c \int_{\omega-\mu}^{\omega} u (\psi_\mu)_x \dd x
				+ c \int_{x_M-\mu}^{x_M}
					 u (\psi_\mu)_x  \dd x
				+ \int_{\omega-\mu}^{\omega} u_x u (\psi_\mu)_x \dd x
				+ \int_{x_M-\mu}^{x_M} u_x u (\psi_\mu)_x \dd x
			\\&
			\leq c^{3/2} \sqrt \mu
				+ c \int_{x_M-\mu}^{x_M}
u (\psi_\mu)_x  \dd x
				+ \int_{\omega-\mu}^{\omega} u_x u (\psi_\mu)_x \dd x
				+ \int_{x_M-\mu}^{x_M} u_x u (\psi_\mu)_x \dd x
			\\&
			\leq c^{3/2} \sqrt \mu
				+ \int_{\omega-\mu}^{\omega} u_x u (\psi_\mu)_x \dd x
				+ \int_{x_M-\mu}^{x_M} u_x u (\psi_\mu)_x \dd x
			\\&
			\leq c^{3/2} \sqrt \mu
				+ \int_{x_M-\mu}^{x_M} u_x u (\psi_\mu)_x \dd x
			\leq c^{3/2} \sqrt \mu.
		\end{split}
	\ee
	The first inequality above comes from the choice of $\psi_\mu$~\eqref{e.c81608} and the H\"older bound on $u$~\eqref{e.c81607}.  The second inequality comes from the fact that, for $\mu$ sufficiently small, $u>0$ and $(\psi_\mu)_x<0$ on $[x_M-\mu,x_M]$.  The third inequality comes from the monotonicity of $u$ on $(-\infty, \omega)$, as recalled in~\eqref{e.c81607bis}, and the choice of $\psi_\mu$~\eqref{e.c81608}. The final inequality comes from~\eqref{e.c81702} and~\eqref{e.c81608}.

	Thus, we conclude that
	\be
		\int_{\omega}^{x_M} u(1-u) \dd x
			= \lim_{\mu\to0} \int \psi_\mu u(1-u) \dd x
			\leq 0.
	\ee
		The left hand side is strictly positive since $u$ is continuous and $u(x_M) \in (0,1)$.  This is a contradiction and concludes the proof that $u$ is nonincreasing.  This completes the proof.
\end{proof}

Finally, we turn to:

\begin{proof}[Proof of \Cref{t.lower.pm}.(i) using \Cref{t.lower.pm}.(ii)]
Fix $\ep\geq 0$. It is enough to show that for any sequence $(\chi_n, \nu_n)$, with $-\frac{1}{\chi_n}\rightarrow \ep$, $\nu_n\rightarrow 0$,  and corresponding traveling wave solutions $(c_n, u_n, v_n)$ to~\eqref{e.tw}, satisfy the lower bound
\be
	\liminf_{n\to\infty} c_n \geq c_{\pmeps}^*.
\ee
Thus, consider such a sequence. If 	$\liminf_{n\to\infty} c_n = \infty$, 
then we are finished.  Hence, we may assume that the limit inferior of $c_n$ is finite and, up to passing to a subsequence, we may assume that $c_n$ converges to it.  
In other words, we assume that there is $c \in [0,\infty)$ such that
\be
	\lim_{n\to\infty} c_n = c.
\ee
Thus, we are in the setting of \Cref{t.lower.pm}.(ii), and hence obtain a subsequence $(c_{n_k}, u_{n_k}, v_{n_k})$ such that $(c_{n_k}, u_{n_k})$ converges to a distributional solution $(c,u)$ of~\eqref{e.pm}. If $\ep>0$, then the results of \cite{KawasakiShigesadaIinuma} imply $c\leq c^*_{\pmeps}$. And, in the case $\ep=0$,   the results of \cite{Aronson1980}, together with \Cref{p.pm classical}, imply  $c\leq c^*_{\pmeps}$. Thus the proof is complete.
\end{proof}

\section{The hyperbolic scaling regime: the structure of $\cZ$, monotonicity, exponential decay, and the lower bound on the speed}\label{s.general.hyp}

In this section, we deduce several general facts that hold for any solution to~\eqref{e.hyp}.  These are used in several places throughout the sequel, so it is convenient to establish them here, although many proofs are postponed to \Cref{s.technical} due to their length and complexity.

\ssection{The structure of $\cZ$ and the behavior of $u$ on $\cZ^c$}

\ssection{Monotonicity of $u$ for the hyperbolic model}

We state a monotonicity lemma that applies also to the hyperbolic model (notice that \Cref{lem:Chris monot} is stated only for~\eqref{e.tw}). 
\begin{lemma}\label{lem:monotonicity.hyp}
	Suppose that $(c,u,v)$ is a traveling wave solution to~\eqref{e.hyp}.  Let
	\be
		x_d = \inf \left\{x : u(x) < \frac{2}{2+\sfrac{1}{\nu}}\right\}.
	\ee
	Then $u$ is nonincreasing on $(x_d, \infty)$.

\end{lemma}
\begin{proof}
We consider two cases based on \Cref{prop:cardZleq1}.  First, when $\cZ = \emptyset$, $u$ and $v$ are smooth.  In this case, the argument of \Cref{lem:Chris monot} can be repeated verbatim.  It is sketched below (see the final paragraph of this proof).

Next consider the case when $\cZ = \{x_0\}$ for some $x_0$.  There is nothing to prove on $(x_0,\infty)$ as $u \equiv 0$ on that domain.  Additionally, recalling \Cref{prop:cardZleq1}, we have $u(x_0^-) > u(x_0^+)$.

On the other hand, the proof that $u$ is nonincreasing on $(x_d, x_0)$ is exactly the same as \Cref{lem:Chris monot}, so we only sketch it briefly.  It is proved by contradiction.  Take the leftmost local minimum lying below $\sfrac{2}{(2+\nu)}$.  Then, due to~\eqref{e.hyp}, $u = \sfrac{(\nu + v)}{(\nu+1)}$ at this point.  On the other hand, $u \geq \sfrac{2}{(2+\nu)}$ to the left of this point.  Using~\eqref{e.kernel}, this implies that $v > \sfrac{1}{(2+\nu)}$ at this point, which, in turn, yields $\sfrac{(\nu + v)}{(\nu+1)} > \sfrac{2}{(2+\nu)}$ at this point, a contradiction.  The only wrinkle in this context is the possible discontinuity of $u$, but this is avoided by the domain restriction assumption: $x < x_0$.  As such, we omit the details and refer the interested reader to \cite[Lemma~2.3]{Henderson2021}.
\end{proof}

\ssection{Quantified exponential decay}

A key part of our work in \Cref{p.speed.hyp} and \Cref{t.upper.hyp} is a quantitative exponential decay bound that we state below.  For convenience, up to translation, we use the normalization~\eqref{e.normalization.hyp}, so that
\be\label{e.c8027}
	u (0) \leq \frac{\nu}{\nu + 1}.
\ee
We note that, importantly, the proof below has the advantage of applying equally to three different settings, one of which is `slab problem,' where $(u,v)$ satisfies~\eqref{e.tw} only on the finite interval $[-L,L]$ (see \Cref{s.upper} and, more specifically, \eqref{eq:topologicaldegree-defFtau}).  In this case, we must specify which choice of $v$ we use  since boundary data is not imposed. We take the solution $v$ defined by
\be\label{e.c080101}
	 v(x) = \cK_\nu * \bar u
	 \quad\text{ where }~~
	 \bar u(x) = \begin{cases}
	 			1 \quad&\text{ for } x \leq -L,\\
				u(x) \quad&\text{ for } x \in [-L,L],\\
				0 \quad&\text{ for } x \geq L.
	 		\end{cases}
\ee
The result is the following:
\begin{proposition}
	\label{lem:exp decay}
	Fix any $\chi_0<0$,  $\nu_M>{\nu}_m>0$ and $C_M>0$.  Suppose that $c \in [0,C_M]$, $\nu \in [\nu_m, \nu_M]$, and $u$ satisfies~\eqref{e.c8027}.  Assume one of the three settings below: either
	\begin{enumerate}[(i)]
		\item $\chi \leq \chi_0$ and $(c,u,v)$ solves~\eqref{e.tw} on $\R$ with $u(+\infty) = 0$, or
		\item $(c,u,v)$ solves~\eqref{e.hyp} on $\R$ with $u(+\infty) = 0$, or
		\item $(c,u,v)$ solves~\eqref{e.tw} on $[-L,L]$ for some $L>0$ with $u(-L) = 1$, $u(L)=0$, and  $u(x)\equiv 0 $ for $x\geq L$ and $v$ satisfies~\eqref{e.c080101}.
	\end{enumerate}
	Then
	\be\label{e.exponential_decay}
	u(x)
	\leq C u(0) e^{- \theta x} \text{ for all $x\geq 0$},
	\ee
	where $C$ and $\theta$ depend only on $\chi_0$, ${\nu}_M$, ${\nu}_m$, and $C_M$. 
\end{proposition}

Let us point out the content of the above result.  Working on (at least) the half-line $[0,\infty)$, standard ODE theory tells us that $u$ decays exponentially {\em eventually}.  It, however, does not tell us {\em when} the exponential decay `kicks in.'  This leaves open the possibility of a sequence of traveling waves $(c_n, u_n, v_n)$ of~\eqref{e.hyp} such that $u_n$ is `nearly constant' and $O(1)$ on $[0,n]$ for a large $n$ before $u_n$ begins decaying exponentially to zero.  The above proposition rules this out.  This is crucial in the proof of the lower bound~\eqref{e.c81209} in \Cref{p.speed.hyp}.

As the proof is quite long and technical, it is postponed to \Cref{s.technical}.

\ssection{General bounds on the traveling wave speed: \Cref{p.speed.hyp}}\label{s.speed.hyp}

We now prove the bounds on the minimal speed $c_{\rm hyp}^*$.

\begin{proof}[Proof of \Cref{p.speed.hyp}]
We consider two cases depending on the structure of $\cZ$.

\smallskip
\noindent{\bf Case one: $\cZ = \{x_0\}$.} 
First, we show the upper bound as it is very simple.  
By \Cref{prop:cardZleq1},  
$u(x) = 0$ for all $ x> x_0$.  We thus note that
\be
	\begin{split}
		c
			&= - v_x(x_0)
			= \frac{1}{2\nu} \int \sign(y) e^{-\frac{|y|}{\sqrt\nu}} u(x_0-y) dy
			= \frac{1}{2\nu} \int_0^\infty e^{-\frac{y}{\sqrt\nu}} u(x_0-y) dy
			\\&
			\leq \frac{1}{2\nu} \int_0^\infty e^{-\frac{y}{\sqrt\nu}} dy
			= \frac{1}{2\sqrt\nu}.
	\end{split}
\ee
For future use, we point out that, from the third inequality and the fact that $u \equiv 0$ on $(x_0,\infty)$, we have
\be\label{e.c81605}
	c = -v_x(x_0) = \frac{1}{\sqrt\nu} v(x_0).
\ee

	Next, we establish a lower bound for $v(x_0)$.  
By definition of $\cZ$ (\Cref{def:hyp}) and by \Cref{prop:cardZleq1}, we find
\be\label{e.c72902}
	v_x(x_0) = - c,
	\quad
	u(x_0^-) = \frac{\nu + v(x_0)}{\nu + 1},
	\quad\text{ and }\quad
	u(x_0^+) = 0.
\ee
By \Cref{lem:monotonicity.hyp},
\be\label{e.c81204}
	u(x) \geq \min\left\{\frac{2}{2+\sfrac{1}{\nu}},\frac{\nu + v(x_0)}{\nu + 1}\right\}
		\qquad\text{ for all } x < x_0.
\ee
We claim that
\be\label{e.c72903}
	v(x_0) \geq \frac{\nu}{2\nu + 1}.
\ee
If this were not true then the minimum in~\eqref{e.c81204} is given by
\be
	u(x) \geq \frac{\nu + v(x_0)}{\nu + 1}.
\ee
It then follows that
\be
	v(x_0)
		= \int_{-\infty}^{x_0} \frac{e^{-|x_0 - y|/\sqrt\nu}}{2\sqrt\nu} u(y) dy
		\geq \int_{-\infty}^{x_0} \frac{e^{-|x_0 - y|/\sqrt\nu}}{2\sqrt\nu} \frac{\nu + v(x_0)}{\nu + 1} dy
		= \frac{1}{2} \frac{\nu + v(x_0)}{\nu + 1},
\ee
which implies that~\eqref{e.c72903} does in fact hold.

On the other hand, using~\eqref{e.c72902}
\be\label{e.c72901}
	v(x_0)
		= \int_{-\infty}^{x_0} \frac{e^{-|x_0 - y|/\sqrt\nu}}{2\sqrt\nu} u(y) dy
		= -\sqrt\nu \int_{-\infty}^{x_0} \Big(\frac{e^{-|x_0 - y|/\sqrt\nu}}{2\sqrt\nu}\Big)_x u(y) dy
		= - \sqrt \nu v_x(x_0)
		= \sqrt \nu c.
\ee
Putting together~\eqref{e.c72902} and~\eqref{e.c72901}, we find
\be
	c \geq \frac{\sqrt\nu}{2\nu+1}.
\ee 
This concludes the proof in the case where $\cZ = \{x_0\}$.

\smallskip
\noindent
{\bf Case two: $\cZ$ is empty.}
We now consider the case when $\cZ = \emptyset$.  In this case,
\be\label{e.c72905}
	\sup(-v_x) \leq c.
\ee
Indeed, $-v_x \neq c$ on $\R$ since $\cZ =\emptyset$ and, since $v(-\infty) = 1$ and $v(+\infty) = 0$, it cannot be that $-v_x > c$ on $\R$.

If $c > 1$, the proof is finished.  Hence, we consider only the case where $c < 1$.   We first note that, up to translation, the normalization~\eqref{e.c8027} in conjunction with the monotonicity of $u$ (recall Lemma \ref{lem:monotonicity.hyp}) yield
\be\label{e.c80204}
	u(x) \geq \frac{\nu}{\nu+ 1} \quad\text{ when } x < 0
	\qquad\text{ and }\qquad
	u(x) \leq \frac{\nu}{\nu+ 1} \quad\text{ when } x > 0.
\ee
Next, let $C$ and $\theta$ be as in~\eqref{e.exponential_decay} and take
\be\label{e.c72904}
	L= \max\left\{0, \frac{1}{\theta} \log\left( \frac{2C}{1+\theta\sqrt{\nu}}\right)\right\}.
\ee
Notice that $C$ and $\theta$ depend only on $\nu$ as we are working under the assumption that $c<1$.  Hence $L$ depends only on $\nu$.

Using the expression~\eqref{e.v'}, followed by the exponential decay~\eqref{e.exponential_decay} and ~\eqref{e.c80204}, yields,
\begin{align}
	-\sqrt \nu v_x(0)
	&= \int_{-\infty}^0 \frac{e^{-\frac{|y|}{\sqrt \nu}}}{2 \sqrt \nu} u(y) dy
		- \int_0^\infty \frac{e^{-\frac{|y|}{\sqrt \nu}}}{2 \sqrt \nu} u(y)  dy
	\\&\geq
		 \frac{\nu}{\nu+ {1}}\int_{-\infty}^0 \frac{e^{-\frac{|y|}{\sqrt \nu}}}{2 \sqrt \nu} dy
		- \frac{\nu}{\nu+ {1}}\int_0^L \frac{e^{-\frac{|y|}{\sqrt \nu}}}{2 \sqrt \nu} dy
		- \frac{C\nu}{\nu+ {1}} \int_L^\infty \frac{e^{-\frac{|y|}{\sqrt \nu}}}{2 \sqrt \nu} e^{- \theta y} dy.
\end{align}
Carrying out the integration and simplifying yields,
\begin{align}
	- \sqrt \nu v_x(0)
	&\geq
	\frac{\nu e^{\frac{-L}{\sqrt{\nu}}}}{2(\nu+1)}
	- \frac{C\nu}{\nu+1} \frac{1}{2(1 + \theta\sqrt\nu)} e^{- \left(\frac{1}{\sqrt \nu} + \theta\right) L} \\ 
	& = \frac{\nu}{\nu+1} e^{-\frac{L}{\sqrt \nu}} 	
	\left(\frac{1}{2} - \frac{C}{2(1+\theta\sqrt{\nu})}e^{-\theta L}\right).
\end{align}
Our choice of $L$~\eqref{e.c72904} then implies that
\be
	- v_x(0)
		\geq \frac{\nu}{4 (\nu+1)} e^{-L/\sqrt\nu}
		= \frac{\nu}{4(\nu + 1)} \min\left\{1, \Big(\frac{1+\theta\sqrt\nu}{2 C} \Big)^{\theta}{\sqrt\nu}\right\}.
\ee
Recalling~\eqref{e.c72905} finishes the proof of case two.
\end{proof}

\section{The hyperbolic scaling regime: proof of \Cref{t.lower.hyp}}
\label{s.lower.hyp}

With the general results of the previous section in hand, we may now prove \Cref{t.lower.hyp}; that is, for any sequence $(c_n, u_n, v_n)$ solving~\eqref{e.tw} with $\chi_n \to -\infty$ and $\nu_n\to \nu > 0$, we have that
\be
	\liminf_{n\to\infty} c_n \geq c_{\hypnu}^*.
\ee
Notice that the normalization~\eqref{e.normalization.hyp}, along with \Cref{lem:Chris monot} and the fact that $u_n(-\infty) = 1$ and $u_n(+\infty) =0$ yields
\be\label{e.normalization2}
	\inf_{x\leq 0} u_n(x)
		= \delta
		= \sup_{ x\geq 0}  u_n(x).
\ee
Moreover, by \Cref{lem:exp decay}, there are $C, \theta >0$, independent of $n$, such that
\be
	u_n(x) \leq C e^{-\theta x}
		\qquad\text{ for all } x > 0.
\ee
We use these inequalities in order to guarantee the nontriviality of a limit.

We now begin the proof.

\begin{proof}[Proof of \Cref{t.lower.hyp}]
As in the proof of \Cref{t.lower.pm}, it suffices to prove \Cref{t.lower.hyp}.(ii) as the claim (i) follows by simply taking a subsequence $c_{n_k}$ that realizes the limit inferior.  Up to the extraction of a subsequence, $u_n \to u$ and $v_n \to v$ weak-$*$ in $L^\infty$ and $W^{2,\infty}$, respectively, and $c_n \to c \geq 0$.    
It immediately follows $v$ that solves the second line of \eqref{e.hyp} weakly.

We have the following two results, whose proofs we postpone until \Cref{ss.lemmas.hyp}.  The first is regarding the smoothness of $u$ on $\cZ^c$:
\begin{lemma}\label{l.hyp_n_bounds}
In the setting above, for every $k \in \N$, $x_0 \in \cZ^c$, and $\mu>0$ such that $(x_0- \mu, x_0 + \mu) \subset \cZ^c$, there is $n_0$ and $C$ such that
\be
	\|u_n\|_{C^k(x_0-\sfrac{\mu}{2},x_0+\sfrac{\mu}{2})} \leq C
		\qquad\text{ for all } n \geq n_0.
\ee
The constant $C$ depends only on $k$, $x_0$, $\mu$, $c$, and $v$.  The constant $n_0$ also depends on these parameters and, additionally, the convergence rate of $|c_n - c|$ and $\|v_n - v\|_{W^{1,\infty}}$.
\end{lemma}
\begin{lemma}\label{l.Z_discrete}
In the setting above, $(c,u,v)$ satisfy~\eqref{e.c81501}. 
\end{lemma}

With \Cref{l.hyp_n_bounds} and \Cref{l.Z_discrete}, we see that $u_n$ converges to $u$ in $C^1$ on $\cZ^c$.  It follows that $(c,u,v)$ is a solution of~\eqref{e.hyp} in the sense of \Cref{def:hyp}.

Next, we check the boundary conditions $u(-\infty) = 1$ and $u(+\infty) = 0$.  The latter follows directly since, by \Cref{lem:exp decay}, each $u_n$ satisfies
\be
	u_n(x) \leq C e^{-\theta x}
		\qquad\text{ for all } x \geq 0.
\ee

The limit on the left is more delicate.  First notice that by \Cref{l.hyp_n_bounds}, there is $\bar x$ such that $u$ is smooth on $(-\infty, \bar x)$.  Next notice that
\be
	\limsup_{x\to-\infty} u(x) \leq 1.
\ee
Hence, it is enough to show that $\ell \geq 1$, where
\be
	\ell = \liminf_{x\to-\infty} u(x).
\ee
This is now our aim.  We make note of two facts first:
\be\label{e.c80804}
	\liminf_{x\to-\infty} v(x) \geq \ell
		\quad\text{ and }\quad
	\ell \geq \delta > 0,
\ee
where the first inequality follows by~\eqref{e.kernel} and the second from~\eqref{e.normalization2}.  Since $u \leq 1$, either $u$ is eventually monotonic (that is, up to decreasing $\bar x$, $u$ is monotonic on $(-\infty,\bar x)$) or $u$ has a sequence of local minima on which its value tends to $\ell$.  In either case, we can choose $x_n$ to be a sequence of points tending to $-\infty$ such that
\be
	\lim_{n\to\infty} u(x_n) = \ell
	\quad\text{ and }\quad
	\lim_{n\to\infty} u_x(x_n) = 0.
\ee
Evaluating~\eqref{e.hyp2} at $x_n$, we find
\be
	0 = \lim_{n\to\infty} u_x(x_n) (-c - v_x(x_n))
		= \lim_{n\to\infty} u(x_n) \left( \frac{\nu + v(x_n)}{\nu} - \frac{\nu+1}{\nu} u(x_n)\right)
		\geq \ell \left( \frac{\nu + \ell}{\nu} - \frac{\nu + 1}{\nu} \ell\right).
\ee
By~\eqref{e.c80804}, it follows that
\be
	\frac{\nu + \ell}{\nu} \leq \frac{\nu + 1}{\nu} \ell
	\quad\text{ which is equivalent to }
	0 \leq \nu(\ell-1).
\ee
We conclude that $\ell = 1$, which, due to the discussion above, yields $u(-\infty) = 1$, as claimed.

Thus, $(c,u,v)$ is a traveling wave solution to~\eqref{e.hyp}, which concludes the proof.
\end{proof}

\ssection{Proof of the first main lemma: \Cref{l.hyp_n_bounds}}\label{ss.lemmas.hyp}

\begin{proof}
We show only the $C^1$ regularity as the higher regularity may be established by differentiating the equation and apply the same argument.  Due to the convergence of $v_n$ to $v$, there is $\delta>0$ and $n_0$ such that if $n\geq n_0$ then
\be
	|(v_n)_x + c|
		\geq \delta
		\qquad\text{ in }
		(x_0 - \sfrac{3\mu}{4}, x_0 + \sfrac{3\mu}{4}).
\ee
Using then~\eqref{e.tw2}, we find, on $(x_0 - \sfrac{3\mu}{4}, x_0 + \sfrac{3\mu}{4})$,
\be\label{e.c80503}
	\begin{split}
		|(u_n)_x| &= \frac{1}{|(v_n)_x + c_n|} \left|\frac{1}{|\chi|} (u_n)_{xx} + u_n\left(\frac{\nu + v_n}{\nu} - \left(\frac{\nu + 1}{\nu}\right) u_n\right)\right|
		\\&
			\leq \frac{1}{\delta} \left|\frac{1}{|\chi|} (u_n)_{xx} + u_n\left(\frac{\nu + v_n}{\nu} - \left(\frac{\nu + 1}{\nu}\right) u_n\right)\right|.
	\end{split}
\ee

By the mean value theorem, there is $\xi_\ell\in (x_0 - \sfrac{3\mu}{4}, x_0 - \sfrac{\mu}{2})$ and $\xi_r \in (x_0 +\sfrac{\mu}{2}, x_0 + \sfrac{3\mu}{4})$ such that
\be\label{e.c80502}
	\begin{split}
	&|(u_n)_x(\xi_\ell)| = \left|\frac{u_n(x_0 - \sfrac{\mu}{2}) - u_n(x_0 - \sfrac{3\mu}{4})}{\sfrac{\mu}{4}}\right|
		\leq \frac{4}{\mu}
	\quad\text{ and}
	\\&
	|(u_n)_x(\xi_r)| = \left|\frac{u_n(x_0 + \sfrac{3\mu}{4})-u_n(x_0 + \sfrac{\mu}{2}) }{\sfrac{\mu}{4}}\right|
		\leq \frac{4}{\mu}.
	\end{split}
\ee
The second inequality follows from the fact that $0 \leq u_n \leq 1$.

By elliptic regularity theory, $u_n$ is smooth.  Hence, by the extreme value theorem, there is $\xi \in [\xi_\ell, \xi_r]$ such that
\be\label{e.c80501}
	|(u_n)_x(\xi)|
		= \max_{[\xi_\ell,\xi_r]} |(u_n)_x|
		\geq\max_{[x_0 - \sfrac{\mu}{2},x_0 + \sfrac{\mu}{2}]} |(u_n)_x|.
\ee
The equality above follows from the definition of a maximum, and the inequality is because $[x_0 - \sfrac{\mu}{2}, x_0 + \sfrac{\mu}{2}] \subset [\xi_\ell, \xi_r]$, by construction. 
If $\xi = \xi_\ell$ or $\xi_r$, then the conclusion follows from~\eqref{e.c80502} and~\eqref{e.c80501}.  On the other hand, if $\xi$ is an interior minimum, then $(u_n)_{xx}(\xi) = 0$ and~\eqref{e.c80503} yields
\be
	|(u_n)_x(\xi)|
		\leq \frac{1}{\delta} \left|u_n\left(\frac{\nu + v_n}{\nu} - \left(\frac{\nu + 1}{\nu}\right) u_n\right)\right|
		\leq \frac{1}{\delta} \frac{\nu + 1}{\nu}.
\ee
In view of~\eqref{e.c80501}, the proof is finished.
\end{proof}

\ssection{Proof of the second main lemma: \Cref{l.Z_discrete}}

\begin{proof}
	Consider any closed interval $[\underline x, \bar x] \subset \cZ$ with $\underline{x} < \bar{x}$.   
We note that, by the definition of $\cZ$, it must be that $v_x = -c$ on $[\underline x,\bar x]$.   
We have two important consequences from this:
\be\label{e.c80702}
	v_{xx} = 0 \quad\text{ on } [\underline x,\bar x],
\ee
and
\be\label{e.c80703}
	v \text{ is nonincreasing on } [\underline x,\bar x].
\ee

We claim that, up to extracting a subsequence,
\be\label{e.c80601}
	u_n \to v
		\quad \text{ as } n \to \infty \text{ in } L^2([\underline x,\bar x]).
\ee
We postpone its proof momentarily and show how to conclude.

Fix any smooth function $\psi$ with support in $(\underline x,\bar x)$, multiply it against~\eqref{e.tw2}, and take $n\to\infty$ to find
\be
	\begin{split}
	\int \psi_x (v_x + c) u \dx
		&+ \int \psi v_{xx} u \dx
		= \lim_{n\to\infty} \left(\int \psi_x \left((v_n)_x + c_n\right) u_n \dx
		+ \int \psi (v_n)_{xx} u_n \dx\right)
		\\&
		= - \lim_{n\to\infty} \int \psi \left((v_n)_x + c_n\right) (u_n)_x \dx
		= \lim_{n\to\infty} \int \psi \left(\frac{\nu_n + v_n}{\nu_n} - \left(\frac{\nu_n+1}{\nu_n} \right)u_n\right)u_n \dx
		\\&
		= \int \psi \left(\frac{\nu + v}{\nu} - \left(\frac{\nu+1}{\nu}\right) v\right)v \dx.
	\end{split}
\ee
Above we used that $(v_n)_{xx} \rightharpoonup v_{xx}$ in $L^2([\underline x,\bar x])$ and that $v_n \to v$ in $L^\infty([\underline x,\bar x])$.
Recalling that $v_x + c = 0$ on $[\underline x,\bar x]$, yields
\be
	0 = \int \psi \Big(\frac{\nu + v}{\nu} - \Big(\frac{\nu+1}{\nu} \Big) v\Big)v\dx,
\ee
which concludes the proof.

We now prove~\eqref{e.c80601}, which is the most difficult part.  We first note that $u_n$ converges weakly to $v$ in $L^2([\underline x,\bar x])$.  Using the second equation in~\eqref{e.tw}, we have that the weak limit $u$ of $u_n$ clearly satisfies
\be
	u
		= v - \nu v_{xx}.
\ee
Using this and~\eqref{e.c80702}, we deduce that $u =v$; hence,
\be\label{e.c80704}
	u_n \rightharpoonup  v
	\quad\text{ in }
	L^2([\underline x,\bar x]).
\ee

Fix $\mu, \delta>0$.  Our next step is to show that, for all $n$ sufficiently large and up to extracting a subsequence,
\be\label{e.c80701}
	u_n \geq v_n - \mu
		\quad\text{ on }
		(\underline x + \delta, \bar x-\delta).
\ee
Before showing~\eqref{e.c80701}, we show how to conclude~\eqref{e.c80601} from it.  Noting~\eqref{e.0<u<1} and using that $u_n, v_n \rightharpoonup v$, we have that
\be
	\begin{split}
	\limsup_{n\to\infty} &\int_{\underline x}^{\bar x} (u_n - v_n)^2 \dx
		\\&\leq 2\delta + \limsup_{n\to\infty}\int_{\underline x + \delta}^{\bar x - \delta} |u_n - v_n| \dx
		\leq 2\delta + \mu |\bar x - \underline x| + \limsup_{n\to\infty}\int_{\underline x + \delta}^{\bar x - \delta}  |u_n - v_n + \mu| \dx
		\\&= 2\delta + \mu |\bar x- \underline x| + \limsup_{n\to\infty}\int_{\underline{x} + \delta}^{\bar{x} - \delta}  (u_n - v_n + \mu) \dx
		\\& \leq 2\delta + 2\mu |\bar x-\underline{x}| + \limsup_{n\to\infty} \int \1_{[\underline  x+\delta, \bar x - \delta]} u_n \dx
		- \liminf_{n\to\infty} \int \1_{[\underline x +\delta, \bar x - \delta]} v_n \dx
		\\&
		= 2\delta + 2\mu |\bar x - \underline x| + \int \1_{[\underline x +\delta, \bar x - \delta]} v \dx
		-  \int \1_{[\underline x +\delta, \bar x - \delta]} v \dx
		= 2 \delta + 2\mu |\bar x - \underline x|.
	\end{split}
\ee
Note that~\eqref{e.c80701} was used in the first equality.  Clearly,~\eqref{e.c80601} follows after taking $\delta,\mu\to 0$.

We now show that~\eqref{e.c80701} holds.  As the proof is quite intricate, we briefly outline the main idea here. We first show that, were~\eqref{e.c80701} to be false, the weak convergence of $u_n$ and $v_n$ to $v$ implies the existence of points $y_n$ and $z_n$ between which $u_n - v_n$ travels from being bigger than $-\mu/2$ to being $-\mu$.  Choosing a `good' point between them, we can use the fact that $v_n$ is `nearly' decreasing (see~\eqref{e.c80706}) along with the partial monotonicity result \Cref{l.no_local_min} to find a nontrivial interval where $u_n - v_n$ remains smaller than $-\mu/2$, which is not consistent with the weak convergence of $u_n$ and $v_n$ to $v$.

First, by the weak convergence of $u_n$ to $v$, we have that there exists $y_n \in (\underline x,\underline x + \delta)$ such that
 \be\label{e.c8112}
 	u_n(y_n) \geq v_n(y_n) - \frac{\mu}{2}.
\ee
If this were not true,
\be
	\int \psi(v_n - u_n)
		\geq \frac{\mu}{2} \int \psi\dd x,
\ee
for any nonnegative, nontrivial test function $\psi$ supported in $(\underline x, \underline x + \delta)$.  This contradicts the fact that $u_n \rightharpoonup v$ and $v_n \to v$.

Next, let $z_n$ to be the first time in $(y_n, x_n - \delta)$ that
\be\label{e.c8111}
	u_n(z_n) = v_n(z_n) - \mu
\ee
If no such point exists, then~\eqref{e.c80701} holds and we are finished.  Hence, arguing by way of contradiction, we assume $z_n$ exists.  Notice that, by the $C^1$-convergence of $v_n$ to $v$ and~\eqref{e.c80703}, there is $n$ sufficiently large such that
\be\label{e.c80706}
	(v_n)_x \leq  \frac{\mu}{100(1+ \bar x - \underline x)} \quad\text{ on } [\underline x,  \bar x].
\ee
As a consequence, using the concavity of $v$, we have
\be
	u_n(z_n)
		= v_n(z_n) - \mu
		\leq (v_n(y_n) + (v_n)_x(y_n) (z_n - y_n)) - \mu
		< v_n(y_n) + \frac{\mu}{100} - \mu
		< u_n(y_n),
\ee
where the first inequality follows by~\eqref{e.c80706} and
 the second inequality is due to the choice of $y_n$~\eqref{e.c8112}.

We next claim that there exists a point $\zeta_n \in [y_n,z_n]$ such that
\be\label{e.c80707}
	u_n(\zeta_n) \leq v_n(\zeta_n) - \frac{3\mu}{4}
	\quad\text{ and } \quad
	(u_n)_x(\zeta_n) < 0.
\ee
We construct this as follows. Let
\be
	\zeta_n = \min \left\{ x \in [y_n,z_n) : u_n(x) = v_n(x) - \frac{3\mu}{4} - \frac{\mu (x - y_n)}{50(1 + \bar x - \underline x)}\right\}.
\ee  
That $\zeta_n$ exists follows from the definition of $z_n$~\eqref{e.c8111} and that $\zeta_n > y_n$ follows from the definition of $y_n$~\eqref{e.c8112}.  Since $u_n - v_n > \sfrac{-3\mu}{4} - \sfrac{\mu (x-y_n)}{50(1+\bar x -\underline x)}$ in $(y_n, \zeta_n)$, it must be that
\be
	(u_n)_x (\zeta_n) \leq  \left(v_n - \frac{3\mu}{4} - \frac{\mu (x-y_n)}{50(1+\bar x - \underline x)}\right)_x(\zeta_n).
\ee
After applying~\eqref{e.c80706}, we conclude~\eqref{e.c80707}.  

We claim that
\be\label{e.c80708}
	u_n(y) \leq v_n(y) - \frac{\mu}{2}
		\qquad\text{ for all } y \in (\zeta_n, \bar x).
\ee
Before proving this, we show how to conclude the proof by contradiction that started with the existence of $z_n$.  Indeed, up to taking a further subsequence, we have that $\zeta_n \to \zeta_\infty$, for some $\zeta_\infty$.  Then, take any nonnegative, nontrivial test function $\psi$ supported on a compact subset of $(\zeta_\infty, \bar x)$.  Note that the fact that $\zeta_n < z_n \leq \bar x - \delta$ ensures that the support of $\psi$ is nontrivial.  By~\eqref{e.c80708}, we observe that
\be
	\liminf_{n\to\infty} \int \psi (v_n - u_n)
		\geq \frac{\mu}{2} \int \psi \dx
		> 0.
\ee
This contradicts the weak convergence of $u_n$ and $v_n$ to $v$.

We now show that~\eqref{e.c80708} holds.  To establish this, let $y_{\min}$ be the first local minimum of $u_n$ after $\zeta_n$.  We first consider the case where $y_{\min}$ does not exist or where
\be
	y_{\min} \geq \bar x.
\ee
In this case, from to the lack of an interior minimum and the fact that $(u_n)_x(\zeta_n) < 0$ (by construction, see~\eqref{e.c80707}), it follows that $u_n$ is decreasing on $(\zeta_n, \bar x)$.  In this case, for any $y \in (\zeta_n, \bar x)$,
\be\label{e.c80709}
	u_n(y)
		\leq u_n(\zeta_n)
		\leq v_n(\zeta_n) - \frac{3\mu}{4}
		\leq v_n(y) + \frac{\mu}{100(1 + \bar x - \underline x)} (y - \zeta_n) - \frac{3\mu}{4}
		< v_n(y) - \frac{\mu}{2}.
\ee
The first inequality is because $u_n$ is decreasing, the second inequality is due to~\eqref{e.c80707}, the third inequality is due to~\eqref{e.c80706}, and the last inequality is due to the width of the domain.  This establishes~\eqref{e.c80708}.

We next consider the case where $y_{\min} \in (\zeta_n, \bar x)$.  The same reasoning as in~\eqref{e.c80709} yields
\be
	u_n(y_{\min})
		< v_n(y_{\min}) - \frac{\mu}{2}.
\ee
\Cref{l.no_local_min}, however, implies that
\be
	u_n(y_{\min})
		\geq \frac{\nu + v_n(y_{\min})}{\nu + 1}
		\geq v_n(y_{\min}).
\ee
Thus, we have reached a contradiction, and hence, this case cannot occur.  This concludes the proof of~\eqref{e.c80708}.  As detailed above, this yields~\eqref{e.c80701}, which, in turn, completes the proof of the lemma.
\end{proof}

\section{The partially matching upper bounds: proof of Theorems~\ref{t.upper.hyp} and \ref{t.upper.pm}}
\label{s.upper}

Here we construct special solutions that, in a sense, saturate the bounds obtained in \Cref{t.lower.pm} and \Cref{t.lower.hyp}.  We begin by constructing a traveling wave solution to~\eqref{e.hyp}.

\ssection{A sequence converging to the discontinuous hyperbolic wave: \Cref{t.upper.hyp}}\label{s.upper.hyp}

Here we establish the existence of a sequence of solutions to \eqref{e.tw} that converges to a discontinuous hyperbolic traveling wave.  
Our construction proceeds in the usual way, by approximating traveling waves solutions to \eqref{e.tw} by solutions to a well-chosen Dirichlet problem in a finite slab.  The main novelty in our setting is to construct a solutions such that $v_x(0) + c \approx 0$ as $|\chi|\to\infty$.  In view of \Cref{prop:cardZleq1} and \Cref{l.Z_discrete}, this is enough to deduce that, in the limit, there is a jump discontinuity at the origin.  
The convergence of the sequence thus yield a new method to construct traveling waves for the hyperbolic problem \eqref{e.hyp}, different from the original one provided in \cite{FuGrietteMagal-M3AS} and that does not require the same technical assumption.

\sssection{The main proposition and the proof of \Cref{t.upper.hyp}}

The main proposition used in establishing \Cref{t.upper.hyp} is the following:

\begin{proposition}\label{prop:decreasing-wave}
	Fix any $\chi<0$, $\nu>0$, and $\delta\in (0, \sfrac{\nu}{(\nu+1)})$.  There exists a traveling wave $(c, u, v)$ solving \eqref{e.tw} such that $u$ and $v$ are decreasing and $c, u$ satisfy,
	\be
		c\in \left(\frac{1}{\sqrt{|\chi|}}\sqrt{\frac{\nu}{\nu+1}}, \frac{1}{\sqrt{\nu}}+\frac{2}{\sqrt{|\chi|}}\sqrt{\frac{\nu + 1}{\nu}}\right),
		\quad
		u(0) = \delta,
		\quad
		u(-\infty)=1,
		\quad
		\text{and}
		\quad
		u(+\infty) = 0.  
	\ee
	Moreover, for $|\chi|$ sufficiently large, there is a constant $C>0$, depending on $\nu$ and $\delta$ only, and a point $x_\chi \in (-C,C)$ such that
	\be\label{eq:T5.1-Znonempty}
		(v_x(x_\chi)+c)^2 \leq \frac{4}{|\chi|} \frac{\nu + 1}{\nu}.
	\ee
\end{proposition}

Before establishing \Cref{prop:decreasing-wave}, we show how to use it to prove \Cref{t.upper.hyp}.

\begin{proof}[Proof of \Cref{t.upper.hyp}]
Fix any $\nu_{\hypnu} >0$ and take any sequence $\chi_n \to -\infty$, $\nu_n\rightarrow \nu_{\hypnu}$.   Let $(c_n,u_n,v_n)$ be the solution of~\eqref{e.tw} guaranteed by \Cref{prop:decreasing-wave}.  By \Cref{t.lower.hyp}.(ii), there is $(c,u,v)$ that is a solution of~\eqref{e.hyp} in the sense of \Cref{def:hyp} to which $(c_n,u_n,v_n)$ converges along a subsequence.  
This additionally gives that
\be
	u(-\infty) = 1
	\quad\text{ and }\quad
	u(+\infty) = 0.
\ee
The monotonicity of $u$ and $v$ follows directly from that of $u_n$ and $v_n$.

Let $x_n \in (-C, C)$ be the point such that
\be\label{e.c80910}
	((v_n)_x(x_n) + c_n)^2 \leq \frac{4}{|\chi_n|} \frac{\nu_n + 1}{\nu_n}.
\ee
	Up to passing to a further subsequence, there exists $x_* \in [-C,C]$ such that $x_n \to x_*$ as $n\to\infty$.  The convergence of $v_n$ is weak-$*$ in $W^{2,\infty}$, so that~\eqref{e.c80910} implies, 
\be
	v_x(x_*) + c = 0.
\ee
This, along with \Cref{prop:cardZleq1} and \Cref{l.Z_discrete}, implies that $\cZ = \{0\}$ and that
\be
	u(x_*^-) = \frac{\nu + v(x_*)}{\nu + 1} > 0
	\quad\text{ and }\quad
	u(x_*^+) = 0.
\ee
This concludes the proof, up to translating by $x_*$.
\end{proof}

\sssection{Constructing a solution: the proof of \Cref{prop:decreasing-wave}}

The main step in constructing the traveling wave of \Cref{prop:decreasing-wave} is to solve the `slab problem':
\begin{equation}\label{eq:slab-nupos}
	\begin{cases} 
		-c u_x - \frac{1}{|\chi|} u_{xx} = (uv_x)_x +u(1-u), & x\in (-L, L), \\
		u(-L)=1, u(L)=0,
	\end{cases}
\end{equation}
where
\be
\label{e.240943}
	v = \mathcal{K}_{\nu}* \bar{u}
	\qquad\text{ where }
	\bar u(x)
		= \begin{cases}
			1 \quad&\text{ for }x < -L,\\
			u(x) \quad&\text{ for } x \in [-L,L],\\
			0 \quad&\text{ for } x > L.
		\end{cases}
\ee
This is achieved through a number of steps and relies on the Leray-Schauder index.  

To this end, for fixed $\chi<0$, $\nu>0$, $c>0$, and $\tau\in [0,1]$, we define the 
  operator $\mathcal{F}_\tau$ acting on $u\in C^1([-L, L])$ by
	\be
		\mathcal{F}_\tau(u):=\tilde u(x), 
	\ee
	where $\tilde u(x)$ is the unique solution to the equation
	\be
	\label{eq:topologicaldegree-defFtau}
		\begin{cases}
			-\frac{1}{|\chi|}\tilde u_{xx}(x)-c\tilde u_x(x) +\tilde u(x)=\tau (u(x)v_x(x))_x + u(x)\big(2-u(x)\big), &\text{ for all } x\in (-L, L) \\
			\tilde u(-L) = 1, \tilde u(L)=0,
		\end{cases}
	\ee
for $v$ given by~\eqref{e.240943}. For any bounded open set $\Omega\subset  C^1([-L, L])$, we use $i(\mathcal{F}_\tau, \Omega)$ to denote the Leray-Schauder index of $\mathcal{F}_\tau$ acting on $\Omega$. We refer the reader to \cite[Chapter 14]{Zeidler} for the definition of Leray-Schauder index. 

We remark that fixed points $u$ of $\mathcal{F}_\tau$ correspond exactly to solutions of,
	\begin{equation}\label{eq:slab-nupos-tau}
		\begin{cases} 
			-c u_x - \frac{1}{|\chi|} u_{xx} = \tau (uv_x)_x +u(1-u), & x\in (-L, L), \\
			u(-L)=1, u(L)=0,
		\end{cases}
	\end{equation}
	where $v$ is given by ~\eqref{e.240943}. For $\tau=0$ we recover a classical slab problem for FKPP waves.  For $\tau=1$ we have our target problem \eqref{eq:slab-nupos}.

\begin{lemma}[Index in the slab]\label{lem:index=1}
	Let $0<\chi_m<\chi_M$, $0<\nu_m<\nu_M$, $ 0<c_m<c_M$ and $\alpha\in(0, 1)$ be fixed. Fix $|\chi|\in [\chi_m, \chi_M]$, $\nu\in[\nu_m, \nu_M]$,  and $c\in[c_m, c_M]$.	There exists $M_0>0$ such that  for all $M\geq M_0$, the Leray-Schauder index of $\mathcal{F}_\tau$ in the bounded open set 
	\begin{equation}\label{eq:topologicaldegree-Omega}
		\Omega:=\left\{u\in C^{1, \alpha}\big([-L, L]\big)\,:\,\begin{cases} 0< u(x)< 1,& \text{ for all } x\in(-L, L),\\
			-M< u_x(x)<0,&\text{ for all } x\in[-L, L],
		\end{cases}\right\} \subset C^{1, \alpha}([-L, L])
	\end{equation}
	is equal to $1$ for $\tau\in [0,1]$:
	\be
	\label{eq:topologicaldegree1}
		i(\mathcal{F}_\tau, \Omega)=1 \text{  for $\tau\in [0,1]$}. 
	\ee
	The constant $M_0$ only depends on $\chi_m$, $\chi_M$, $\nu_m$, $\nu_M$, $c_m$, and $c_M$.
\end{lemma}
\begin{proof}
First, we note that $v$ solves the equation 
	\be
		-\nu v_{xx}(x)=u(x)-v(x) \text{ pointwise  for all } x\in (-L, L), 
	\ee
	so that $v$ is bounded in $C^{2, \alpha}\big((-L, L)\big)$ and continuous as a function of $u\in C^{0, \alpha}([-L, L])$ in the same space; if $u\in C^{1, \alpha}\big((-L, L)\big) $ then $v\in C^{3, \alpha}\big((-L, L)\big)$ and depends continuously on $u$.
	Clearly $\mathcal{F}_\tau$ maps $C^{1, \alpha}([-L, L])$ into $ C^{1, \alpha}([-L, L])$ and is continuous with respect to the parameter $\tau$ for the $C^{1, \alpha}$ norm.
	Moreover by the Schauder estimates $\mathcal{F}_\tau$ is compact for the $C^{1, \alpha}([-L, L])$ topology.

	Next we show that the Leray-Schauder index of $\mathcal{F}_\tau $ (see \cite[Proposition 14.5]{Zeidler}) is non-zero in the open set $\Omega$
	for a sufficiently large constant $M$.
	The proof consists of two main parts: 
	first, we show that the Leray-Schauder index of $\mathcal{F}_0$ is $1$, and, second, we establish that $\mathcal{F}_\tau$ has no fixed point on the boundary of $\Omega$ for $\tau\in (0,1]$.  The consequence of these two facts and the homotopy invariance property of the Leray-Schauder index is that~\eqref{eq:topologicaldegree1} holds.

\medskip
	\noindent\textbf{Step one: the Leray-Schauder index of $\mathcal{F}_0$ is $1$.}\\
	For  $\tau=0$, the equation~\eqref{eq:slab-nupos-tau} is the classical FKPP traveling wave equation on the slab $[-L, L]$. Hence, this step is essentially `folklore'; however, being unable to find a published proof, we include it here.  First, we note that it is known that this equation has a unique solution $u^0$ which is strictly decreasing and satisfies $\sup_{x\in (-L, L)}|u^0_x(x)|<+\infty$ (and hence lies in $\Omega$ upon taking $M$ large enough): the uniqueness and monotonicity follow from `sliding' arguments~\cite{BerestyckiNirenberg}, while the bound on $u^0_x$ follows from elliptic regularity theory. Since fixed points of $\mathcal{F}_\tau$ correspond exactly to solutions of~\eqref{eq:slab-nupos-tau}, we deduce that $\mathcal{F}_0$ has a unique fixed point $u^0\in\Omega$ and $u^0$ is strictly decreasing.

	Next we compute $i(\mathcal{F}_0, \Omega)$, the Leray-Schauder index of $\mathcal{F}_0$ in $\Omega$. Since $ u^0(x)$  is the unique fixed-point of $\mathcal{F}_0$ in $\Omega$, we have by \cite[Proposition 14.5]{Zeidler}:
	\begin{equation}\label{eq:topologicaldegree}
		i(\mathcal{F}_0, \Omega):=i(D\mathcal{F}_0(u^0), B(0, 1)) = (-1)^a, 
	\end{equation}
	where $D\mathcal{F}_0$ is the Fr\'echet derivative of $\mathcal{F}_0$ in $C^1([-L, L])$ and $a$ is the sum of algebraic multiplicities of all eigenvalues of $D\mathcal{F}_0(u^0)$ that are greater than $1$. This formula, however, is conditional on the fact that 1 is not an eigenvalue of $D\mathcal{F}_0(u^0)$, which we prove now.  
	The Fr\'echet derivative $D\mathcal{F}_0$ can be computed as $D\mathcal{F}_0(u) h=\tilde h$ where
	\begin{equation}\label{eq:topologicaldegree-Frechet1}
		\begin{cases}
			-\frac{1}{|\chi|}\tilde h_{xx}(x)-c\tilde h_x(x)+\tilde h(x) = h(x)\big(2-2u(x)\big),  & \text{ for all } x\in (-L, L), \\ 
			\tilde h(-L)=\tilde h(L)=0.
		\end{cases}
	\end{equation}
	Notice that the coefficient in the right-hand side of \eqref{eq:topologicaldegree-Frechet1} is a nonnegative function, therefore $D\mathcal{F}_0(u^0)$ is order-preserving.
	
	We now show that $D\mathcal{F}_0(u^0)$ does not have an eigenvalue larger than one.  We argue by contradiction, in which case, $D\mathcal{F}_0(u^0)$ has principal eigenvalue $\lambda \geq 1$.  Then there must exist $h(x)>0$ such that $D\mathcal{F}_0(u^0)h=\lambda h$; that is,
	\be\label{e.c82501}
		\begin{cases}
			-c h_x-\frac{1}{|\chi|} h_{xx} + h= h\dfrac{\big(2-2u^0\big)}{\lambda}\leq  h\big(2-2u^0\big),  &\quad \text{ in } (-L, L), \\ 
			 h(-L)= h(L)=0.
		\end{cases}
	\ee
	However the function $k(x):=-u^0_x(x)>0$ satisfies 
	\be\label{e.c82502}
		\begin{cases}
			-c k_x-\frac{1}{|\chi|} k_{xx} + k=  k\big(2-2u^0\big),  &\quad \text{ in }  (-L, L), \\ 
			 k(-L)=-u^0_x(-L)>0,  k(L)=-u^0_x(L)>0,
		\end{cases}
	\ee
	where the sign of $u^0_x(-L)$ and $u^0_x(L)$ are known from Hopf's Lemma.  Due to the boundary conditions, we have that $A h < k$ for $A$ sufficiently small.  Letting
	\be
		A_0 = \sup\{ A : Ah < k \text{ on } (-L,L)\},
	\ee
	we see that $A_0 h \leq k$ and there exists $x_0 \in(-L,L)$ such that $A_0 h(x_0) = k(x_0)$.  If $\lambda>1$, it follows that
	\be
		A_0 h_x(x_0) = k(x_0)
			\quad\text{ and }\quad
		A_0 h_{xx}(x_0) \leq k_{xx}(x_0),
	\ee
	and that violates~\eqref{e.c82501} and~\eqref{e.c82502} (notice that $A_0 h$ also satisfies~\eqref{e.c82501}). If $\lambda=1$, the strong maximum principle implies that $A_0h-k=:C$ is a negative constant. Taking the difference between \eqref{e.c82501} and \eqref{e.c82502} we find that 
	\be 
	C = \big(2-2u^0(x)\big)C \Longleftrightarrow 1-2u^0(x)\equiv 0, 
	\ee
	which is a contradiction since $u(L)=0$.
	We conclude that $D\mathcal F_0(u^0)$ does not have any eigenvalues larger than $1$.  We, thus, conclude that $a=0$ in~\eqref{eq:topologicaldegree} and
	\begin{equation}\label{eq:topologicaldegree2}
		i(\mathcal{F}_0, \Omega)=1. 
	\end{equation}\medskip

	\noindent\textbf{Step two: there is no fixed point in the boundary of $\Omega$}.\\
	Suppose by contradiction that there exists $\tau\in(0,1]$ and a function $u\in\partial\Omega$ such that $\mathcal{F}_\tau(u)=u$. Since $u\in \partial\Omega$, at least one of the following equalities holds: 
	\begin{align}
		\label{eq:boundary-1} u(x) & = 0 & \text{ for some  } x\in (-L, L), \\ 
		\label{eq:boundary-2}u(x) &= 1 & \text{ for some  } x\in (-L, L), \\ 
		\label{eq:boundary-3}u_x(x)&=0  & \text{ for some  } x\in [-L, L], \\ 
		\label{eq:boundary-4}u_x(x)&=-M & \text{ for some  } x\in [-L, L].
	\end{align}
	We will show that none of those equalities can hold, which is a contradiction. 

First, we note that elliptic regularity theory implies that, by increasing $M$, we can ensure that~\eqref{eq:boundary-4} does not hold.  And, due to the boundary condition on $u$, we have that~\eqref{eq:boundary-1} and~\eqref{eq:boundary-2} cannot occur unless~\eqref{eq:boundary-3} occurs. Thus, we conclude that  there exists $x^*\in [-L, L]$ such that $u_x(x^*)=0$.

	We now use Hopf's Lemma to find,
	\be
		u_x(-L)<0 \text{ and } u_x(L)<0, 
	\ee
	so we obtain that, actually, $x^*\in (-L, L)$. Defining $w(x):=u_x(x)$ we have
	\begin{equation}\label{eq:topologicaldegree-aprioribound-1}
		\begin{cases}
			-c w_x(x) -\frac{1}{|\chi|} w_{xx}(x) = \tau (w_x(x) v_x(x) + 2w(x)v_x(x) + u(x)v_x(x))\\
				\phantom{-c w_x(x) -\frac{1}{|\chi|} w_{xx}(x) = } 
				+ w(x)(1-2u(x)), & \text{ if } x\in(-L, L), \\
			w(-L) = u_x(-L),  w(L) = u_x(L).
		\end{cases}
	\end{equation}
	Evaluating~\eqref{eq:topologicaldegree-aprioribound-1} at $x=x^*$ (which is a local maximum for $w$) we have
	\begin{equation}\label{eq:utaudecreasing-2}
		0\leq \tau u(x^*) v_x(x^*).
	\end{equation}
	However, since $u$ is decreasing and non-constant, we have $v_x(x^*)<0$, therefore \eqref{eq:utaudecreasing-2} is a contradiction. Therefore \eqref{eq:boundary-3} cannot hold. This concludes the proof of the claim of Step 2 and, therefore, the proof of the lemma as well.
\end{proof}

We can immediately deduce from Lemma \ref{lem:index=1} the existence of a continuum of waves in $\tau$ and $c$  for a  fixed $L>0$.
\begin{corollary}\label{cor:continuum-fixed-points}
	For each $0<\bar c<+\infty$, there exists a connected set $\mathcal C\subset [0, \bar c]\times C^{1, \alpha}([-L, L])$ such that (i) for each $(c,u)\in \mathcal C$ the function $u$ solves \eqref{eq:slab-nupos} with speed $c$, and (ii) there is a pair $(c,u) \in \mathcal{C}$ for each $c\in [0,\bar c]$.
\end{corollary}
\begin{proof}
    Let $\bar c>0$ be given. By Lemma \ref{lem:index=1} there exists a $M>0$ such that for $c=0$, the index of $\mathcal F_1$ (defined in \eqref{eq:topologicaldegree-defFtau}) in the set $\Omega$ defined in \eqref{eq:topologicaldegree-Omega} is equal to 1. Moreover it has been shown that there are no fixed-points of $\mathcal F_1$ with $u\in\partial\Omega$ and $\mathcal F_1$ is continuous with respect to $c$. By a direct application of the global continuation principle \cite[Theorem 14.C]{Zeidler}, there exists a continuum $\mathcal C$ composed with fixed-points of $\mathcal F_1$, connecting $\{0\}\times \Omega$ to $\{\bar c\}\times \Omega$. Since fixed points of $\mathcal{F}_\tau$ correspond exactly to solutions of~\eqref{eq:slab-nupos-tau}, Corollary \ref{cor:continuum-fixed-points} is proved. 
\end{proof}

\begin{lemma}\label{lem:sub-super-slab-0}
	Let $u(x)$ be a decreasing solution to \eqref{eq:slab-nupos}. Then $\underline{\varphi}(x)\leq u(x)\leq \overline{\varphi}(x)$, where $\underline{\varphi}(x)$ is the solution to the FKPP equation
	\begin{equation} \label{eq:slab-sub}
		\begin{cases} 
			-c\underline{\varphi}_x(x) -\frac{1}{|\chi|} \underline{\varphi}_{xx}(x) = \underline{\varphi}(x)\left(1-\left(\frac{\nu + 1}{\nu}\right)\underline{\varphi}(x)\right), & x\in (-L, L), \\
			\underline{\varphi}(-L)=\sfrac{\nu}{(\nu+1)}, ~~\underline{\varphi}(L)=0,
		\end{cases}
	\end{equation}
	and $\overline{\varphi}(x) $ is the solution to the FKPP equation  
	\begin{equation}\label{eq:slab-super}
		\begin{cases} 
			-\left(c-\frac{1}{\sqrt{\nu}}\right)\overline{\varphi}_x(x) -\frac{1}{|\chi|} \overline{\varphi}_{xx}(x) = \left(\frac{\nu + 1}{\nu}\right)\overline{\varphi}(x)\left(1-\overline{\varphi}(x)\right), & x\in (-L, L), \\
			\overline{\varphi}(-L)=1,~~ \overline{\varphi}(L)=0.
		\end{cases}
	\end{equation}
	In particular if $\varepsilon\in (0, 1)$ and  $L$ is sufficiently large, we have 
	\begin{equation}\label{eq:slab-sub-0}
		u(0)\geq (1-\varepsilon)\, \frac{\nu}{1+\nu} ,
			\qquad \text{ if } c< \frac{2}{\sqrt{|\chi|}} \sqrt{\frac{\nu}{\nu+1}},
	\end{equation}
	and
	\begin{equation}\label{eq:slab-super-0}
		u(0)\leq e^{-\frac{L|\chi|}{2}\left(c-\frac{1}{\sqrt{\nu}}+\sqrt{\left(c-\frac{1}{\sqrt{\nu}}\right)^2-4\frac{1}{|\chi|}\left(\frac{\nu + 1}{\nu}\right)}\right)},
			\qquad\text{ if } c-\frac{1}{\sqrt{\nu}}\geq \frac{2}{\sqrt{|\chi|}}\sqrt{\frac{\nu + 1}{\nu}}.
	\end{equation}

\end{lemma}
\begin{proof}
	Indeed $u(x)$ is a subsolution to \eqref{eq:slab-super} and a supersolution to \eqref{eq:slab-sub}. Since both \eqref{eq:slab-super} and \eqref{eq:slab-sub} satisfy the comparison principle, we have
	\be
	\label{e.08290201}
		\underline{\varphi}(x)\leq u(x)\leq \overline{\varphi}(x) \text{ for all } x\in (-L, L). 
	\ee

	To obtain \eqref{eq:slab-sub-0}, it is enough to show that $\underline{\varphi}(x) $ converges to $\frac{\nu}{\nu+1}$ locally uniformly as $L\to+\infty$ since $c$ is smaller than the minimal speed of the FKPP equation~\eqref{eq:slab-super} (when defined on $\R$), which is $\sfrac{2}{\sqrt{|\chi|}} \sqrt{\sfrac{\nu}{(\nu+1)}}$.
	We briefly outline why this is true.  First, it is simple to construct a subsolution to~\eqref{eq:slab-sub} of the form
	\be
		\eps_R e^{-\lambda_R x} \cos\left(\frac{\pi x}{2 R}\right)
			\quad\text{ on } (-R, R),
	\ee
	with $\eps_R$, $\lambda_R$, and $R$ chosen depending only on $\sfrac{2}{\sqrt{|\chi|}}\sqrt{\sfrac{\nu}{(\nu+1)}} - c$.  Hence, $\underline \phi(0) > \eps_R$ for all $L\geq R$.  On the other hand, after taking $L\to\infty$, $\underline \phi$ converges to some function $\underline \phi_\infty$ solving~\eqref{eq:slab-sub} on $\R$.  Since $c$ is smaller than the minimal speed, $\phi$ must be a trivial solution.  By above, $\underline\phi_\infty>0$, and, hence, can only be $\sfrac{\nu}{(\nu+1)}$.
		
	To obtain \eqref{eq:slab-super-0} we take $c$ satisfying the inequality in \eqref{eq:slab-super-0} and note that the function
	\be
		\overline{\overline{\varphi}}(x)
			:=Ae^{-\frac{|\chi|}{2}\left(c-\frac{1}{\sqrt{\nu}}+\sqrt{\left(c-\frac{1}{\sqrt{\nu}}\right)^2-4\frac{1}{|\chi|}\left(\frac{\nu + 1}{\nu}\right)}\right)x}
	\ee
	is a super-solution for \eqref{eq:slab-super} as long as $\overline{\overline{\varphi}}(\pm L) \geq \bar{\varphi}(\pm L)$.  This last condition is equivalent to the condition: 	\be
		A\geq e^{-\frac{L|\chi|}{2}\left(c-\frac{1}{\sqrt{\nu}}+\sqrt{\left(c-\frac{1}{\sqrt{\nu}}\right)^2-4\frac{1}{|\chi|}\left(\frac{\nu + 1}{\nu}\right)}\right)}.
	\ee
	Hence, taking $A$ to be exactly equal to the quantity on the right-hand side of the previous line, we find $\bar\varphi \leq \overline{\overline \varphi}$.  Together with~\eqref{e.08290201}, this implies~\eqref{eq:slab-super-0}.
	Lemma \ref{lem:sub-super-slab-0} is proved.
\end{proof}

As a consequence of Corollary \ref{cor:continuum-fixed-points} and Lemma \ref{lem:sub-super-slab-0}, we obtain the following result:
\begin{proposition}\label{prop:existence-slab-delta} 
	For all $\delta \in (0, \sfrac{\nu}{(\nu+1)})$, there exists $L_\delta>0$ such that for all $L\geq L_\delta$ there exists a solution $(u,c_\delta)$ to \eqref{eq:slab-nupos} such that $u\in C^1([-L, L])$ is decreasing and satisfies $u(0)=\delta$, 
    and the speed satisfies,
		\be
		c_\delta\in\left(\underline{c}, \bar c\right),
	\ee
	where $\underline{c}$ and $\bar{c}$ are given by,
			\be
		\label{e.under c bar c}
	\underline{c}=	\frac{1}{\sqrt{|\chi|}}\sqrt{\frac{\nu}{\nu+1}} \quad\text{ and }\bar c = \frac{1}{\sqrt \nu} + \frac{2}{\sqrt{|\chi|}} \sqrt{ \frac{\nu+1}{\nu}}.
	\ee
\end{proposition}
\begin{proof}
	Fix $\eps>0$ such that $\delta < (1-\eps) (\sfrac{\nu}{(\nu+1)})$.  
	Let $\mathcal{C}$ be the connected set of solutions  $ (c,u)\in \left[0, \bar c\right]\times C^{1, \alpha}([-L, L])$ to \eqref{eq:slab-nupos} provided by Corollary \ref{cor:continuum-fixed-points}. 	The map $\Phi:  \mathcal{C}
	\to \mathbb R$ defined by $\Phi(c,u) = u(0)$ 
	is continuous.  Hence, the image $\Phi(\mathcal{C})$ is connected.

	Let $\bar u$ be a solution corresponding to $\bar c$ and let $\underline u$ be a solution corresponding to $\underline c$ such that $(\bar c, \bar u), (\underline c, \underline u) \in \mathcal{C}$.  These exist due to \Cref{cor:continuum-fixed-points}.  Next, let $L$ be large enough so that the conclusion of Lemma \Cref{lem:sub-super-slab-0} holds; thus, we have,
	\be
		\Phi(\underline c, \underline u)
			\leq \exp\left(-\frac{L|\chi|}{2}\left(\bar c-\frac{1}{\sqrt{\nu}}+\sqrt{\left(\bar c-\frac{1}{\sqrt{\nu}}\right)^2-\frac{4}{|\chi|}\left(\frac{\nu + 1}{\nu}\right)}\right)\right)
		\quad\text{and}\quad
		\Phi(\bar c, \bar u)
			\geq (1-\varepsilon) \frac{\nu}{\nu+1}.
	\ee
	Hence, due to the fact that $\Phi(\mathcal{C})$ is connected, we deduce that
	\begin{equation}
		\left[
			\exp\left(-\frac{L|\chi|}{2}
				\left(\bar c-\frac{1}{\sqrt{\nu}}+\sqrt{\left(\bar c-\frac{1}{\sqrt{\nu}}\right)^2
				-\frac{4}{|\chi|}
				\left(\frac{\nu + 1}{\nu}\right)}\right)\right),
			(1-\varepsilon) \frac{\nu}{\nu+1}
		\right]
		\subset
		\Phi(\mathcal{C}).
	\end{equation}
	Thus, by further increasing $L$ as to guarantee 
		\be
		\exp\left(-\frac{L|\chi|}{2}
				\left(\bar c-\frac{1}{\sqrt{\nu}}+\sqrt{\left(\bar c-\frac{1}{\sqrt{\nu}}\right)^2
				-\frac{4}{|\chi|}
				\left(\frac{\nu + 1}{\nu}\right)}\right)\right)
			< \delta
	\ee
holds, we find that	there exists $(c, u) \in\mathcal{C}$ with $\Phi(c,u)=u(0)=\delta$.  This completes the proof.
\end{proof}
We are now in the position to prove \Cref{prop:decreasing-wave}.
\begin{proof}[Proof of \Cref{prop:decreasing-wave}]

	Fix $\nu>0$, $\chi<0$, and $\delta\in \left(0, \frac{\nu}{\nu+1}\right)$. By Proposition \ref{prop:existence-slab-delta}, for $L$ sufficiently large,  there exists a $(c_L, u_L, v_L)$ with 
	\be\label{e.c80901}
		c_L\in \left(\underline{c}, \bar{c}\right), \text{where $\underline{c}$ and $\bar c$ are given by	~\eqref{e.under c bar c},}
	\ee 
	such that $u$ solves \eqref{eq:slab-nupos} and
\be
	u_L(0)=\delta .
\ee

	By a standard diagonalization procedure, up to taking a subsequence, there exists $(c,u,v)$ and a sequence $(c_{L_n}, u_{L_n}, v_{L_n})$ with $L_n\to +\infty$ such that $c_{L_n}\to c$ and  $u_{L_n}\to u$ and $v_{L_n}\to v$ locally uniformly in $C^2$.  The convergence implies that $(c, u, v)$ solves \eqref{e.tw}.  The limit $u(+\infty)$ follows from \Cref{lem:exp decay} and the limit $u(-\infty) = 1$ follows easily by standard arguments (see, e.g., the arguments used to establish that $u(-\infty) =1$ in the proof of \Cref{t.lower.hyp}) and the facts that $u$ is monotonic and $u(0) = \delta$.

	The remaining thing to prove~\eqref{eq:T5.1-Znonempty}: the smallness of $v_x + c$ at a point for $|\chi|$ sufficiently large.  This is the main difficulty in the proof.  We do this now.  Actually, we establish this at a point $x_n \in (-L_n,L_n)$ and argue that $x_n$ remains in a bounded interval around the origin as $n\to\infty$.  For the remainder of the proof, we'll denote $(c_{L_n}, u_{L_n}, v_{L_n})$  as $(c_n, u_n, v_n)$. In addition, we now we clarify what we mean by $|\chi|$ sufficiently large.  First, we recall from Proposition \Cref{p.speed.hyp} that  $c^*_{\hypnu}>0$ holds. Now we take $|\chi|$ sufficiently large so that \Cref{t.lower.hyp}.(i) guarantees that any traveling wave solution of~\eqref{e.hyp} satisfies
	\be\label{e.c80906}
		c \geq \frac{c^*_{\hypnu}}{2} > 0.
	\ee
In addition, we take $|\chi|$ large enough so that	\be\label{e.c80907}
		\frac{c^*_{\hypnu}}{8}
			\geq \sqrt{\frac{4}{|\chi|} \frac{\nu+1}{\nu}}
	\ee
holds.	In the remainder of the proof, $|\chi|$ will not be further increased.

We shall now establish~\eqref{eq:T5.1-Znonempty}. To this end, let
\begin{equation}
	\tilde u(x):=u_n(x)e^{|\chi|\frac{c_n x+v_n(x)}{2}}. 
\end{equation}
Then $\tilde u$ satisfies the equation
\begin{equation}\label{eq:slab-tildenupos}
		- \frac{1}{|\chi|} \tilde u_{xx} =\tilde u\left(1+ \frac{v_n(x)}{\nu}-\left(\frac{\nu + 1}{\nu}\right)u_n(x)-\frac{((v_n)_x+c_n)^2}{4\frac{1}{|\chi|}}\right)
		\quad\text{ in } (-L_n,L_n).
\end{equation}
Let us take $L_n\geq \left(1-\frac{2}{|\chi|}\log\frac{\delta}{2}\right)\underline{c}^{-1}$ and consider any $x\leq  \left(\frac{2}{|\chi|}\log\frac{\delta}{2}-1\right)(\underline{c})^{-1}$. 
From the definition of $\tilde{u}$, 	that $u_n\leq 1$ and $v_n \leq 1$, the fact that $x<0$ and the upper bound on $c_n$ from~\eqref{e.c80901},  we find,
\be\label{eq.22090201}
\tilde{u}(x)\leq e^{\frac{|\chi|}{2}(c_nx+1)}\leq e^{\frac{|\chi|}{2}(\underline{c}x+1)}
\leq e^{\frac{|\chi|}{2}(\frac{2}{|\chi|}\log\frac{\delta}{2})}=\frac{\delta}{2}.
\ee
Thus, we have,
\be\label{e.c81301}
		\sup_{x\in\left(-L_n, \left(\frac{2}{|\chi|}\log\frac{\delta}{2}-1\right)(\underline{c})^{-1}\right)}\tilde u(x)\leq \frac{\delta}{2}<\delta,
			\qquad \tilde u(0) = \delta e^{\frac{|\chi|}{2}v(0)}\geq \delta,
			\qquad \tilde u(L_n)=0<\delta, 
	\ee
which proves that $\tilde u$ has an interior global maximum $ x_n$ satisfying,
	\begin{equation}\label{eq:x*-lower}
		x_n\in \left(\left(\frac{2}{|\chi|}\log\frac{\delta}{2}-1\right)\left(\underline{c}\right)^{-1},L_n\right). 		
	\end{equation}

Testing~\eqref{eq:slab-tildenupos} at the location of the maximum $x_n$, we have
\be
	0\leq -\frac{1}{|\chi|} \tilde u_{xx}
		= \tilde u\left(1+ \frac{v_n(x_n)}{\nu}-\left(\frac{\nu + 1}{\nu}\right)u_n(x_n)-\frac{|\chi|((v_n)_x(x_n)+c_n)^2}{4}\right), 
\ee
so that
    \be\label{e.67517}
	\frac{|\chi|(v_x(x_n)+c_n)^2}{4} \leq \frac{\nu + 1}{\nu}.
\ee

	Next, we show that $x_n$ is contained  in a bounded interval around the origin.  The fact that $x_n$ cannot be too negative follows from \eqref{eq:x*-lower}. We now show that $x_n$ cannot be too positive.  First notice that, due to~\eqref{e.c80906}, we may assume that
	\be
		c_n \geq \frac{c_{\hypnu}^*}{4},
	\ee
	up to increasing $n$.  Combining this with~\eqref{e.c80907} and~\eqref{e.67517} yields
	\be\label{e.c80908}
		|v_x(x_n) + c_n|
			\leq \frac{c_n}{2}.
	\ee
	From~\eqref{e.kernel} and \Cref{lem:exp decay}, it is clear that there are positive constants $C', \theta'>0$, independent of $L$ and $\chi$ (recall that we have already restricted to $\chi$ sufficiently large, so there is no dependence on $\chi$ through \Cref{lem:exp decay}), such that
	\be\label{e.c80903}
		|v_x|\leq C' e^{-\theta'x}.
	\ee 
	Combining this with~\eqref{e.c80908}, we find that either $x_n \leq 0$ or
	\be
		c_n - C' e^{-\theta' x_n}
			\leq \frac{c_n}{2},
	\ee
	which yields
	\be
		\frac{1}{C'} e^{\theta' x_n}
			\leq \frac{2}{c_n}
			\leq \frac{8}{c_{\hypnu}^*}.
	\ee
	Taking the logarithm of both sides, we find
	\be\label{e.c80909}
		x_n \leq
			\frac{1}{\theta'} \log\Big( \frac{8 C'}{c_{\hypnu}^*}\Big),
	\ee
	which is the desired bound.

	Up to passing to a further subsequence, there is $x_*$ such that $x_n \to x_*$ as $n\to\infty$.  From the $C^1$ convergence (recall that $\|v_n\|_{W^{2,\infty}} \leq C(1 + \sfrac 1\nu)$) of $v_n$ to $v$, it follows that
	\be
		(v_x(x_*) + c)^2
			\leq \frac{4}{|\chi|} \frac{\nu + 1}{\nu}.
	\ee
	The bounds on $x_*$ follow directly from~\eqref{eq:x*-lower} and~\eqref{e.c80909}.  This completes the proof.
	\end{proof}

\ssection{Converging to the porous medium wave: \Cref{t.upper.pm}}\label{s.upper.pm}

We now show that discontinuous traveling wave solutions to~\eqref{e.hyp} converge to traveling wave solutions to~\eqref{e.pm}. 
We work with the solutions constructed in \Cref{t.upper.hyp}.  We point out the properties that we use: first,
\be\label{e.c72101}
	u_\nu \text{ is nonincreasing},
	\quad
	(v_\nu)_x(0) = -c_\nu,
	\quad\text{ and }\quad
	u_\nu \equiv 0 \text{ in } (0,\infty).
\ee
The first is due to \Cref{t.upper.hyp}.(i), the second follows from~\eqref{e.c81605}, and the third follows from \Cref{t.upper.hyp}.(ii) and \Cref{prop:cardZleq1}. 
Additionally, we have that
\be\label{e.c72102}
	-c_\nu < (v_\nu)_x(x) < 0 \quad\text{ for } x\in (-\infty,0)
	\quad\text{ and }\quad
	u_\nu \text{ is smooth in } (-\infty,0).
\ee

Next, we observe that $u$ and $v$ are close, depending on $\nu$.  Indeed, as $(c,u,v)$ is constructed as a limit of solutions to~\eqref{e.tw}, we may apply the estimate \Cref{lem:u-v}
\be\label{e.c72104}
	\|u_\nu-v_\nu\|_{L^\infty} \leq C (\sqrt c_\nu + 1) \nu^{\sfrac{1}{8}},
\ee
for some universal $C>0$.  Similarly, we also have, from \Cref{lem:apriori bd nu to 0} and \Cref{c.holder},
\be\label{e.c72106}
	\int |u_\nu(1-u_\nu) \log u_\nu| \dd x\leq c_\nu
	\quad\text{ and }\quad
	[v_\nu]_{C^{\sfrac{1}{2}}}, 
	~~ \|(v_\nu)_x \|_{L^2} \leq \sqrt{c_\nu}.
\ee

We begin by obtaining an upper bound on $c$.  Note that this does not follow from the previously established upper bound in \Cref{p.speed.hyp}.  For notational ease, we drop the $\nu$ subscript here.
\begin{lemma}\label{lem:c leq 2}
Suppose that $(c,u,v)$ is a traveling wave solution to~\eqref{e.hyp} satisfying~\eqref{e.c72101} and~\eqref{e.c72102}.  Then
	\be
		v(x) \geq (c+x)|x| \quad\text{ for all } x\in (-c,0)
		\qquad\text{ and } \qquad
		c \leq 2.
	\ee
\end{lemma}
\begin{proof}
First notice that the second claimed inequality follows from the first since $v \leq 1$.  Hence, we now focus on the first inequality.

Fix any $x < 0$.  Integrate~\eqref{e.hyp} over $[x,0]$ and recall~\eqref{e.c72101} to find
\be
	\int_x^0 u(1-u)\dd x
		= - c u(0) + c u(x)
			- u(0) v_x(0)
			+ u(x) v_x(x)
		= (c+v_x(x))u(x).
\ee
Since $u$ is nonincreasing, we have
\be
	|x| u(x)
		\geq \int_x^0 u(1-u)\dd x
		=  (c+v_x(x)) u(x).
\ee
After dividing by $u(x)$ and rearranging, we find
\be
	c - |x| \leq -v_x(x).
\ee

By the mean value theorem, we can find $\xi_x \in [x,0]$ so that
\be
	\begin{split}
	v(x) &= v(0) + v_x(\xi_x) x
		= v(0) + (- v_x(\xi_x))|x|
		\\
		&\geq v(0) + (c - |\xi_x|) |x|
		\geq v(0) + (c - |x|) |x|
		\geq (c-|x|)|x|.
	\end{split}
\ee
The proof is completed by evaluating the above at $x=-c/2$:
\be
	v\Big(-\frac{c}{2}\Big) \geq \Big(c - \frac{c}{2}\Big) \frac{c}{2} =  \frac{c^2}{4}. \qedhere
\ee
\end{proof}

Next, we show that $c$ cannot degenerate to zero as $\nu\searrow 0$.  Notice that this is important because the generic lower bound in \Cref{p.speed.hyp} degenerates as $\nu \searrow 0$.
\begin{lemma}\label{lem:c nonzero}
	Fix any $\nu_M>0$.  Suppose that $(c,u,v)$ is a traveling wave solution to~\eqref{e.hyp} with $\nu \in (0,\nu_M)$ satisfying~\eqref{e.c72101}-\eqref{e.c72106}.  Then there is $\underline c > 0$, depending only on $\nu_M$, such that
	\be
		c  > \underline c.
	\ee
\end{lemma}
\begin{proof}
	First note that we need only check $\nu \ll 1$ due to \Cref{p.speed.hyp}.  Before beginning, we describe the intuition behind the proof. If $c$ is small,~\eqref{e.c72106} forces $u$ to transition from $1$ to $0$ `quickly' and $v$ to be nearly constant.  This is not consistent with the fact that~\eqref{e.c72104} makes $u$ and $v$ `close.'
	
	To this end, let
	\be
		x_\ell = v^{-1}(2/3)
		\quad\text{ and }\quad
		x_r = v^{-1}(1/3).
	\ee
	Notice that, due to~\eqref{e.c72106}, it follows that
	\be\label{e.c72108}
		\frac{1}{3}
			= |v(x_\ell) - v(x_r)|
			\leq \sqrt{c |x_\ell - x_r|}.
	\ee
	Then, in view of~\eqref{e.c72104} and up to decreasing $\nu_0$, we have
	\be\label{e.c72109}
		\frac{1}{4} \leq u(x) \leq \frac{3}{4}
			\qquad\text{ for all } x \in [x_\ell, x_r].
	\ee
	Incorporating~\eqref{e.c72108} and~\eqref{e.c72109} into~\eqref{e.c72106} yields
	\be
		\frac{1}{9c} \frac{3 \log (4/3)}{16}
			\leq |x_\ell - x_r| \min_{[x_\ell,x_r]} u(1-u) |\log u|
			\leq \int_{x_\ell}^{x_r} u(1-u) |\log u|\dd x
			\leq c.
	\ee
	Rearranging this completes the proof.
\end{proof}

We are now in position to complete the proof of the theorem.
\begin{proof}[Proof of \Cref{t.upper.pm}]
	We establish this by showing that any sequence has a subsequence that converges in the claimed manner.  To this end, fix $(c_n, u_n, v_n)$ with $\nu_n \searrow 0$ as $n\to\infty$.
	
	Using \Cref{lem:c leq 2},~\eqref{e.c72106}, and~\eqref{e.c72104} and after passing to a subsequence, we obtain $\bar c$ and $\bar u$ such that $c_n \to \bar c$, $u_n \to \bar u$ locally uniformly, and $v_n \to \bar v$ locally uniformly in $C^\alpha_{\rm loc}$ for any $\alpha \in (0,1/2)$ and weakly in $H^1_{\rm loc}$.  From~\eqref{e.c72104} we know that $\bar u=\bar v$ so $\bar u$ is continuous (even $C^{\sfrac{1}{2}}$).  We also note that, since $u_n$ is nonincreasing for every $n$, so is $\bar u$.  We use this often in the next paragraphs.
	
	First, we check the qualitative behavior of $\bar u$.  Note that, due to \Cref{lem:c nonzero}, we have that $\bar c > 0$.  Then, from \Cref{lem:c leq 2} and the monotonicity of $u$, we have that $\bar u > 0$ for $x \in (-\infty, 0)$.  In fact, we have
	\be
		\bar u(x) \geq \bar u(-\bar c/2) \geq \bar c^2/4
		\qquad\text{ for all } x \leq -\frac{\bar c}{2}.
	\ee
	In view of~\eqref{e.c72106} and using the same arguments as in the proof of~\eqref{e.c81701}, we conclude that $\bar u(-\infty ) = 1$.  Additionally, from~\eqref{e.c72101}, we have that $\bar u(x) \equiv 0$ on $(0,\infty)$.  In summary,
	\be
		\bar u(-\infty),
			\quad\text{ and } \quad 
			\{u = 0\} = [0,\infty).
	\ee

	Next, we briefly show that $\bar u$ is a distributional solution to \eqref{e.pm}.  Testing the equation satisfied by $u_n$ with a smooth function $\varphi$ supported on a compact interval $I\subset(-\infty, 0)$, we find that:
	\begin{equation}
		c_n\int (u_n)\varphi_x + \int u_n(v_n)_x\varphi_x = \int u_n(1-u)_n\varphi.
	\end{equation}
	The first integral converges to $\int \bar u\varphi_x $ and the last to $\int \bar u(1-\bar u)\varphi$ because $u_n\to u$ locally uniformly. The middle integral converges to $\int\bar u\bar v_x \varphi_x = \int \bar u\bar u_x \varphi_x$ because $u_n\to u$ locally uniformly (hence strongly in $L^2(I)$), and $v_n\to \bar v $ weakly in $H^1(I)$. Thus $\bar u\in H^1_{\mathrm{loc}}$ is a distributional solution of~\eqref{e.pm} on $(-\infty, 0)$. It follows from classical arguments that $\bar u\in C^{2, \alpha}_{\mathrm{loc}}(-\infty, 0)$ and is a classical solution of~\eqref{e.pm} on $(-\infty, 0)$, and that $u$ is strictly decreasing. Thus $u$ is a solution of~\eqref{e.pm} in the sense of Lemma \ref{p.pm classical} (\eqref{item:Aronson soln}), and we conclude that $u$ is a distributional solution on $\mathbb R$.  
	
	Since $\bar u = 0$ on $(0,\infty)$ (recall~\eqref{e.c72101}), it follows that $\bar u$ is the minimal speed traveling wave solution to~\eqref{e.pm}~\cite{Aronson1980}.  It follows that $\bar c = 1/\sqrt 2$.  This completes the proof.
\end{proof}

\section{Proofs of technical lemmas}\label{s.technical}

\ssection{Exponential decay: \Cref{lem:exp decay}}

\begin{proof}[Proof of \Cref{lem:exp decay}]
	First, we point out that the normalization~\eqref{e.c8027} guarantees that $u$ is nonincreasing on $x \geq 0$.  In cases (i) and (iii) this is due to \Cref{lem:Chris monot} and the boundary data $u(+\infty) = 0$ and $u(L) = 0$, respectively.  In case (ii), this follows from \Cref{lem:monotonicity.hyp}. 
	Moreover, we can also compare with points to the left of the origin.  Indeed, the monotonocity results above indicate that $u(x) \geq  u(0)$ for $x \leq 0$.  Hence, 
\be
\label{e.monot conseq}
	\text{if $x\geq 0$ and $z\leq x$ then $u(z)\geq u(x)$.}
\ee

Next, the proof for solutions of~\eqref{e.hyp} is significantly easier than for solutions of~\eqref{e.tw}, so we show only the proof in the setting of the latter equation.  Actually, by formally taking $\chi = -\infty$ in the computations below, one arrives immediately at a proof for solutions of~\eqref{e.hyp}.

We start by defining the following constants. First, let
\be
	A_0 = 4 \log 8
	\quad\text{ so that }\quad
	e^{-3A_0/4}+e^{-A_0/4} \leq 1/4.
\ee
Next, define,
\be
\label{d.A}
	A=
		\max\left\{A_0,
			\frac{1}{\nu},
			\frac{64 c(2\nu+1)}{\nu}
		\right\}
	\quad\text{and}\quad
	\mu
		= \min\left\{\frac{1}{8},
			\frac{1}{16c (2\nu+1)},
			\frac{|\chi_0|A\nu}{2}
		\right\}.
\ee
We shall prove that
\be\label{e.exp_decay}
	u (x_0 + A \sqrt{ \nu}) \leq (1 - \mu) u (x_0).
\ee
as long as $x_0 \geq 0$ and $u$ solves~\eqref{e.tw} on $[x_0, x_0 + A \sqrt \nu)$.  We point out that the dependence on $\nu_m$, $\nu_M$, and $C_M$ is clear due to the explicit nature of $A$ and $\mu$.

	Before beginning with the proof of~\eqref{e.exp_decay}, we show how it yields the claim.  Fix any $x$ and let $n_x=\lfloor \frac{x}{A\sqrt{\nu}}\rfloor$ be such that
\be
	A n_x \sqrt \nu
		\leq x
		< A(n_x+1) \sqrt \nu.
\ee
Then, by the monotonicity of $u$ pointed out at the outset of the proof, it follows that
\be
	u(x) \leq u(A n_x \sqrt \nu).
\ee
On the other hand, applying~\eqref{e.exp_decay} $n_x$ times yields
\be
	u(x) \leq u(A n_x \sqrt \nu)
		\leq (1-\mu)u(A (n_x -1) \sqrt \nu)
		\leq \cdots
		\leq (1-\mu)^{n_x} u(0)
		\leq (1-\mu)^{\frac{x}{A\sqrt\nu}-1} \frac{\nu}{\nu+1}.
\ee
Hence, up to establishing~\eqref{e.exp_decay}, the proof is complete.

We now prove~\eqref{e.exp_decay}.  We argue by contradiction: assume there exists $x_0>0$ such that 
\be\label{e.c8023}
    u(x_0 + A \sqrt{  \nu}) \geq  (1 - \mu) u(x_0).
\ee
Since the proof is quite intricate, let us explain the main idea.  The monotonicity of $u$ and~\eqref{e.c8023} indicates that $u$ is approximately constant on a large interval.  Given its definition, it follows that $v\approx u$ on this interval.  Using then~\eqref{e.tw} and expanding the $(v_x u)_x$ term, we find
\be
	\begin{split}
	- c u_x &= u\Big(1 + \frac{v}{\nu} - \Big(\frac{\nu + 1}{\nu}\Big) u\Big)
			+ v_x u_x
			+ \frac{1}{|\chi|} u_{xx}
		\\&\approx
		u\Big(1 - u\Big)
			+ v_x u_x
			+ \frac{1}{|\chi|} u_{xx}.
	\end{split}
\ee
Let us ignore the second order term, which is only a technical issue.  Next note that we expect $v_x u_x \geq 0$ due to the monotonicity of $u_x$ on $[0,\infty)$.  Finally, since $u$ changes by $\mu u$ over an interval of length, $A\sqrt\nu$, we expect:
\be
	\frac{c \mu u}{A\sqrt\nu}
		\approx - c u_x 
		\geq
		u\Big(1 - u\Big)
		+ o(1).
\ee
Canceling a $u$ term on each side and noting that $u \leq \nu/(\nu+1)<1$, would yield a contradiction since the left hand side tends to zero as $A\to\infty$ or $\mu\to 0$, but the right hand side is positive.  This is roughly how the proof proceeds.  Most of the technical difficulty is in dealing with the second order term.

We now proceed with the proof. 
The first step is to establish that, under the assumption~\eqref{e.c8023}, we have, 
\be
\label{e.v u/8}
v(x)\geq u(x_0)/8 \text{ for all $x \in [x_0, x_0 + 3A \sqrt{ \nu}/4]$}.
\ee
To this end, from~\eqref{e.kernel}, we have 
\be
v (x)
        \geq  \int_{-\infty}^{\frac{-3A\sqrt{\nu}}{4}} \frac{e^{-\frac{|y|}{\sqrt\nu}}}{2\sqrt \nu} u (x+y) dy
        		+ \int_{{-3A\sqrt{\nu}/4}}^0 \frac{e^{-\frac{|y|}{\sqrt\nu}}}{2\sqrt \nu} u (x+y) dy
		+ \int_0^{A \sqrt \nu/4}\frac{e^{-\frac{|y|}{\sqrt\nu}}}{2\sqrt \nu} u (x+y) dy
.
        \ee
Notice that $x - 3 A \sqrt\nu/4 \leq x_0$. 
We bound the first integral with~\eqref{e.monot conseq} and the second and third integrals with~\eqref{e.c8023}.  This yields
\begin{align}
    v (x)
        &\geq \int_{-\infty}^{{-3A\sqrt{\nu}/4}} \frac{e^{\frac{y}{\sqrt\nu}}}{2\sqrt \nu} u (x_0) dy
        {+\int_{-3A\sqrt{\nu}/4}^0 \frac{e^{\frac{y}{\sqrt\nu}}}{2\sqrt \nu} (1-\mu) u (x_0) dy}
        + \int_0^{A \sqrt \nu /4} \frac{e^{-\frac{y}{\sqrt\nu}}}{2\sqrt \nu} (1-\mu) u (x_0) dy
        \\&= \frac{1}{2}\left({e^{-3A/4}+\big(1-e^{-3A/4}\big)(1-\mu)} +\big(1- e^{-A/4}\big) (1-\mu)\right)u (x_0)
       \\&= \left(\frac{1}{2}(1-\mu)(2- e^{-3A/4} -e^{-A/4})+\frac12 e^{-3A/4}\right)u(x_0)\\
       &\geq \left((1-\mu)+ (1-\mu)\left(\frac{- e^{-3A/4} -e^{-A/4}}{2}\right)\right)u(x_0),
        \end{align}
        from which~\eqref{e.v u/8} follows by using  that $A\geq A_0$ and $\mu\leq 1/8$.

Next, we use~\eqref{e.v u/8} and the monotonicity of $u$ to find, for all $x \in [x_0, x_0 + 3A \sqrt{ \nu}/4]$,
\[
1 + \frac{v (x)}{\nu} - \left(\frac{\nu + 1}{\nu}\right) u (x) \geq 1+ \left(-1-\frac{1}{4\nu}\right)u(x_0).
\]
We use the inequality $u(x_0)\leq 2\nu/(2\nu+1)$, which follows from~\eqref{e.c8027} and monotonicity, to bound the right-hand side from below and obtain
\be
\label{e.RHS est exp}
1 + \frac{v (x)}{\nu} - \left(\frac{\nu + 1}{\nu}\right) u (x) \geq
 1- \frac{4\nu-1}{4\nu}\frac{2\nu}{2\nu+1} = \frac12\frac{1}{2\nu+1} \text{ for all $x \in [x_0, x_0 + 3A \sqrt{ \nu}/4]$}.
\ee

Our next ingredient is the observation that, for all $x\geq x_0$, we have,
    \begin{equation}\label{e.q99862}
        \sqrt{\nu}v_x(x) = \int_0^{+\infty} \dfrac{e^{-\frac{y}{\sqrt{\nu}}}}{2\sqrt{\nu}} \big(u (x+y)-u (x-y)\big)\dd y \leq 0.
    \end{equation}
Indeed, the equality follows from the expression~\eqref{e.kernel} for $v$.  The inequality holds pointwise on the integrand, which can be seen by considering two cases.  First, when $x_0 - y \geq 0$, this is due to~\eqref{e.monot conseq}.  In the other case, when $x_0 - y < 0$, \Cref{lem:Chris monot} implies that $u(x_0 - y) \geq \nu/(\nu+1)$.  On the other hand,~\eqref{e.monot conseq} implies that
\be
	u(x_0 + y) \leq u(0) = \frac{\nu}{\nu+1}.
\ee
Hence, $u(x_0 +y) - u(x_0 - y) \leq 0$, as claimed.  This finishes the proof of~\eqref{e.q99862}.

Next, define
\be\label{e.c8025}
	F := \left\{x \in [x_0, x_0 + A\sqrt\nu] : u'(x) > - \frac{2\mu u(x_0)}{A\sqrt\nu}\right\}.
\ee
These are the points where $u$ is flat, hence the $F$ notation. Note that, due to~\eqref{e.c8023},
\be
	\mu u(x_0)
		\leq u(x_0) - u(x_0 + A\sqrt\nu)
		\leq \int_F u_x(x)\dd x
		< |F| \frac{2 \mu u(x_0)}{A\sqrt \nu},
\ee
so that
\be
	|F| \geq \frac{A \sqrt \nu}{2}.
\ee
It is useful to further restrict $F$.  Let
\be\label{e.c8029}
    \tilde F = F \cap [x_0, x_0 + 3A \sqrt{ \nu}/4],
        \quad\text{ which has }
        |\tilde F| \geq \frac{A \sqrt\nu}{4}.
\ee

We shall now establish a lower bound on $u_{xx}$ that holds on $\tilde{F}$. To this, end, we use~\eqref{e.c8025} and the equation~\eqref{e.tw2} satisfied by $u$ and $v$, to find,  for all $x \in \tilde F$,
\begin{align}
    \frac{2 \mu u (x_0)}{A \sqrt \nu} c 
    &\geq - c  u _x(x)
        = u (x) \left(1 + \frac{v (x)}{\nu} - \left(\frac{\nu + 1}{\nu}\right) u (x) \right)
        + v_x(x) u_x(x) + \frac{u_{xx}(x)}{|\chi|}.
        \end{align}
    Then, using~\eqref{e.c8023} and~\eqref{e.RHS est exp} to bound the first term on the right-hand side from below and~\eqref{e.c8025} to estimate the second term yields
        \begin{align}
    \frac{2 \mu u (x_0)}{A \sqrt \nu} c 
        &\geq u(x_0)(1-\mu) \frac12\frac{1}{2\nu_M+1}
        + 0 + \frac{u_{xx}(x)}{|\chi|}.
\end{align}
Rearranging this and recalling the definition of $A$ and $\mu$~\eqref{d.A}, we find
\begin{align}\label{e.q1012356}
\frac{u_{xx}(x)}{u(x_0)|\chi|}\leq 
    \frac{2 \mu}{A \sqrt{\nu}} c
    - (1-\mu) \frac12\frac{1}{2\nu+1}\leq \frac{1}{8(2\nu+1)}-\frac{1}{4}\frac{1}{2\nu+1}=-\frac{1}{8}\frac{1}{2\nu+1}.
\end{align}
On the other hand, we bound $u_{xx}$ off of $\tilde{F}$ by again employing the equation~\eqref{e.tw2}:
\be
\label{e.u'' off F}
\begin{split}
    \frac{u_{xx}(x)}{|\chi|} &=-cu_x(x) - u_x(x)v_x(x) -u(x)\left(1+\frac{v(x)}{\nu}-\left(\frac{\nu + 1}{\nu}\right)u(x)\right) \\
    &\leq -c u_x(x),
    \end{split}
\ee
where the second inequality follows from the estimates \eqref{e.q99862} as well as~\eqref{e.RHS est exp}.

To begin to put these ingredients together, we use Taylor's theorem to find,
\begin{align}
	u\big(x_0+A\sqrt{\nu}\big) & = u(x_0)+u_x(x_0)A\sqrt{\nu}+\int_{0}^{A\sqrt{\nu}}y u_{xx}(x_0+A\sqrt{\nu}-y)\dd y.
    \end{align}
Next, recall from~\eqref{e.c8029} that $\tilde F \subset [x_0 , x_0 + 3A \sqrt \nu /4]$ and $|\tilde F| \geq A \sqrt \nu / 4$. Together with the fact that $u_x(x_0)$ is non-positive we thus find,
\be
\label{e.Taylor}
\begin{split}
    u\big(x_0+A\sqrt{\nu}\big)
    	& \leq u(x_0)+\int_{x_0+A\sqrt{\nu}-\tilde F} y u_{xx}(x_0+A\sqrt{\nu}-y)\dd y \\
	&\quad \quad + \int_{x_0+A\sqrt{\nu}-[x_0, x_0+A\sqrt{\nu}]\backslash\tilde F} y u_{xx}(x_0+A\sqrt{\nu}-y)\dd y.
    \end{split}
\ee
We shall now bound each of the terms on the right-hand side. For the first term, we use the estimate~\eqref{e.q1012356} to find,
\be\label{e.c72908}
	\begin{split}
		\int_{x_0+A\sqrt{\nu}-\tilde F} &yu_{xx}(x_0+A\sqrt{\nu}-y)\dd y
    	\leq -\frac{1}{8(2\nu+1)}|\chi|u(x_0) \int_{x_0+A\sqrt{\nu}-\tilde F}y\dd y \\
    &\leq -\frac{1}{8(2\nu+1)}|\chi|u(x_0) \int_{0}^{A\sqrt{\nu}/{4}}y\dd y
    \leq-\frac{1}{8(2\nu+1)}|\chi|u(x_0)\dfrac{A^2\nu}{32},
	\end{split}
\ee
where, to obtain the second inequality, we used that $|\tilde F|\geq \frac{A\sqrt{\nu}}{4}$ (recall~\eqref{e.c8029}) and that the minimum of
\be
	G\mapsto \int_{G}y\dd y
	\qquad\text{ for any } G \in \mathcal{G}
\ee
is attained at $G = [0,A\sqrt\nu/4]$, where we have defined $\mathcal{G}$ to be the set of Lebesgue measurable sets $G$ in $[0, +\infty]$ with $|G|\geq \frac{A\sqrt{\nu}}{4}$.

For the portion of the integral off of $\tilde F$, we use~\eqref{e.u'' off F} to find
\be\label{e.c72909}
	\begin{split}
		&\int_{x_0+A\sqrt{\nu}-[x_0, x_0+A\sqrt{\nu}]\backslash\tilde F} y u_{xx}(x_0+A\sqrt{\nu}-y)\dd y
		= \int_{[x_0, x_0+A\sqrt{\nu}]\backslash\tilde F}(x_0+A\sqrt{\nu}-y) u_{xx}(y)\dd y \\
    &\qquad\leq A\sqrt{\nu} \int_{x_0}^{x_0+A\sqrt{\nu}} -|\chi|cu_y(y)\dd y
    = -A|\chi|c\sqrt{\nu}\big(u(x_0)-u(x_0+A\sqrt{\nu})\big)\\
    &\qquad\leq A|\chi|c\sqrt{\nu} \mu u(x_0).
	\end{split}
\ee

We now use~\eqref{e.c72908} and~\eqref{e.c72909} to bound the second and third terms on the right-hand side of~\eqref{e.Taylor} from above, and also use~\eqref{e.c8023} to bound the left-hand side of~\eqref{e.Taylor} from below, to obtain
\begin{align}
    (1-\mu)u(x_0)&\leq u(x_0+A\sqrt{\nu}) \leq u(x_0)-\dfrac{|\chi|A^2\nu}{32\cdot 8(2\nu+1)}u(x_0) + A|\chi|c\sqrt{\nu}\mu u(x_0) \\
    &=u(x_0)\left(1-|\chi|A\sqrt\nu\left(\frac{A\sqrt\nu}{32\cdot 8(2\nu+1)}-c \mu\right)\right).
 \end{align}
Our choice of $A$ and the fact that $\mu \leq 1/8$ implies that
    \begin{align}
    (1-\mu)u(x_0)&\leq u(x_0)\left(1-|\chi|A\sqrt\nu /2\right),
\end{align}
which yields the desired contradiction by our choice of $\mu$, completing the proof.
\end{proof}

\ssection{The structure of $\cZ$: \Cref{prop:cardZleq1}}

We begin by establishing some properties of the set of discontinuities of $u$.

\sssection{Some helpful lemmas}

First we obtain a semi-explicit form of $u$ on $\cZ^c$.  This allows us to deduce its behavior at each boundary point of $\cZ$.

\begin{lemma}\label{lem:explicit-solution}
	Under the assumptions of \Cref{prop:cardZleq1}, 
	suppose that there is $x_m \in \cZ^c$.  Take the maximal interval $(x_-,x_+) \subset \cZ^c$ such that $x_m \in (x_-,x_+)$.  
Let $\tau:\mathbb R\to I_x$ be the solution to the differential equation 
    \begin{equation}\label{eq:scaled-time}
	\left\{\begin{aligned}\relax
	    &\tau'(t) = -c-v'\left(\tau(t)\right), \\ 
	    &\tau(0) = x_m,
	\end{aligned}\right.
    \end{equation}
    for any $x_m \in (x_-, x_+)$.  Then, denoting $\bar v(t):=v\left(\tau(t)\right)$, we have,
    \be\label{eq:explicitsol-hyperbolic}
	u\big(\tau(t)\big)
		= \dfrac{u(x_m) \exp\big\{\int_{0}^t \frac{\nu + \bar v(s)}{\nu}\dd s\big\}}{1+u(x_m)\frac{\nu + 1}{\nu}\int_0^t \exp\big\{\int_0^r \frac{\nu + \bar v(s)}{\nu}\dd s\big\}\dd r} 
			\qquad \text{ for all }t\in\mathbb R.
    \ee
    Moreover, assuming that $u(x_m) > 0$ and that if $x_\pm = \pm \infty$, then the limit $v(x_\pm)$ exists, we have:  when $v_x(x_m)<-c$,
    \begin{align}
	\lim_{t\to -\infty} u\big(\tau(t)\big)
		&=\lim_{y\searrow x_-}u(y) \in \left\{ 0, \dfrac{\nu + v(x_-)}{\nu + 1}\right\}
		~\text{ and }~
		& \lim_{t\to+\infty} u\big(\tau(t)\big)
		&=\lim_{y\nearrow x_+}u(y) = \dfrac{\nu+v(x_+)}{\nu + 1},
    \end{align}
    	and, when $v_x(x_m)>-c$,
    \begin{align}
	\lim_{t\to+\infty} u\big(\tau(t)\big)
		&=\lim_{y\searrow x_-}u(y)
		= \dfrac{\nu+v(x_-)}{\nu + 1}, 
		~\text{ and }~
	& \lim_{t\to -\infty} u\big(\tau(t)\big)
		&=\lim_{y\nearrow x_+}u(y) \in \left\{ 0, \dfrac{\nu + v(x_+)}{\nu + 1}\right\}.
    \end{align}
\end{lemma}

The next results establish that, if $u$ jumps to $0$ at some point $x_0\in\mathbb R$ and $v_x$ does not oscillate too much at this point, then $u$ stays equal to $0$ either on $(x_0, +\infty)$ or on $(-\infty, x_0)$.
\begin{lemma}\label{lem:single-jump-0}
	Suppose that the conditions of \Cref{prop:cardZleq1} hold.
    \begin{enumerate}[(i)]
	\item If there exists  $x_0\in \mathcal Z$ and $\bar x_0 > x_0$ such that
		\begin{equation} 
		    u(x_0^+):=\lim_{x\searrow x_0} u(x)=0
		\end{equation}
		and $(x_0, \bar x_0) \subset \cZ^c$, then
		$c>0$ and
	    \be
		u(x) \equiv 0
			\quad\text{ and }\quad
		v(x)  = c\sqrt{\nu} e^{-\frac{x-x_0}{\sqrt{\nu}}}
			\qquad\text{ for all } x > x_0.
	    \ee

	\item 
	There cannot exist  $x_0\in \mathcal Z$ and $\underline x_0 < x_0$ such that $(\underline x_0, x_0) \subset \mathcal Z^c$ and
		\begin{equation} 
			u(x_0^-):=\lim_{x\nearrow x_0}u(x)=0.
		\end{equation}
    \end{enumerate}
\end{lemma}

Our last result allows us to exclude the possibility of a singular point $x\in \mathcal{Z}$ inside the region $\{u>0\}$.  Since its proof is so short, we include it here.
\begin{lemma}
\label{lem:structure Z}
    Suppose that the conditions of \Cref{prop:cardZleq1} hold.  Let $(x_0, x_1)\subset \mathcal Z^c$ be a maximal connected subset. It cannot be that $x_0$ and $x_1$ are finite and $u(x_0^+), u(x_1^-)>0$.
\end{lemma}
\begin{proof}
    Suppose by contradiction that both $x_0$ and $x_1$ are finite. Then, by \Cref{lem:explicit-solution}, 
    \be
	\lim_{x\to x_0^+} v_{xx}(x)= \lim_{x\to x_0^+}\dfrac{v(x)-u(x)}{\nu} =\frac{1}{\nu}\left(v(x_0)-\dfrac{v(x_0)+\nu}{1+\nu}\right) = \dfrac{\nu(v(x_0)-1)}{1+\nu}<0.
    \ee
    Therefore 
    \be
	v_x(x) = v_x(x_0)+ \int_{x_0}^x v_{xx}(y)\dd y  < v_x(x_0)=-c, 
    \ee
    if $x$ is sufficiently close to $x_0$.

    Similarly, 
    \be
	\lim_{x\to x_1^-} v_{xx}(x)= \lim_{x\to x_1^-}\dfrac{v(x)-u(x)}{\nu} =\frac{1}{\nu}\left(v(x_1)-\dfrac{v(x_1)+\nu}{1+\nu}\right) = \dfrac{\nu(v(x_1)-1)}{1+\nu}<0.
    \ee
    Therefore 
    \be
	v_x(x) = v_x(x_1)+ \int_{x_1}^x v_{xx}(y)\dd y  > v_x(x_1)=-c, 
    \ee
    if $x$ is sufficiently close to $x_1$.

    By the continuity of $v_x$ and the mean value theorem, there must exists $x_*\in (x_0, x_1)$ such that 
    \be
	v_x(x_*)=-c, 
    \ee
    in which case $x_*\in \mathcal Z$; however this is impossible because $I =(x_0, x_1)\subset \mathcal Z^c$ by definition. The contradiction proves the lemma.
\end{proof}

\sssection{The proof of \Cref{prop:cardZleq1}}
We are now in a position to prove \Cref{prop:cardZleq1}.  We write the proof in multiple parts.  First we show that $\cZ$ has the decomposition claimed (that is, either $\mathcal{Z}=\varnothing$ or $\mathcal{Z}=\{x_0\}$ contains exactly one point).  Then we prove the qualitative behavior claimed in cases (1) and (2).  
\begin{proof}[Proof of the decomposition of $\cZ$ in \Cref{prop:cardZleq1}]
	We split the proof in two cases, based on whether $c$ is zero or positive.  We begin with the former, which we show cannot occur.
	
	\medskip

	\noindent\textbf{Case one: $c=0$.}  If $\cZ = \R$, then $v_x \equiv 0$.  From the second equation in~\eqref{e.hyp}, it follows that $u = v$.  From~\eqref{e.hyp2}, we have
	\be
		0 = u \left( (\nu + v) - (\nu + 1) u\right)
			\qquad\text{ in } \R.
	\ee
	These two identities can hold if $u(x) = 0$ or $1$ for every $x\in \R$.  Since $u = v$, it is continuous.  Hence, $u \equiv 0$ or $1$.  This contradicts the hypothesis that $u\not \equiv 0$ and $u\not\equiv 1$ and, thus, concludes the proof when $\cZ = \R$.
	
	If $\cZ \neq \R$, we may take any maximal interval $(x_0, x_1) \subset \cZ^c$, and it must be that either $x_0 > -\infty$ or $x_1 < +\infty$.   Let us consider the former case, but the latter is similar (and, thus, omitted).  The maximality of this interval implies that $x_0 \in \cZ$.
	
	By \Cref{lem:single-jump-0}, if $u(x_0^+)=0$, then $c>0$, which is a contradiction.  Hence, $u(x_0^+)>0$.  If $x_1$ is finite and $u(x_1^-)  = 0$, we have that $u(x_0^+) = 0$ due to \Cref{lem:single-jump-0}, which is a contradiction.  If $x_1$ is finite and $u(x_1^-)>0$ then \Cref{lem:structure Z} is violated.  (Note that, in all cases, the existence of the limit follows from \Cref{lem:explicit-solution}).  It follows that $x_1=+\infty$.
	
	Now, as $c = 0$ and $(x_0,\infty) \subset \cZ^c$, the definition of $\cZ$ (see \Cref{def:hyp}) and the intermediate value theorem imply that $v_x$ has a constant sign on $(x_0,\infty)$.  Since $v$ is bounded, there must be a sequence $y_n \to \infty$ such that $v_x(y_n) \to 0$ as $n\to\infty$.  Note also that, since $x_0 \in \cZ$, $v_x(x_0) = 0$.  Using these facts, along with~\eqref{e.hyp} and that $u$ is bounded and $c=0$, we find
	\be
		\begin{split}
		\int_{x_0}^\infty u(1-u) \dd x
			&= \lim_{n\to\infty} \int_{x_0}^{y_n} u(1-u) \dd x
			= \lim_{n\to\infty} \int_{x_0}^{y_n} (v_x u)_x \dd x
			\\&
			= \lim_{n\to\infty} \left(v_x(y_n) u(y_n) - v_x(x_0) u(x_0^+)\right)
			= 0.
		\end{split}
	\ee
	Since $0 \leq u \leq 1$ and $u\in C_{\rm loc}^1(x_0,\infty)$, it follows that either $u \equiv 0$ or $u\equiv 1$, which is, as above, a contradiction.  It follows that $c\neq 0$.

	\medskip

	\noindent\textbf{Case two: $c>0$.} 
	First, we show that $\cZ^c$ is nonempty and any maximal interval it contains is infinite.
	
	Since $v$ is bounded and $c>0$, it cannot be that $\cZ = \R$.  Hence, $\cZ^c$ is nonempty.  Take any maximal interval $(x_0, x_1) \subset \cZ^c$.  There are three cases to consider: either (i) $x_0 = -\infty$ and $x_1 = +\infty$, (ii) $x_0 > -\infty$, or (iii) $x_1 < +\infty$.  The proof is finished in case (i).  Case (iii) is handled exactly as case (ii); hence, we only show the proof of case (ii).  
	
	Suppose that $x_0 > -\infty$.  If $u(x_0^+) = 0$, it follows from \Cref{lem:single-jump-0} that
	\be
		v_x(x) = - c e^{-\frac{x-x_0}{\sqrt \nu}}
			> -c
			\qquad\text{ for all } x > x_0.
	\ee
	Hence, $(x_0, \infty) \subset \cZ^c$, finishing the claim.  In the other case, where $u(x_0^+)>0$, then \Cref{lem:structure Z} implies that $x_1 = +\infty$.  This concludes the proof that $\cZ^c$ is nonempty and any maximal intervals it contains are infinite.
	
	From the above, we conclude that $\cZ$ is either empty, a single point, or a closed interval.  We need to rule out the last option.  The proof for this was outlined in \Cref{s.main results} (see the discussion around~\eqref{e.c81502}), and, hence, is omitted.  This completes the proof.
	\end{proof}

	\begin{proof}[Proof of the qualitative behavior of $u$ in \Cref{prop:cardZleq1}.(1)]
	The case where $\cZ = \emptyset$ requires only that we establish the regularity of $u$.  This, however, follows directly from \Cref{lem:explicit-solution}.
	\end{proof}
	
	\begin{proof}[Proof of the qualitative behavior of $u$ in \Cref{prop:cardZleq1}.(2)]
	Here, there is $x_0$ such that $\cZ =\{x_0\}$.  By construction of $\cZ$, $v_x + c$ has a constant sign on $(-\infty,x_0)$ and on $(x_0,\infty)$.  Since $v$ is bounded, it follows that $v_x + c > 0$ on both intervals.  Hence,
	\be\label{e.c80805}
		v_x + c >0 \quad \text{ on }\cZ^c
	\ee
	We require~\eqref{e.c80805} in order to determine which case to use in \Cref{lem:explicit-solution} when we apply it below.
	
	By \Cref{lem:explicit-solution}, if $u$ is positive anywhere on a connected component of $\cZ^c$, then it is positive on the entire connected component.  Hence, we have three cases: either $u>0$ on $\cZ^c$, $u\equiv 0$ on $(-\infty,x_0)$, or $u \equiv 0$ on $(x_0,\infty)$.
	
	{\bf Case one: $u>0$ on $\cZ^c$.} Then \Cref{lem:explicit-solution} implies that, up to a redefinition at $x_0$, $u$ is continuous on $\R$ and we have
	\be
		\nu v_{xx}(x_0)
			= v(x_0) - u(x_0)
			= v(x_0) - \frac{\nu + v(x_0)}{\nu + 1}
			< 0.
	\ee
	In the last inequality, we used that $\{u < 1\}$ has positive measure so that, due to~\eqref{e.kernel}, $v(x_0) < 1$. On the other hand, $x_0$ is the location of a minimum of $v_x + c$, implying that
	\be
		v_{xx}(x_0) = 0.
	\ee
	This is a contradiction.  Hence, case one may not occur.

	{\bf Case two: $u \equiv 0$ on $(-\infty,x_0)$.}  We apply \Cref{lem:single-jump-0} and find that $u \equiv 0$ on $(x_0,\infty)$.  Hence, $u = 0$ almost everywhere, which contradicts our assumption that $\{0 < u\}$ has positive measure.  Hence, case two cannot occur.
	
	{\bf Case three: $u \equiv 0$ on $(x_0, \infty)$.}  By assumption, $u$ is positive on a positive measure subset of $\R$.  Hence, $u$ is positive on a positive measure subset of $(-\infty,x_0)$.  We then apply \Cref{lem:explicit-solution} to conclude that
	\be
		u(x_0^-)
			\in \Big\{0, \frac{\nu + v(x_0)}{\nu + 1}\Big\}.
	\ee
	If $u(x_0^-) = 0$, then \Cref{lem:single-jump-0} implies that $u \equiv 0$ on $(-\infty, x_0)$, which contradicts the assumption that $\{0 < u\}$ has positive measure.  It follows that
	\be
		u(x_0^-) = \frac{\nu + v(x_0)}{\nu + 1},
	\ee
	which concludes the proof.
	\end{proof}

The combination of the above establishes \Cref{prop:cardZleq1} in full.

\sssection{Proofs of \Cref{prop:cardZleq1}'s helper lemmas}

\begin{proof}[Proof of \Cref{lem:explicit-solution}]
    We first note that the vector field on the right-hand side of \eqref{eq:scaled-time} is globally Lipschitz continuous, therefore the solution to \eqref{eq:scaled-time} is well-defined and unique on $\mathbb R$. Since $u$ satisfies \eqref{e.hyp2} on $(x_-,x_+)$, the function $\bar u(t):=u\big(\tau(t)\big)$ solves the equation
    \be
	\bar u'(t)=\tau'(x)u'\big(\tau(t)\big)=\bar u(t)\left(\frac{\nu + \bar v(t)}{\nu}-\left(\frac{\nu + 1}{\nu}\right)\bar u(t)\right)
		\qquad\text{ for all } t \in \R.
    \ee
 This implies that $\bar u(t)=u\big(\tau(t)\big)$ must be given by \eqref{eq:explicitsol-hyperbolic}.  
    
    Now, assume $u(x_m)>0$.  Note that, for all $t\in\mathbb R$, we have
    \be
	\left(1+u(x_m)\left(\frac{\nu + 1}{\nu}\right)\int_0^t e^{\int_0^r \frac{\nu + 1}{\nu}\bar v(s)\dd s}\dd r\right)u\big(\tau(t)\big)
		=u(x_m ) \exp\Big\{\int_{0}^t \frac{\nu + 1}{\nu}\bar v(s)\dd s\Big\}, 
    \ee
    therefore $1+u(x_m)\left(\frac{\nu + 1}{\nu}\right)\int_0^t e^{\int_0^l \frac{\nu + 1}{\nu}\bar v(s)\dd s}\dd l$ does not vanish for $t\in\mathbb R$.  It may, however, tend to zero as $t \to -\infty$.

    Consider first the case where $v_x(x_m)<-c$. Then $\tau$ is strictly increasing with $\lim_{t\to+\infty}\tau(t)=x_+$ and $\lim_{t\to-\infty}\tau(t)=x_-$ . 
    We first compute the limit as $t \to +\infty$.  
    Notice that, as $t\to+\infty$,
	\be
		\exp\left\{\int_0^t \frac{\nu + \bar v(s)}{\nu}\dd s\right\}
			\to+\infty
		\quad\text{ and }\quad
		\int_0^t\exp\left\{\int_0^r \frac{\nu + \bar v(s)}{\nu} \dd s\right\} \dd r \to+\infty.
	\ee
	Hence, by l'H\^{o}pital's rule,
	\be\label{e.c80801}
		\begin{split}
		\lim_{t\to\infty} \bar u(t)
			&= \lim_{t\to\infty} \frac{u(x_m)\exp\Big\{\int_0^r \frac{\nu + \bar v(s)}{\nu} \dd s\Big\}}{1 + u(x_m) \frac{\nu + 1}{\nu} \int_0^t\exp\Big\{\int_0^r \frac{\nu + \bar v(s)}{\nu} \dd s\Big\} \dd r}
			= \lim_{t\to\infty}
				\frac{u(x_m)\frac{\nu + \bar v(t)}{\nu} \exp\Big\{\int_0^r \frac{\nu + \bar v(s)}{\nu} \dd s\Big\}}
					{u(x_m) \frac{\nu+1}{\nu} \exp\Big\{\int_0^r \frac{\nu + \bar v(s)}{\nu} \dd s\Big\}}
			\\&
			= \lim_{t\to\infty} \frac{\nu + \bar v(t)}{\nu + 1}
			= \frac{\nu + v(x_+)}{\nu + 1}.
		\end{split}
	\ee

	Next we deal with the limit $t\to -\infty$. We note that the numerator of the right-hand side of~\eqref{eq:explicitsol-hyperbolic} approaches 0 as $t$ approaches $-\infty$. As for the denominator, it is clearly increasing and bounded on $(-\infty, 0)$, and therefore it has a limit. We distinguish two cases, based on whether that limit is positive or zero.  \smallskip
	
	\noindent\textbf{Case one:} suppose that
	\be
		\lim_{t\to-\infty} \left(1+u(x_m)\left(\frac{\nu + 1}{\nu}\right)\int_0^t \exp\left\{\int_0^r\frac{\nu + 1}{\nu}\bar v(s)\dd s\right\} \dd r\right) >0.
	\ee
	Then, recalling that the numerator tends to zero, we find
    \be
	\lim_{y\searrow x^-}u(y) = \lim_{t\to-\infty}\bar u(t) = 0.
    \ee
    \smallskip

	\noindent\textbf{Case two:} suppose that
	\be
    	\lim_{t\to-\infty} \left(1+u(x)\left(\frac{\nu + 1}{\nu}\right)\int_0^t \exp\left\{\int_0^r\frac{\nu + 1}{\nu}\bar v(s)\dd s\right\} \dd r\right)
		=0.
	\ee
	Then, arguing exactly as in~\eqref{e.c80801} yields
	\be
		\begin{split}
		\lim_{t\to-\infty} \bar u(t)
			= \frac{\nu + v(x_-)}{\nu + 1}.
		\end{split}
	\ee
	This concludes the proof in the case where $v_x(x_m) < -c$.\medskip
    
    If $v_x(x_m)>-c$ then $\tau $ is a strictly decreasing function with $\tau(t)\to x_\pm$ as $t\to\mp \infty$.  By applying a similar method, we get the conclusion claimed in the lemma.  \Cref{lem:explicit-solution} is proved. 
\end{proof}

\begin{proof}[Proof of \Cref{lem:single-jump-0}]
	We include only the proof of (i), as the proof of (ii) is similar.  Without loss of generality, we may take $\bar x_0 \in \cZ \cup \{+\infty\}$; that is, $(x_0, \bar x_0)$ is the maximal connected interval in $\cZ^c$ with $x_0$ as a left endpoint.	Since $0 \leq u, v \leq 1$,
	\be
		v_{xx}(x_0)
		= \frac{v(x_0)-u(x_0)}{\nu}
		\leq \frac{1}{\nu}.
	\ee
	It follows that
    \be
	v_x(x) = v_x(x_0)+\int_{x_0}^xv_{xx}(y)\dd y\leq -c + \frac{1}{\nu}(x-x_0),
	\ee
	from which we deduce the inequality
	\be
		-c-v_x(x)\geq -\frac{1}{\nu}(x-x_0).
    \ee
	Let
	\be\label{e.c80802}
		\dbar x_0 = \sup \{ x > x_0: u(x) \leq \min\{v(x), 1/4\}\}.
	\ee
	We first show that $\dbar x_0 \geq \bar{x}_0$, and then we use this to show that $\bar{x}_0 = +\infty$, which will conclude the proof.
	
	Since $u \in C_{\rm loc}^1(x_0, \bar x_0)$, $u(x_0^+) = 0$, and $v>0$ (recall~\eqref{e.kernel}), $\dbar x_0$ is well-defined and $\dbar x_0 > x_0$.  Then, for any $x \in (x_0, \min\{\bar x_0, \dbar x_0\})$,
    \begin{align}
        u_x(x) &= \frac{u(x)}{-c-v_x(x)}\left(\frac{\nu + v(x)}{\nu}-\left(\frac{\nu + 1}{\nu}\right)u(x)\right)
        		\\&
        		\leq \frac{u(x)}{-\frac{1}{\nu}(x-x_0)} \left(\frac{\nu + v(x)}{\nu}-\left(\frac{\nu + 1}{\nu}\right)u(x)\right)
		\leq - \frac{3\nu}{4} \frac{u(x)}{x-x_0}.
    \end{align}
	It immediately follows that $u$ is decreasing on $(x_0, \min\{\bar x_0, \dbar x_0\})$; however, since $u(x_0^+) = 0$, it follows that $u \equiv 0$ on this set.  Following the definition of $\dbar x_0$~\eqref{e.c80802}, it follows that $\dbar x_0 \geq \bar{x}_0$.

    We now show that $\bar x_0 = +\infty$.  We argue by contradiction, assuming that $\bar x_0 < +\infty$.  In this case, $\bar x_0 \in \cZ$ (see the first paragraph of the proof).  Hence, by the previous paragraph and~\eqref{e.hyp}, we find
    \be\label{e.c80803}
		\begin{cases}
			\nu v_{xx} = v \quad&\text{ in } (x_0, \bar x_0),\\
			v_x(x_0) = v_x(\bar x_0) = -c.
		\end{cases}
    \ee
    This is clearly not possible as the first line implies that $v_x$ is strictly increasing over $[x_0, \bar x_0]$ (recall that $v$ is strictly positive due to~\eqref{e.kernel}), while the second line implies that $v_x$ is the same at two points.  It follows that $\bar x_0 = +\infty$; hence, $u \equiv 0$ on $(x_0,\infty)$.    

    We omit the proof of the form of $v$ on $(x_0,\infty)$ as this follows directly from the fact that $\nu v_{xx} = v$ on $(x_0,\infty)$ along with the boundedness and positivity of $v$.
    
    Finally, it is clear from the form of $v$ and its positivity that $c>0$. This concludes the proof.
\end{proof}

\bibliographystyle{plain}
\bibliography{biblio.bib}

\end{document}